\documentclass[a4paper,reqno,10pt]{amsart}

\usepackage[top=2.5cm,bottom=2.5cm,left=2.5cm,right=2.5cm]{geometry}

\usepackage{amssymb,amsmath,amsthm,array,bm,mathtools,bbm,enumitem,booktabs}
\usepackage{graphicx}
\usepackage{xcolor}
\usepackage[pagewise]{lineno}
\usepackage[colorlinks=true,bookmarks=true,bookmarksnumbered=true,bookmarkstype=toc,linktocpage=true,hypertexnames=false]{hyperref}
\usepackage{mathtools}
\mathtoolsset{showonlyrefs=true}

\newtheorem{thm}{Theorem}[section]

\newtheorem{lem}[thm]{Lemma}
\newtheorem{prop}[thm]{Proposition}
\newtheorem{cor}[thm]{Corollary}
\newtheorem{rem}[thm]{Remark}

\newtheorem{ass}[thm]{Assumption}

\numberwithin{equation}{section}
\allowdisplaybreaks

\newcommand{\mcl}{\mathcal{L}}

\newcommand{\mbbs}{\mathbb{S}}
\newcommand{\mbbn}{\mathbb{N}}
\newcommand{\E}{\mathbb{E}}
\newcommand{\Prob}{\mathbb{P}}

\newcommand{\p}{\partial}

\newcommand{\cil}{\xrightarrow{\mcl}}

\newcommand{\cip}{\xrightarrow{p}}

\newcommand{\argmax}{\mathop{\rm argmax}}

\newcommand{\abs}[1]{\left|#1\right|}
\newcommand{\R}{\mathbb{R}}
\newcommand{\ind}{\mathbbm{1}}
\newcommand{\Y}{\mathbb{Y}}

\newcommand{\al}{\alpha}
\newcommand{\gam}{\gamma}

\newcommand{\Lam}{\Lambda}

\newcommand{\Del}{\Delta}

\title[Hybrid switching L\'{e}vy-driven SDEs]{
	Ergodicity and High-Frequency Inference for Hybrid Switching L\'{e}vy-Driven Stochastic Differential Equations
}

\author{Yuzhong Cheng}
\address{Kyushu University, Institute of Mathematics for Industry, 744 Motooka Fukuoka, Japan}
\email{cheng.yuzhong.451@m.kyushu-u.ac.jp}

\date{}



\begin{document}
	\setlength{\baselineskip}{4.5mm}
	
	\begin{abstract}
		Hybrid switching L\'evy-driven stochastic differential equations with pure-jump noise and state-dependent switching rates are studied under high-frequency observation. A three-stage inference procedure is proposed for the drift, scale, and switching-rate parameters, combining a staged Gaussian quasi-likelihood with an intensity-type contrast. Checkable sufficient conditions for weighted exponential ergodicity are established for the hybrid process; the proof does not rely on Brownian smoothing, but uses a fixed skeleton-chain argument combining small-jump accessibility and regime connectivity.  Under ergodicity and the high-frequency sampling scheme, consistency, joint asymptotic normality, and a polynomial-type large deviation inequality are proved for the full estimator. The joint limit exhibits a transparent covariance structure: the drift and scale blocks are coupled through the third moment of the driving L\'evy noise, whereas the switching-rate block is asymptotically uncorrelated with the continuous-coefficient blocks. Numerical experiments for models driven by normal inverse Gaussian noise illustrate the finite-sample behavior of the proposed estimators.
		
	\end{abstract}
	
	\maketitle

	\section{Introduction}
	
	Switching stochastic differential equations are a standard framework for
	random dynamical systems whose local behavior changes across finitely many
	regimes. In the classical setting the regime process is an autonomous
	continuous-time Markov chain, while more general hybrid models allow the
	switching rates to depend on the current continuous state; see
	\cite{mao2006stochastic,yin2009hybrid,XiYin2011}.
	
	This paper studies the hybrid switching L\'evy-driven stochastic
	differential equation 
	\begin{equation}
		\label{eq:intro-model}
		dX_t
		=
		b(X_t,\Lambda_t,\alpha)\,dt
		+
		c(X_{t-},\Lambda_{t-},\gamma)\,dL_t,
	\end{equation}
	where \(L\) is a one-dimensional pure-jump L\'evy process and
	\(\Lambda\in\mathbb S=\{1,\ldots,m\}\) has state-dependent switching rates
	\(q_{ij}(X_t,\vartheta)\). The unknown parameter is
	\(\zeta=(\alpha,\gamma,\vartheta)\). We observe the full hybrid path
	\(\mathcal D_n=\{(X_{t_j},\Lambda_{t_j})\}_{j=0}^n\) at times
	\(t_j=jh_n\), where \(h_n\to0\), \(T_n:=nh_n\to\infty\), and
	\(n h_n^2\to0\). Our objective is to derive checkable ergodicity conditions
	for the model and then estimate the full parameter from this high-frequency
	ergodic sample.
	
	The long-time condition \(T_n\to\infty\) is essential for the drift and
	switching-rate parameters. As in the standard high-frequency theory for
	ergodic SDEs, a fixed time horizon does not provide enough information for
	consistent drift estimation, and in the present hybrid model it also gives
	only finite switching information. The invariant law averages generated by
	an ergodic trajectory are therefore the deterministic limits of the
	quasi-likelihood contrasts. Exponential ergodicity is used not only to
	identify these limits, but also to control the behavior of the process at
	infinity and to provide the mixing and moment bounds needed for the central
	limit and polynomial large deviation arguments.

	Ergodicity for switching diffusions and switching jump-diffusions has been
	studied extensively; see, for example,
	\cite{Xi2009,XiYin2011,XiYin2017}. In many such results the diffusion part
	plays an essential role through smoothing, irreducibility, or strong
	Feller-type arguments. The present model has no Brownian component, so these
	approaches do not apply directly. For L\'evy-driven SDEs without switching,
	exponential ergodicity and mixing estimates are available in
	\cite{Masuda2007,Kulik2009}, but those results do not cover the additional
	state-dependent switching structure considered here.
	
	From the statistical viewpoint, a substantial part of the literature on
	switching SDEs is computational, especially when the regime is hidden or the
	model is fitted by simulation-based methods; see, for example,
	\cite{Hibbah2020,MariMari2023}. On the theoretical side, high-frequency quasi-likelihood
	methods are well developed for ergodic diffusions and L\'evy-driven SDEs,
	including Gaussian or quasi-likelihood approaches in
	\cite{Kessler1997,Gobet2002,UchidaYoshida2012,Mas13as} and the two-step
	procedure of \cite{MasudaUehara2017}. By contrast, there appears to be
	almost no high-frequency asymptotic theory for switching L\'evy-driven SDEs,
	and even the Markovian switching diffusion case has only recently been
	studied in \cite{Yuzhong2025}.
	
	In this paper we first prove checkable sufficient conditions for weighted
	exponential ergodicity of the pure-jump state-dependent hybrid model, and
	then construct a three-stage estimator for the full parameter \(\zeta\). The
	ergodicity proof is based on a fixed skeleton-chain argument, a small-jump
	minorization of the L\'evy measure, regime connectivity on compact sets, and
	a Foster--Lyapunov drift condition, in the spirit of
	\cite{meyn1993stability,Masuda2007}.
	The estimation procedure combines a staged Gaussian quasi-likelihood for
	\((\alpha,\gamma)\), adapted from the high-frequency L\'evy-driven SDE
	methodology of \cite{Mas13as,MasudaUehara2017}, with an intensity-type
	contrast for \(\vartheta\), adapted from the counting-process likelihood
	framework in \cite{Andersen1993} and its discrete-time
	Markov jump analogue in \cite{bladt2005statistical}. Both parts are
	extracted from the same observed hybrid path.
	
	The ergodicity conditions consist of dissipativity of the drift, boundedness
	and local positivity of the scale coefficient, a small-jump lower bound on
	the L\'evy measure, and upper and lower control of the switching rates. In
	particular, the switching mechanism is required to satisfy a uniform upper
	bound on the total switching rate and positive lower bounds on
	\(q_{ij}(\cdot,\vartheta_0)\) on compact \(x\)-sets. Relative to the
	non-switching L\'evy-driven SDE setting of \cite{Masuda2007,Kulik2009}, this
	is the additional ingredient needed to control regime movement and guarantee
	accessibility of the discrete states. Such switching-rate conditions also
	parallel the role played by regime-connectivity assumptions in the
	ergodicity theory of hybrid diffusions and jump-diffusions
	\cite{XiYin2011,XiYin2017}.
	
	Our main results are consistency, joint \(T_n^{1/2}\)-asymptotic
	normality, and a polynomial-type large deviation inequality for the full
	estimator. A key technical point is that the hybrid structure generates
	additional within-step remainder terms: the increment \(\Delta_jX\) is
	affected by possible regime changes inside \([t_{j-1},t_j]\), and the
	endpoint transition indicators only approximate the continuous-time
	switching counts. We show that these terms are negligible under
	\(n h_n^2\to0\). Compared with the existing quasi-likelihood theory for
	ordinary L\'evy-driven SDEs \cite{Mas13as,MasudaUehara2017}, the continuous
	part \((\hat\alpha_n,\hat\gamma_n)\) keeps the same \(T_n^{1/2}\)-rate and
	the same asymptotic covariance structure as in the non-switching case.
	For the switching part, we also obtain the \(T_n^{1/2}\)-rate for
	\(\hat\vartheta_n\). Compared with the recent Markovian switching diffusion
	result of \cite{Yuzhong2025}, we allow state-dependent switching and
	pure-jump L\'evy noise. An interesting feature of the joint limit is that
	the asymptotic covariance blocks between the continuous part
	\((\hat\alpha_n,\hat\gamma_n)\) and the switching part \(\hat\vartheta_n\)
	are zero. Thus, although \(X\) and \(\Lambda\) interact dynamically and the
	two parameter blocks are estimated from the same observed hybrid path, the
	corresponding estimators are asymptotically uncorrelated.
	
	The rest of the paper is organized as follows.
	Section~\ref{sec:model} introduces the model and assumptions.
	Section~\ref{sec:exp-erg} gives sufficient conditions for exponential
	ergodicity. Section~\ref{sec:estimation} develops the three-stage estimator
	and its asymptotic theory. Section~\ref{sec:numerical} reports numerical
	experiments. Proofs are collected in Section~\ref{sec:thmproof}, and
	auxiliary lemmas are given in Section~\ref{sec:tech}.

	\section{Model and assumptions}
	\label{sec:model}
	
	\subsection{Model}
	Let
	$\zeta=(\alpha,\gamma,\vartheta)
	\in
	\Theta_\alpha\times\Theta_\gamma\times\Theta_\vartheta
	\subset
	\mathbb R^{p_\alpha}\times\mathbb R^{p_\gamma}\times\mathbb R^{p_\vartheta}$.
	Let $\mathbb S=\{1,\ldots,m\}.$
	On a filtered measurable space
	\((\Omega,\mathcal F,\{\mathcal F_t\}_{t\ge0})\), let
	\(\{\mathbb P_\zeta:\zeta\in\Theta\}\) be a family of probability measures
	under which the filtration is usual. 
	
	For \(i\neq j\), let \(q_{ij}(\cdot,\vartheta)\ge0\), $q_{ii}(x,\vartheta) := - \sum_{j\neq i}q_{ij}(x,\vartheta)$, set $Q(x,\vartheta):=(q_{ij}(x,\vartheta))_{i,j\in\mathbb S}$.
	Choose consecutive (with respect to the lexicographic ordering on $\mathbb S\times \mathbb S$) 
	left-closed, right-open intervals $\Gamma_{ij}(x,\vartheta)
	\subset\mathbb R_+$ with $|\Gamma_{ij}(x,\vartheta)|=q_{ij}(x,\vartheta)$ for $i\neq j$ (see, for example, \cite{yin2009hybrid,ZhuYinBaran2017}),
	and define
	\[
	h(x,i,z;\vartheta)
	:=
	\sum_{j\neq i}(j-i)\mathbf 1_{\Gamma_{ij}(x,\vartheta)}(z).
	\]
	Under \(\mathbb P_\zeta\), the process
	\((X,\Lambda)\in\mathbb R\times\mathbb S\) solves
	the following stochastic differential equation:
	\begin{equation}
		\label{eq:model}
		\begin{cases}
			dX_t
			=
			b(X_{t},\Lambda_{t},\alpha)\,dt
			+
			c(X_{t-},\Lambda_{t-},\gamma)\,dL_t,
			\\[0.4em]
			d\Lambda_t
			=
			\displaystyle
			\int_{\mathbb R_+}
			h(X_{t-},\Lambda_{t-},z;\vartheta)\,N(dt,dz),
		\end{cases}
	\end{equation}
	where 
	$L$
	is a one-dimensional pure-jump L\'evy process with L\'evy measure \(\nu\),
	and \(N(dt,dz)\) is a Poisson random measure on
	\(\mathbb R_+\times\mathbb R_+\) with intensity \(dt\,dz\).
	From \eqref{eq:model}, the switching process \(\Lambda\) satisfies, for \(j\neq i\),
	\[
	\Prob_\zeta
	\left(
	\Lambda_{t+\delta}=j
	\mid
	\mathcal F_t,\ X_t=x,\ \Lambda_t=i
	\right)
	=
	q_{ij}(x,\vartheta)\delta+o(\delta),
	\qquad \delta\downarrow0.
	\]
	The matrix \(Q(x,\vartheta)\) is called the state-dependent rate matrix of \(\Lambda\).
	
	The initial condition $X_0$, $L$ and $N$ are independent.
	
	Throughout, $\Theta$ is assumed to be compact and convex, with sufficiently regular boundary. We denote the true parameter value by $\zeta_0=(\alpha_0,\gamma_0,\vartheta_0)\in\operatorname{int}(\Theta)$.

	In the special case where $Q(x) \equiv Q$, the process $\Lambda$ reduces to a continuous-time Markov chain. This configuration corresponds to the standard Markovian switching framework (see, e.g., \cite{mao2006stochastic}).
	For general background on L\'evy-driven stochastic equations we refer to
	\cite{applebaum2009levy,Sat99}. For hybrid and
	regime-switching stochastic systems, see
	\cite{mao2006stochastic,yin2009hybrid,XiYin2011}.

	\subsection{Notations}
	Throughout, for \(a,u,v\in\mathbb R^d\) and
	\(A\in\mathbb R^{d\times d}\), write
	$a[v]:=a^\top v$,
	$A[u,v]:=u^\top A v$,
	$a^{\otimes2}:=aa^\top$ .
	
	For \(\zeta\in\Theta\), let \(\mathbb E_\zeta\) denote expectation under
	\(\mathbb P_\zeta\). Set
	$\mathbb P:=\mathbb P_{\zeta_0}$,
	$\mathbb E:=\mathbb E_{\zeta_0}$,
	$\mathbb E_{j-1}[\cdot]
	:=
	\mathbb E[\cdot\mid\mathcal F_{t_{j-1}}]$.
	For \(z=(x,i)\in\mathbb R\times\mathbb S\), let
	\(\mathbb P_{\zeta,z}\) be the law of \((X,\Lambda)\) under parameter
	\(\zeta\) with initial state \(z\), and let
	\(\mathbb E_{\zeta,z}\) be the corresponding expectation. We abbreviate $\mathbb P_z:=\mathbb P_{\zeta_0,z}$, $\mathbb E_z:=\mathbb E_{\zeta_0,z}$.
	
	For a set \(A\), let \(A^c\) denote its complement. For sets \(A,B\), write
	\(A\Subset B\) if \(\overline A\) is compact and
	\(\overline A\subset\operatorname{int}(B)\).
	
	For increments of the processes, $\Delta_j X := X_{t_j} - X_{t_{j-1}}$ and $\Delta_j L := L_{t_j} - L_{t_{j-1}}$. For the discretely sampled coefficients, we write $b_{j-1}(\alpha) := b(X_{t_{j-1}}, \Lambda_{t_{j-1}}, \alpha)$ and $c_{j-1}(\gamma) := c(X_{t_{j-1}}, \Lambda_{t_{j-1}}, \gamma)$.
	
	\subsection{Assumptions}
	
	\begin{ass}
		\label{ass:Levy}
		$\mathbb{E}[L_1]=0$, $\mathbb{E}[L_1^2]=1$, and $\mathbb{E}|L_1|^q<\infty \quad (\forall q>0).$
	\end{ass}
	
	\begin{ass}
		\label{ass:coeff}
		\begin{enumerate}[label=(C\arabic*),leftmargin=3.0em]
			\item There exists a constant $C>0$ such that for all $x,y\in\mathbb{R}$, $i\in\mathbb{S}$, \(\alpha\in\Theta_\alpha\), and \(\gamma\in\Theta_\gamma\),
			$$ \begin{aligned}
				|b(x,i,\alpha)-b(y,i,\alpha)|+|c(x,i,\gamma)-c(y,i,\gamma)| &\le C|x-y|,\\
				|b(x,i,\alpha)|^2+|c(x,i,\gamma)|^2 &\le C(1+|x|^{2}).
			\end{aligned} $$
			
			\item For each $i\in\mathbb{S}$, \(b(\cdot,i,\cdot)\in C^{2,3}(\mathbb R\times\Theta_\alpha)\) and
			\(c(\cdot,i,\cdot)\in C^{2,3}(\mathbb R\times\Theta_\gamma)\). Furthermore, there exists a constant $C>0$ such that
			$$ \max_{i\in\mathbb{S}}\sup_{(x,\alpha,\gamma)\in\mathbb{R}\times\Theta_\alpha\times \Theta_\gamma}
			\frac{1}{1+|x|^C}
			\left(
			\max_{\substack{0\le k\le3\\0\le \ell\le2}}
			\Big\{
			|\partial_\alpha^{k}\partial_x^{\ell}b(x,i,\alpha)|
			+|\partial_\gamma^{k}\partial_x^{\ell}c(x,i,\gamma)|
			\Big\}
			+c(x,i,\gamma)^{-1}
			\right)<\infty, $$
			where $\partial_\theta^{k}$ and $\partial_x^{\ell}$ denote the $k$-th and $\ell$-th order partial derivatives with respect to $\theta$ and $x$, respectively.
			
			\item $\inf_{(x,i,\gamma)\in\mathbb{R}\times\mathbb{S}\times\Theta_\gamma}c(x,i,\gamma) \ge c_0 >0$. 
			
			\item 
			For each \(i\neq j\),
			\(q_{ij}\in C^{1,3}(\mathbb R\times\Theta_\vartheta)\) and
			\[
			0<
			\min_{i\neq j}\inf_{(x,\vartheta)}
			q_{ij}(x,\vartheta),
			\qquad
			\max_{i}\sup_{(x,\vartheta)}
			\sum_{j\neq i}q_{ij}(x,\vartheta)
			<\infty,
			\]
			and, for some \(C>0\),
			\[
			\max_{i\neq j}
			\sup_{(x,\vartheta)\in\mathbb R\times\Theta_\vartheta}
			\frac{1}{1+|x|^C}
			\max_{\substack{0\le \ell\le1\\0\le r\le3}}
			\left\{
			\left|
			\partial_x^\ell\partial_\vartheta^r q_{ij}(x,\vartheta)
			\right|
			+
			\left|
			\partial_x^\ell\partial_\vartheta^r \log q_{ij}(x,\vartheta)
			\right|
			\right\}
			<\infty .
			\]
		\end{enumerate}
	\end{ass}
	
	Under Assumption \ref{ass:coeff} the system \eqref{eq:model} admits a unique, non-explosive strong solution. Furthermore, the joint process $(X, \Lambda)$ is a strong Markov process; see \cite{XiYin2011}.

	

	\begin{ass}
		\label{ass:ergodic}
		\begin{enumerate}[label=(E\arabic*),leftmargin=3.0em]
			\item
			For every $q>0$, there exist constants $C_q>0$, $a_q>0$ and a measurable function $V_q:\R\times\mbbs\to[1,\infty)$ such that for all $t\ge0$ and $z=(x,i)\in\mathbb{R}\times\mathbb{S}$,
			$$ \sup_{|f|\le V_q}\left| \mathbb{E}_z[f(X_t,\Lambda_t)]-\int_{\mathbb{R}\times\mathbb{S}} f\,d\pi_0 \right| \le C_q e^{-a_q t} V_q(z), $$
			with $V_q(x,i)\asymp 1+|x|^q$.
			
			\item For every $q>0$, $\pi_0$ has finite polynomial moments:
			$ \int_{\mathbb{R}\times\mathbb{S}}|x|^q\,\pi_0(dx,di)<\infty. $
		\end{enumerate}
	\end{ass}
	Under Assumption \ref{ass:ergodic}, the joint Markov process $(X, \Lambda)$ admits a unique invariant probability measure $\pi_0$ on $\mathbb{R}\times\mathbb{S}$. We assume that $(X_0, \Lambda_0) \sim \pi_0$. 
	In Section~\ref{sec:exp-erg}, we provide a direct verification of
	Assumption~\ref{ass:ergodic} for the present pure-jump hybrid model. 

	
	We next define the limiting contrasts. Let 
	\(q_i(x,\vartheta):=\sum_{j\neq i}q_{ij}(x,\vartheta)\). Define
	\begin{align*}
		G_\gamma(x,i,\gamma)
		&:=
		\log\frac{c(x,i,\gamma)^2}{c(x,i,\gamma_0)^2}
		+
		\frac{c(x,i,\gamma_0)^2}{c(x,i,\gamma)^2}
		-1,\\
		G_\alpha(x,i,\alpha)
		&:=
		\frac{\{b(x,i,\alpha)-b(x,i,\alpha_0)\}^2}
		{c(x,i,\gamma_0)^2},\\
		F_Q(x,i,\vartheta)
		&:=
		\sum_{j\neq i}
		q_{ij}(x,\vartheta_0)
		\log\frac{q_{ij}(x,\vartheta)}{q_{ij}(x,\vartheta_0)}
		-
		q_i(x,\vartheta)+q_i(x,\vartheta_0).
	\end{align*}
	Set
	\[
	\mathbb Y_\gamma(\gamma)
	:=
	-\frac12\int_{\mathbb{R}\times\mathbb{S}} G_\gamma(x,i,\gamma)\,\pi_0(dx,di),
	\qquad
	\mathbb Y_\alpha(\alpha)
	:=
	-\frac12\int_{\mathbb{R}\times\mathbb{S}} G_\alpha(x,i,\alpha)\,\pi_0(dx,di),
	\]
	\[
	\mathbb Y_Q(\vartheta)
	:=
	\int_{\mathbb{R}\times\mathbb{S}} F_Q(x,i,\vartheta)\,\pi_0(dx,di),
	\qquad
	\mathbb Y(\zeta)
	:=
	\mathbb Y_\gamma(\gamma)
	+
	\mathbb Y_\alpha(\alpha)
	+
	\mathbb Y_Q(\vartheta).
	\]
	
	For the information matrices, write
	\[
	\Psi_\gamma(x,i)
	:=
	\partial_\gamma\log c(x,i,\gamma_0)^2,
	\qquad
	A_\alpha(x,i)
	:=
	\frac{\partial_\alpha b(x,i,\alpha_0)}
	{c(x,i,\gamma_0)},
	\]
	\[
	G_\alpha^{(2)}(x,i,\alpha,\gamma)
	:=
	\frac{\partial_\alpha b(x,i,\alpha)^{\otimes2}}
	{c(x,i,\gamma)^2},
	\qquad
	\Psi_Q^{ij}(x)
	:=
	\partial_\vartheta\log q_{ij}(x,\vartheta_0).
	\]
	Then
	\[
	\Gamma_\gamma
	:=
	\frac12\int_{\mathbb{R}\times\mathbb{S}}
	\Psi_\gamma(x,i)^{\otimes2}\,\pi_0(dx,di),
	\qquad
	\Gamma_\alpha
	:=
	\int_{\mathbb{R}\times\mathbb{S}}
	G_\alpha^{(2)}(x,i,\alpha_0,\gamma_0)\,\pi_0(dx,di),
	\]
	and
	\[
	\Gamma_Q
	:=
	\int_{\mathbb{R}\times\mathbb{S}}
	\sum_{j\neq i}
	q_{ij}(x,\vartheta_0)\,
	\Psi_Q^{ij}(x)^{\otimes2}
	\,\pi_0(dx,di).
	\]

	The next two assumptions concern global separation and local nondegeneracy. 
	\begin{ass}
		\label{ass:Y-global}
		There exist constants
		\(\chi_\gamma,\chi_\alpha,\chi_Q>0\) such that
		\[
		-\mathbb Y_\gamma(\gamma)
		\ge
		\chi_\gamma|\gamma-\gamma_0|^2,
		\quad
		-\mathbb Y_\alpha(\alpha)
		\ge
		\chi_\alpha|\alpha-\alpha_0|^2,
		\quad
		-\mathbb Y_Q(\vartheta)
		\ge
		\chi_Q|\vartheta-\vartheta_0|^2.
		\]
		Consequently, with
		\(\chi:=\chi_\gamma\wedge\chi_\alpha\wedge\chi_Q\),
		$-\mathbb Y(\zeta)
		\ge
		\chi|\zeta-\zeta_0|^2$.
	\end{ass}
	
	\begin{ass}
		\label{ass:positive-defi}
		The matrices
		$\Gamma_\gamma$,
		$\Gamma_\alpha$,
		$\Gamma_Q$
		are positive definite.
	\end{ass}
	
	We collect the moment inequalities used repeatedly in the subsequent proofs.
	
	\begin{lem}\label{lem:short-time-moment}
		Under Assumptions~\ref{ass:Levy}, \ref{ass:coeff}, and~\ref{ass:ergodic}, for every $q>0$, all $t\ge 0$, and every $h\in(0,1]$:
		\begin{enumerate}
			\item[\textup{(i)}] for every $s\in[t,t+h]$,
			$\Prob(\Lam_s\ne\Lam_t\mid\mathcal F_t)\le C(s-t)\bigl(1+|X_t|^{C}\bigr)$;
			\item[\textup{(ii)}]
			$\sup_{0\le u\le h}\E\bigl[|X_{t+u}-X_t|^q\,\big|\,\mathcal F_t\bigr]
			\le C_q\,h^{(q/2)\wedge 1}\bigl(1+|X_t|^{C_q}\bigr)$;
			\item[\textup{(iii)}]
			$\sup_{0\le u\le h}\E\bigl[|X_{t+u}|^q\,\big|\,\mathcal F_t\bigr]
			\le C_q\bigl(1+|X_t|^{C_q}\bigr)$.
		\end{enumerate}
	\end{lem}
	
	\begin{proof}
		For (i), let $N_t^\Lam:=\sum_{0<s\le t}\mathbf 1_{\{\Lam_s\ne\Lam_{s-}\}}$ denote the jump counting process of $\Lam$. Markov's inequality gives
		\[
		\Prob(\Lam_s\ne\Lam_t\mid\mathcal F_t)
		\le \Prob(N_s^\Lam-N_t^\Lam\ge 1\mid\mathcal F_t)
		\le \E\bigl[N_s^\Lam-N_t^\Lam\mid\mathcal F_t\bigr].
		\]
		Since $N^\Lam$ admits the compensator $u\mapsto\int_0^u\sum_{k\ne\Lam_{r-}}q_{\Lam_{r-}k}(X_{r-})\,dr$,
		\[
		\E\bigl[N_s^\Lam-N_t^\Lam\mid\mathcal F_t\bigr]
		=\E\!\left[\int_t^s\sum_{k\ne\Lam_{r-}}q_{\Lam_{r-}k}(X_{r-})\,dr\,\bigg|\,\mathcal F_t\right],
		\]
		and combining this identity with the polynomial growth bound on $q_{ik}$ from Assumption~\ref{ass:coeff} and part~(iii) below yields (i). Parts (ii) and~(iii) are the standard moment estimates for the hybrid-switching Lévy SDE; see, e.g., \cite[Chapter~2]{yin2009hybrid} and \cite[Chapter~4]{mao2006stochastic}.
	\end{proof}

	
	\section{Conditions for exponential ergodicity}
	\label{sec:exp-erg}
	
	This section gives conditions under which Assumption~\ref{ass:ergodic} holds.
	The stability argument relies on standard Markov-process theory. Once a
	fixed skeleton chain is \(\varphi\)-irreducible and has compact petite sets,
	a Foster--Lyapunov drift condition yields \(V\)-uniform exponential
	ergodicity by the Meyn--Tweedie framework
	\cite{meyn1992criteria,meyn1993stability,down1995exponential,
		meyn2012markov}. The Lyapunov estimates for the pure-jump component are also
	in the spirit of \cite{Masuda2007,Kulik2009}. The model-specific point is
	the verification of irreducibility for the state-dependent hybrid process.
	Since there is no Brownian part, local movement in the continuous coordinate
	must come from a small-jump minorization of the L\'evy measure. Since the
	switching rates depend on \(X\), accessibility of the finite regimes must be
	checked along paths on compact \(x\)-sets. 
	Thus the new ingredient in this section is a direct verification of the petite set and
	irreducibility conditions for the pure-jump state-dependent hybrid model.
	
	We write $\{P_t\}_{t\ge0}$ for the transition semigroup of the Markov
	process $(X,\Lambda)$
	$P_t f(x,i)
	=
	\E_{(x,i)}\!\left[f(X_t,\Lambda_t)\right]$
	for every bounded Borel measurable function
	$f:\R\times\mbbs\to\R$.
	The associated transition kernel is
	$P_t\bigl((x,i),A\bigr)
	=
	\Prob_{(x,i)}\bigl((X_t,\Lambda_t)\in A\bigr)$ for
	$A\in\mathcal B(\R\times\mbbs)$.
	
	For
	\(\zeta=(\alpha,\gamma,\vartheta)\in\Theta\), define
	\(\mathcal D(\mathcal A_\zeta)\) as the set of functions
	\(f:\mathbb R\times\mathbb S\to\mathbb R\) such that \(f(\cdot,i)\in C^2(\mathbb R)\) for each
	\(i\in\mathbb S\), and
	\[
	\int_{\mathbb R}
	\left|
	f(x+c(x,i,\gamma)z,i)-f(x,i)
	-\partial_x f(x,i)c(x,i,\gamma)z\mathbf 1_{\{|z|\le1\}}
	\right|
	\nu(dz)
	<\infty
	\]
	for all \((x,i)\in \mathbb R\times\mathbb S\). For \(f\in\mathcal D(\mathcal A_\zeta)\),
	\begin{align}
		\label{eq:generator}
		\mathcal A_\zeta f(x,i)
		&=
		b(x,i,\alpha)\partial_x f(x,i)
		+
		\int_{\mathbb R}\left(f(x+c(x,i,\gamma)z,i)-f(x,i)
		-\partial_x f(x,i)c(x,i,\gamma)z\mathbf 1_{\{|z|\le1\}}\right)\,\nu(dz)
		\nonumber\\
		&\quad+
		\sum_{j\neq i}
		q_{ij}(x,\vartheta)\{f(x,j)-f(x,i)\},
		\qquad (x,i)\in \mathbb R\times\mathbb S .
	\end{align}
	We write \(\mathcal A_\zeta^{\mathrm{ext}}\) for the corresponding extended
	generator as in \cite{meyn1993stability}.
	
	A measurable function
	\(V:\mathbb R\times\mathbb S\to[1,\infty)\) is called norm-like if
	$\min_{i\in\mathbb S} V(x,i)\to\infty$ as $|x|\to\infty$.

	\begin{ass}
		\label{ass:erg-coeff}
		The following conditions hold.
		
		\begin{enumerate}[label=\textnormal{(E\arabic*)},leftmargin=2.8em]
			
			\item For each \(i\in\mathbb S\), the maps
			$x\mapsto b(x,i,\alpha_0)$,
			and
			$x\mapsto c(x,i,\gamma_0)$
			are twice continuously differentiable. Moreover, for every compact interval
			\(K\subset\mathbb R\),
			\[
			\max_{i\in\mathbb S}
			\sup_{x\in K}
			\sum_{\ell=0}^2
			\left\{
			|\partial_x^\ell b(x,i,\alpha_0)|
			+
			|\partial_x^\ell c(x,i,\gamma_0)|
			\right\}
			<\infty .
			\]
			There exists \(C>0\) such that 
			$|b(x,i,\alpha_0)|
			\le
			C(1+|x|)$
			for $(x,i)\in\mathbb R\times\mathbb S$ .

			\item
			For every compact interval $K\subset\mathbb R$,
			\[
			\inf_{x\in K,\ i\in\mathbb S} c(x,i,\gamma_0)>0,
			\qquad
			\inf_{x\in K}\min_{i\neq j}q_{ij}(x,\vartheta_0)>0 .
			\]
			
			\item For each \(i\neq j\), the map
			$x\mapsto q_{ij}(x,\vartheta_0)$
			is locally Lipschitz. Moreover,
			\[
			\sup_{x\in\mathbb R,\ i\in\mathbb S}
			\sum_{j\neq i}q_{ij}(x,\vartheta_0)
			<\infty.
			\]

			\item $\sup_{(x,i)\in\mathbb R\times\mathbb S}
			|c(x,i,\gamma_0)|
			<\infty$
			
			\item There exist constants \(\lambda_0>0\) and \(K_0>0\) such that
			\[
			x\,b(x,i,\alpha_0)
			\le
			-\lambda_0x^2+K_0,
			\qquad (x,i)\in\mathbb R\times\mathbb S .
			\]
			
			\item There exist constants \(r_0,\kappa_0>0\) such that
			\[
			\nu(B)
			\ge
			\kappa_0\,\lambda\bigl(B\cap(-r_0,r_0)\bigr),
			\qquad B\in\mathcal B(\mathbb R).
			\]
			
		\end{enumerate}
	\end{ass}

	%
	
			
			
			
	
	Conditions (E1), (E2) and (E3) are contained in Assumptions~\ref{ass:coeff}. 
	Other conditions in Assumption~\ref{ass:erg-coeff} play two roles. The
	dissipativity of \(b(\cdot,\cdot,\alpha_0)\), together with the boundedness
	of \(c(\cdot,\cdot,\gamma_0)\), yields a polynomial Foster--Lyapunov drift.
	The small-jump lower bound for \(\nu\) and the compact lower bound for
	\(q_{ij}(\cdot,\vartheta_0)\) yield accessibility of the continuous and
	discrete coordinates, respectively.

	Denote the \(h\)-skeleton chain by
	$\Phi^{(h)}=(\Phi_{nh})_{n\ge0}=(X_{nh},\Lambda_{nh})$.
	We first record the irreducibility statement. Its proof uses the
	compound-Poisson component extracted from the lower bound on \(\nu\), the
	local positivity of \(c(\cdot,\cdot,\gamma_0)\), and the compact
	connectivity of the rates \(q_{ij}(\cdot,\vartheta_0)\).	
	
	\begin{prop}
		\label{prop:fixed-h-skeleton}
		Suppose Assumptions~\ref{ass:Levy} and \ref{ass:erg-coeff} hold.
		There exists a time $h_0>0$ such that the $h_0$-skeleton chain
		$\Phi^{(h_0)}$
		is a $\varphi$-irreducible T-chain, where
		\[
		\varphi(B)
		:=
		\sum_{j=1}^m\sum_{n=1}^\infty
		2^{-j-n}
		\frac{\lambda(B_j\cap I_n)}{1+2n},
		\qquad
		B_j:=\{x\in\mathbb R:(x,j)\in B\}.
		\]
		for \(B\in\mathcal{B}(\mathbb R)\otimes 2^{\mathbb S}\), where
		\(I_n := [-n,n]\).
	\end{prop}
	
	Proposition~\ref{prop:fixed-h-skeleton} provides only the
	irreducibility input. The recurrence input is the Lyapunov drift obtained
	from Assumption~\ref{ass:erg-coeff}: for
	\(V_q(x,i)=(1+x^2)^{q/2}\),
	\[
	\mathcal A_0^{\mathrm{ext}}V_q(x,i)
	\le
	-a_qV_q(x,i)+b_q .
	\]
	Dynkin's formula transfers this estimate to the fixed skeleton chain. Since
	\(V_q\) is norm-like, its sublevel sets are compact and hence petite by
	Proposition~\ref{prop:fixed-h-skeleton}. The
	Meyn--Tweedie theorem then yields
	\(V_q\)-uniform exponential ergodicity.
	
	\begin{thm}[Exponential ergodicity]
		\label{thm:hybrid-exp-erg}
		Suppose Assumptions~\ref{ass:Levy}, \ref{ass:erg-coeff} hold. Then, for every
		\(r>0\), with $V_r(x,i):=(1+x^2)^{r/2}$,
		there exist constants \(a_r,b_r>0\) such that
		\[
		\mathcal A_0^{\mathrm{ext}}V_r(x,i)
		=
		\mathcal A_0V_r(x,i)
		\le
		-a_rV_r(x,i)+b_r,
		\qquad (x,i)\in\mathbb R\times\mathbb S .
		\]
		Moreover, \((X,\Lambda)\) admits a unique invariant probability measure
		\(\pi_0\), \(\pi_0(V_r)<\infty\), and there exist constants
		\(B_r<\infty\), \(\rho_r>0\) such that
		\[
		\sup_{|f|\le V_r}
		\left|
		P_tf(x,i)-\pi_0(f)
		\right|
		\le
		B_rV_r(x,i)e^{-\rho_rt},
		\qquad t\ge0 .
		\]
	\end{thm}

	\begin{rem}
		Theorem~\ref{thm:hybrid-exp-erg} uses a regime-independent Lyapunov
		function, so the switching part of the generator vanishes. More flexible
		criteria are possible with weighted functions
		\(V(x,i)=\beta_i(1+x^2)^{q/2}\), for which the switching term may contribute
		to stabilization, but we do not pursue this extension here.
	\end{rem}


	\section{Estimation}
	\label{sec:estimation}
	
	This section studies high-frequency inference for the full parameter
	\[
	\zeta=(\alpha,\gamma,\vartheta)
	\in
	\Theta_\alpha\times\Theta_\gamma\times\Theta_\vartheta .
	\]
	The observed hybrid path
	$\{(X_{t_j},\Lambda_{t_j})\}_{j=0}^n$
	contains two types of local information. The increments of \(X\) identify
	\((\alpha,\gamma)\), while the endpoint transitions of \(\Lambda\) identify
	\(\vartheta\). We therefore use a three-stage contrast:
	a Gaussian quasi-likelihood for the scale and drift, and an
	intensity-type quasi-likelihood for the switching rates.
	
	\subsection{Contrasts and the three-stage estimator}
	\label{sec:three-stage-estimator}

	For $i\neq k$, define the observed one-step transition indicator
	\[
	\Del_j N_{ik}^{\mathrm{o}}
	:=
	\mathbf 1_{\{\Lam_{t_{j-1}}=i,\ \Lam_{t_j}=k\}},
	\qquad j=1,\dots,n.
	\]
	The variable \(\Del_j N_{ik}^{\mathrm{o}}\) is an endpoint transition
	indicator rather than the full continuous-time transition count. Multiple
	switches within a single interval have probability \(O(h_n^2)\) under the
	bounded-rate condition, and hence are negligible on the \(T_n\)-scale.
	By the state-dependent switching property, one has
	\[
	\Prob_{\vartheta}\!\left(\Lam_{t_j}=k\mid \mathcal F_{t_{j-1}}\right)
	=
	\begin{cases}
		q_{ik}(X_{t_{j-1}},\vartheta)\,h_n+O(h_n^2),
		& \text{on }\{\Lam_{t_{j-1}}=i\},\ k\neq i,\\[1mm]
		1-q_i(X_{t_{j-1}},\vartheta)\,h_n+O(h_n^2),
		& \text{on }\{\Lam_{t_{j-1}}=i\},\ k=i.
	\end{cases}
	\]
	
	\begin{itemize}
		\item For the continuous component, define
		\begin{equation}
			\label{eq:G1-stage}
			\mathbb G_{1,n}(\gamma)
			:=
			-\frac1{2T_n}\sum_{j=1}^n
			\left\{
			h_n\log c_{j-1}(\gamma)^2
			+
			\frac{(\Delta_jX)^2}{c_{j-1}(\gamma)^2}
			\right\},
		\end{equation}
		and
		\begin{equation}
			\label{eq:G2-stage}
			\mathbb G_{2,n}(\alpha;\gamma)
			:=
			-\frac1{2T_n}\sum_{j=1}^n
			\frac{\{\Delta_jX-h_nb_{j-1}(\alpha)\}^2}
			{h_nc_{j-1}(\gamma)^2}.
		\end{equation}
		
		\item 
		For the switching component,
		define
		\begin{equation}
			\label{eq:G3-stage}
			\mathbb G_{3,n}(\vartheta)
			:=
			\frac1{T_n}
			\sum_{j=1}^n\sum_{i=1}^m\sum_{k\neq i}
			\Del_j N_{ik}^{\mathrm{o}}
			\log q_{ik}(X_{t_{j-1}},\vartheta)
			-
			\frac1{T_n}
			\sum_{j=1}^nh_n\sum_{i=1}^m
			\mathbf 1_{\{\Lambda_{t_{j-1}}=i\}}
			q_i(X_{t_{j-1}},\vartheta).
		\end{equation}
	\end{itemize}

	The full estimator is constructed by
	\begin{equation}
		\label{eq:full-three-stage-estimator}
		\hat\gamma_n
		\in
		\argmax_{\gamma\in\Theta_\gamma}\mathbb G_{1,n}(\gamma),
		\qquad
		\hat\alpha_n
		\in
		\argmax_{\alpha\in\Theta_\alpha}
		\mathbb G_{2,n}(\alpha;\hat\gamma_n),
	\end{equation}
	and
	\begin{equation}
		\label{eq:vartheta-estimator}
		\hat\vartheta_n
		\in
		\argmax_{\vartheta\in\Theta_\vartheta}\mathbb G_{3,n}(\vartheta).
	\end{equation}
	We write
	\begin{equation}
		\label{eq:full-estimator}
		\hat\zeta_n
		:=
		(\hat\alpha_n,\hat\gamma_n,\hat\vartheta_n).
	\end{equation}
	
	The three stages use the leading local characteristics of the same observed
	hybrid path. For the continuous coordinate,
	\[
	\E_{j-1}[\Delta_jX]\simeq h_nb_{j-1}(\alpha),
	\qquad
	\E_{j-1}\!\left[
	\{\Delta_jX-h_nb_{j-1}(\alpha)\}^2
	\right]
	\simeq h_nc_{j-1}(\gamma)^2 .
	\]
	Thus \(\mathbb G_{1,n}\) and \(\mathbb G_{2,n}\) are Gaussian
	quasi-likelihood contrasts based on the local variance and local mean of
	\(\Delta_jX\), respectively. This is the standard high-frequency
	Gaussian quasi-likelihood construction for ergodic diffusions and
	L\'evy-driven SDEs \cite{Kes97,Gobet2002,UchidaYoshida2012,Mas13as}, with
	the staged scale--drift construction in the spirit of
	\cite{MasudaUehara2017}. For the switching coordinate,
	\[
	\E_{j-1}[\Del_j N_{ik}^{\mathrm{o}}]
	\simeq
	h_n\mathbf 1_{\{\Lambda_{t_{j-1}}=i\}}
	q_{ik}(X_{t_{j-1}},\vartheta),
	\qquad i\neq k,
	\]
	which leads to the intensity-type contrast \eqref{eq:G3-stage}, the
	high-frequency discrete analogue of the counting-process likelihood
	\cite{Aalen1978,Borgan1984,Andersen1993} and of likelihoods for Markov jump
	processes \cite{bladt2005statistical}. 
	
	These contrasts are used
	for a pure-jump state-dependent hybrid system: the continuous contrast must
	be stable under within-step regime changes, while the switching contrast has
	state-dependent intensities evaluated along the same ergodic path
	\((X,\Lambda)\).
	
	\begin{rem}
		The three criteria may be viewed as components of the composite contrast
		\[
		\mathbb G_n(\alpha,\gamma,\vartheta)
		:=
		\mathbb G_{1,n}(\gamma)
		+
		\mathbb G_{2,n}(\alpha;\gamma)
		+
		\mathbb G_{3,n}(\vartheta).
		\]
		We use staged maximization because the scale, drift, and switching blocks
		enter the high-frequency asymptotics through different local structures.
	\end{rem}
	
	\subsection{Hybrid one-step structure}
	\label{sec:hybrid-one-step}
	
	The hybrid structure affects the continuous
	quasi-likelihood in the short-time expansion of \(\Delta_jX\). Over
	\([t_{j-1},t_j]\), the regime may switch, and this produces an additional
	within-step error absent from ordinary L\'evy-driven SDEs.

	Set
	\[
	\eta_j:=\Del_j X-h_n b_{j-1}(\al_0),
	\qquad
	\xi_j:=\frac{\eta_j^2}{h_n c_{j-1}(\gam_0)^2}-1,
	\qquad
	\rho_{j-1}(\gam):=\frac{c_{j-1}(\gam_0)^2}{c_{j-1}(\gam)^2}-1,
	\]
	and decompose the observed increment as
	\[
	\Del_j X=h_n b_{j-1}(\al_0)+c_{j-1}(\gam_0)\,\Del_j L+r_{j,n}.
	\]
	The following proposition collects the moment estimates for $\eta_j$, $\xi_j$, and $r_{j,n}$ used throughout the proofs.

	\begin{prop}
		\label{prop:hybrid-one-step}
		Suppose Assumptions~\ref{ass:Levy}--\ref{ass:ergodic} hold. Then, for every \(q>0\),
		\begin{enumerate}
			\item $\mathbb E\!\left[
			|r_{j,n}|^q
			\mid\mathcal F_{t_{j-1}}
			\right]
			\le
			Ch_n^{q\wedge2}
			(1+|X_{t_{j-1}}|^{C})$;
			\item $\mathbb E[\eta_j\mid\mathcal F_{t_{j-1}}]
			=
			O\!\left(h_n^2(1+|X_{t_{j-1}}|^{C})\right)$;
			\item $\mathbb E[\eta_j^2\mid\mathcal F_{t_{j-1}}]
			=
			h_nc_{j-1}(\gamma_0)^2
			+
			O\!\left(h_n^{3/2}(1+|X_{t_{j-1}}|^{C})\right)$;
			\item $\mathbb E[\eta_j^3\mid\mathcal F_{t_{j-1}}]
			=
			h_n\mathbb E[L_1^3]c_{j-1}(\gamma_0)^3
			+
			O\!\left(h_n^{3/2}(1+|X_{t_{j-1}}|^{C})\right)$;
			\item $\mathbb E\!\left[
			|\eta_j|^q
			\mid\mathcal F_{t_{j-1}}
			\right]
			\le
			C h_n^{q/2\wedge1}
			(1+|X_{t_{j-1}}|^{C})$.
		\end{enumerate}
	\end{prop}
	Let 
	$\chi_j^{ik}:=\Del_j N_{ik}^{\mathrm o}-\E_{j-1}\!\left[\Del_j N_{ik}^{\mathrm o}\right]$. The following proposition collects the estimates for the switching component.
	\begin{prop}
		\label{prop:hybrid-switching}
		Suppose Assumptions~\ref{ass:Levy}--\ref{ass:ergodic} hold. Then, for every \(i\neq k\),
		\begin{enumerate}
			\item $\E\!\left[
			\Delta_j N_{ik}^{\mathrm{o}}
			\Bigm| \mathcal F_{t_{j-1}}
			\right]
			=
			\mathbf 1_{\{\Lambda_{t_{j-1}}=i\}}
			q_{ik}(X_{t_{j-1}},\vartheta_0)\,h_n
			+r_{j,n}^{ik}$ with \[|r_{j,n}^{ik}|
			\le
			C h_n^{3/2}\bigl(1+|X_{t_{j-1}}|^{C}\bigr);\]
			\item let \(F_{j-1}\) be
			\(\mathcal F_{t_{j-1}}\)-measurable satisfying
			$|F_{j-1}|\le C(1+|X_{t_{j-1}}|^C)$,
			\[
			\left|
			\mathbb E_{j-1}
			\left[
			F_{j-1}\Delta_jL\,\chi_j^{ik}
			\right]
			\right|
			\le
			Ch_n^{3/2}(1+|X_{t_{j-1}}|^C),
			\]
			\[
			\left|
			\mathbb E_{j-1}
			\left[
			F_{j-1}\{(\Delta_jL)^2-h_n\}\chi_j^{ik}
			\right]
			\right|
			\le
			Ch_n^{3/2}(1+|X_{t_{j-1}}|^C).
			\] 
		\end{enumerate}
	\end{prop}
	The proofs of these two propositions are given in Appendix~\ref{sec:tech}.
	
	\subsection{Consistency}
	\label{sec:consistency}

	\begin{thm}[Consistency of the full estimator]
		\label{thm:full-consistency}
		Suppose Assumptions~\ref{ass:Levy}--\ref{ass:ergodic} and
		\ref{ass:Y-global} hold. Then
		\[
		\hat\zeta_n
		=
		(\hat\alpha_n,\hat\gamma_n,\hat\vartheta_n)
		\cip
		\zeta_0 .
		\]
	\end{thm}
	
	\subsection{Joint asymptotic normality}
	\label{sec:joint-an}
	
	We next give the joint limit theorem for the full three-stage estimator.
	The scale and drift scores have the same form as in the staged
	Gaussian quasi-likelihood analysis, while the switching-rate score comes
	from the intensity-type contrast. Define
	\[
	\Delta_{n,\gamma}
	:=
	T_n^{1/2}\partial_\gamma\mathbb G_{1,n}(\gamma_0),
	\quad
	\Delta_{n,\alpha}
	:=
	T_n^{1/2}\partial_\alpha\mathbb G_{2,n}(\alpha_0;\gamma_0),
	\quad
	\Delta_{n,Q}
	:=
	T_n^{1/2}\partial_\vartheta\mathbb G_{3,n}(\vartheta_0).
	\]
	Let
	$\Delta_{n,\zeta}
	:= (\Delta_{n,\alpha},\Delta_{n,\gamma},\Delta_{n,Q})^{\top}$.
	We set
	\[
	\Sigma_\gamma
	:=
	\frac{\kappa_4}{4}
	\int_E
	\Psi_\gamma(x,i)^{\otimes2}\,\pi_0(dx,di),
	\quad
	\Sigma_{\alpha\gamma}
	:=
	\frac{\kappa_3}{2}
	\int_E
	A_\alpha(x,i)\Psi_\gamma(x,i)^\top
	\,\pi_0(dx,di),
	\]
	where
	$\kappa_3:=\int_{\mathbb R}z^3\,\nu(dz)$ and
	$\kappa_4:=\int_{\mathbb R}z^4\,\nu(dz)$.
	Then define
	\[
	\Sigma_\zeta
	:=
	\begin{pmatrix}
		\Gamma_\alpha & \Sigma_{\alpha\gamma} & 0\\
		\Sigma_{\alpha\gamma}^\top & \Sigma_\gamma & 0\\
		0 & 0 & \Gamma_Q
	\end{pmatrix},
	\qquad
	\Gamma_\zeta
	:=
	\begin{pmatrix}
		\Gamma_\alpha & 0 & 0\\
		0 & \Gamma_\gamma & 0\\
		0 & 0 & \Gamma_Q
	\end{pmatrix}.
	\]
	
	\begin{thm}[Joint asymptotic normality] \label{thm:joint-an} Suppose Assumptions~\ref{ass:Levy}--\ref{ass:positive-defi} hold. Then \[ \Delta_{n,\zeta}\cil N(0,\Sigma_\zeta), \] and the full estimator admits the linear expansion \[ \sqrt{T_n}(\hat\zeta_n-\zeta_0) = \Gamma_\zeta^{-1}\Delta_{n,\zeta} + o_p(1). \] Consequently, \[ \sqrt{T_n}(\hat\zeta_n-\zeta_0) \cil N_{p_\alpha+p_\gamma+p_\vartheta} \left(0, \Gamma_\zeta^{-1}\Sigma_\zeta\Gamma_\zeta^{-1} \right). \] \end{thm}
	
	The off-diagonal block between the drift and scale estimators is governed
	by \(\Sigma_{\alpha\gamma}\), and hence by the third moment of the driving
	L\'evy noise. In contrast, the switching-rate score is asymptotically
	uncorrelated with the continuous-coefficient scores. This does not mean that
	the continuous and switching coordinates are independent at finite samples;
	rather, their predictable cross-covariations are of smaller order on the
	\(T_n^{1/2}\)-scale.
	
	\begin{rem}
		The asymptotic normality result above should be understood as a
		quasi-likelihood limit theorem, not as an efficiency statement in the
		Hájek--Le Cam sense. The Gaussian contrasts for the continuous component use
		only the local mean and variance structure of the Lévy-driven increments and
		are not, in general, the exact likelihood contrasts of the underlying
		experiment. Consequently the covariance matrix has the sandwich form
		\(\Gamma_\zeta^{-1}\Sigma_\zeta\Gamma_\zeta^{-1}\), and no claim is made
		that it coincides with the inverse efficient information. The
		switching-rate block is closer to a genuine counting-process likelihood,
		but efficiency of the full estimator would require a separate LAN analysis
		of the exact hybrid experiment.
	\end{rem}

	\subsection{Polynomial-type large deviation inequality and moment consequences}
	\label{sec:pldi}
	
	In this section we establish a polynomial-type large deviation inequality
	for the contrast functions $\mathbb G_{1,n}$, $\mathbb G_{2,n}$, and $\mathbb G_{3,n}$,
	in the spirit of \cite{Yos11,Yoshida2021,Mas13as}.
	
	For each $r>0$, define the local parameter neighborhoods
	\[
	\gam_n(v):=\gam_0+T_n^{-1/2}v,
	\qquad
	\mathbb U_{n,\gam}(r):=\{v:\gam_n(v)\in\Theta_\gam,\ |v|\ge r\},
	\]
	\[
	\al_n(u):=\al_0+T_n^{-1/2}u,
	\qquad
	\mathbb U_{n,\al}(r):=\{u:\al_n(u)\in\Theta_\al,\ |u|\ge r\},
	\]
	\[
	\vartheta_n(w):=\vartheta_0+T_n^{-1/2}w,
	\qquad
	\mathbb U_{n,\vartheta}(r):=\{w:\vartheta_n(w)\in\Theta_\vartheta,\ |w|\ge r\},
	\]
	and the associated local random fields
	\[
	\mathbb Z_{n,\gam}(v)
	:=\exp\!\Bigl(T_n\bigl\{\mathbb G_{1,n}(\gam_n(v))-\mathbb G_{1,n}(\gam_0)\bigr\}\Bigr),
	\]
	\[
	\mathbb Z_{n,\al}(u;\bar\gam_n)
	:=\exp\!\Bigl(T_n\bigl\{\mathbb G_{2,n}(\al_n(u);\bar\gam_n)-\mathbb G_{2,n}(\al_0;\bar\gam_n)\bigr\}\Bigr),
	\qquad \bar\gam_n\in\Theta_\gam,
	\]
	\[
	\mathbb Z_{n,Q}(w)
	:=\exp\!\Bigl(T_n\bigl\{\mathbb G_{3,n}(\vartheta_n(w))-\mathbb G_{3,n}(\vartheta_0)\bigr\}\Bigr).
	\]

	\begin{thm}[PLDI for the three-stage estimator]
		\label{thm:pldi}
		Suppose Assumptions~\ref{ass:Levy}--\ref{ass:ergodic},
		\ref{ass:Y-global}, and \ref{ass:positive-defi} hold. Then, for every
		\(L>0\), there exists \(C_L>0\) such that, for all \(r>0\),
		\[
		\sup_{n\in\mathbb N}
		\mathbb P
		\left(
		\sup_{v\in\mathbb U_{n,\gamma}(r)}
		\mathbb Z_{n,\gamma}(v)
		\ge
		e^{-r^2/C_L}
		\right)
		\le
		\frac{C_L}{r^L},\qquad r>0,
		\]
		\[
		\sup_{n\in\mathbb N}
		\mathbb P
		\left(
		\sup_{u\in\mathbb U_{n,\alpha}(r)}
		\mathbb Z_{n,\alpha}(u;\hat\gamma_n)
		\ge
		e^{-r^2/C_L}
		\right)
		\le
		\frac{C_L}{r^L},\qquad r>0,
		\]
		and
		\[
		\sup_{n\in\mathbb N}
		\mathbb P
		\left(
		\sup_{w\in\mathbb U_{n,\vartheta}(r)}
		\mathbb Z_{n,Q}(w)
		\ge
		e^{-r^2/C_L}
		\right)
		\le
		\frac{C_L}{r^L},\qquad r>0.
		\]
		Consequently, for every \(L>0\), there exists \(C_L'>0\) such that
		\[
		\sup_{n\in\mathbb N}
		\mathbb P
		\left(
		\sqrt{T_n}|\hat\zeta_n-\zeta_0|>r
		\right)
		\le
		\frac{C_L'}{r^L},\qquad r>0.
		\]
	\end{thm}

	\begin{cor}[Moment convergence]
		\label{cor:moments}
		Suppose the assumptions of Theorems~\ref{thm:joint-an} and
		\ref{thm:pldi} hold. Set
		\[
		Y_n:=\sqrt{T_n}(\hat\zeta_n-\zeta_0),
		\qquad
		Y\sim N(0,\Gamma_\zeta^{-1}\Sigma_\zeta\Gamma_\zeta^{-1}).
		\]
		Then, for every continuous function \(f\) of polynomial growth,
		\[
		\mathbb E[f(Y_n)]\to \mathbb E[f(Y)].
		\]
	\end{cor}

	\section{Numerical experiments}
	\label{sec:numerical}
	We present a short simulation study to illustrate the finite-sample behavior
	of the three-stage estimator
	\[
	\hat\zeta_n=(\hat\alpha_n,\hat\gamma_n,\hat\vartheta_n).
	\]
	The experiments are based on two hybrid
	switching SDEs driven by normal inverse Gaussian noise.
	We use \(\mathbb S=\{1,2\}\) and
	\[
	q_{12}(x,\vartheta)=\exp\{\vartheta_{10}+\vartheta_{11}\tanh x\},
	\qquad
	q_{21}(x,\vartheta)=\exp\{\vartheta_{20}+\vartheta_{21}\tanh x\}.
	\]
	
	We consider the following two models.
	\[
	\begin{aligned}
		b(x,i,\alpha)&=-ax+\mu_i,
		\qquad
		c(x,i,\gamma)=\exp(\gamma_i),\\
		\alpha&=(a,\mu_1,\mu_2)=(1.2,-0.8,0.8),\\
		\gamma&=(\gamma_1,\gamma_2)=(-0.35,0.25),\\
		\vartheta&=(-1.2,0.7,-1.0,-0.6),\\
		L_1&\sim \mathrm{NIG}(1,0,1,0),
	\end{aligned}
	\]
	for Model 1, and
	\[
	\begin{aligned}
		b(x,i,\alpha)
		&=-ax+\mu_i+\rho_i\tanh x,\\
		c(x,i,\gamma)
		&=\exp\{\gamma_0+\gamma_1\mathbf 1_{\{i=2\}}+\gamma_2\tanh x\},\\
		\alpha&=(a,\mu_1,\mu_2,\rho_1,\rho_2)=(1.0,-0.5,0.6,0.4,-0.3),\\
		\gamma&=(\gamma_0,\gamma_1,\gamma_2)=(-0.2,0.35,0.25),\\
		\vartheta&=(-1.3,0.8,-1.1,-0.7),\\
		L_1&\sim \mathrm{NIG}(2,1,3\sqrt3/4,-3/4),
	\end{aligned}
	\]
	for Model 2.
	
	For each model we generate \(R=300\) replications. Each replication is
	simulated on an internal grid with step \(\Delta=0.001\) after a burn-in of
	length \(100\), and then subsampled to the observed mesh. The nine sampling
	designs are listed in Table~\ref{tab:num-design}. For each data set we
	compute \(\hat\zeta_n\) by the three-stage procedure
	\eqref{eq:full-three-stage-estimator}--\eqref{eq:full-estimator}, that is,
	by maximizing \(\mathbb G_{1,n}\), \(\mathbb G_{2,n}\), and
	\(\mathbb G_{3,n}\) in sequence.
	The tables report Bias, SD, and RMSE.
	
	For the asymptotic-normality diagnostics, let \(\xi\) be one component of
	\(\hat\zeta_n\). We define the standardized
	error in replication \(r\) by
	\[
	Z_{\xi}^{(r)}
	:=
	\frac{\sqrt{T}\,(\hat\xi^{(r)}-\xi_0)}{\hat s_{\xi}^{(r)}},
	\]
	where \(\xi_0\) is the true value and \(\hat s_{\xi}^{(r)}\) is the plug-in
	asymptotic standard error computed from the corresponding diagonal entry of
	the plug-in covariance matrix associated with the joint limit in
	Theorem~\ref{thm:joint-an}. Thus, if the asymptotic normal approximation is accurate,
	the distribution of \(Z_\xi^{(r)}\) should be close to \(N(0,1)\), so the
	histogram should be centered near \(0\) with variance near \(1\), and the QQ
	plot should be close to a straight line.
	
	\begin{table}[t]
		\centering
		\small
		\caption{Sampling designs. The entries are \(n=T/h\).}
		\label{tab:num-design}
		\begin{tabular}{c|ccc}
			\toprule
			& $h=0.02$ & $h=0.01$ & $h=0.005$ \\
			\midrule
			$T=50$ & 2500 & 5000 & 10000 \\
			$T=100$ & 5000 & 10000 & 20000 \\
			$T=200$ & 10000 & 20000 & 40000 \\
			\bottomrule
		\end{tabular}
	\end{table}
	
	Tables~\ref{tab:num-theta-model1-alpha}--\ref{tab:num-q-model2} and
	Figures~\ref{fig:num-hist}--\ref{fig:num-qq} summarize the results. The main
	finite-sample effect is the horizon length \(T\): for both models and for
	all three blocks of \(\hat\zeta_n\), enlarging \(T\) produces a clear
	reduction in SD and RMSE, whereas the additional gain from refining \(h\)
	over \(\{0.02,0.01,0.005\}\) is comparatively small.
	This is already visible in Model 1 for the \(\alpha\)- and
	\(\vartheta\)-blocks, and becomes even clearer in the harder nonlinear Model 2.
	For example, in Model 2 the RMSE of \(\hat\rho_{1,n}\) decreases from about
	\(0.96\) at \(T=50\) to \(0.74\) at \(T=100\) and \(0.50\) at \(T=200\), while
	the RMSE of \(\hat\vartheta_{11,n}\) decreases from about \(1.05\) to \(0.66\)
	and \(0.39\).
	
	The standardized diagnostics support the asymptotic normality results.
	For the longest design \((T,h)=(200,0.005)\), the histograms and QQ plots show
	that representative components from the \(\alpha\)- and \(\vartheta\)-blocks
	are already close
	to the \(N(0,1)\) benchmark after normalization. The fit is less accurate for
	the \(\gamma\)-block, whose standardized errors remain mildly over-dispersed, but
	the overall behavior is consistent with the \(\sqrt{T}\)-normalization in the
	sample sizes considered here. The additional \(\vartheta\)-only diagnostics
	show that this approximation is stable across all four switching-rate
	components in both models, with empirical variances close to one at
	\((T,h)=(200,0.005)\). The section ends with compact tables and
	four diagnostic plots.
	
	\begin{table}[t]
		\centering
		\scriptsize
		\caption{Model 1: \(\alpha\)-block of \(\hat\zeta_n\). Entries are Bias/SD/RMSE over \(R=300\) replications.}
		\label{tab:num-theta-model1-alpha}
		\setlength{\tabcolsep}{3pt}
		\resizebox{\textwidth}{!}{%
			\begin{tabular}{lccccccccc}
				\toprule
				& \multicolumn{3}{c}{$T=50$} & \multicolumn{3}{c}{$T=100$} & \multicolumn{3}{c}{$T=200$} \\
				\cmidrule(lr){2-4}\cmidrule(lr){5-7}\cmidrule(lr){8-10}
				Parameter & $h=0.02$ & $0.01$ & $0.005$ & $0.02$ & $0.01$ & $0.005$ & $0.02$ & $0.01$ & $0.005$ \\
				\midrule
				$a$ & 0.08/0.20/0.22 & 0.08/0.21/0.22 & 0.09/0.21/0.23 & 0.03/0.14/0.14 & 0.04/0.14/0.14 & 0.04/0.14/0.14 & 0.00/0.09/0.09 & 0.01/0.09/0.09 & 0.02/0.09/0.09 \\
				$\mu_1$ & -0.04/0.18/0.18 & -0.05/0.18/0.19 & -0.05/0.18/0.19 & -0.01/0.12/0.12 & -0.01/0.12/0.12 & -0.02/0.12/0.12 & 0.01/0.08/0.08 & 0.00/0.08/0.08 & -0.00/0.08/0.08 \\
				$\mu_2$ & 0.01/0.33/0.33 & 0.01/0.33/0.33 & 0.02/0.33/0.33 & -0.02/0.22/0.22 & -0.01/0.22/0.22 & -0.01/0.22/0.22 & -0.02/0.14/0.14 & -0.01/0.14/0.14 & -0.01/0.14/0.14 \\
				\bottomrule
			\end{tabular}
		}
	\end{table}
	
	\begin{table}[t]
		\centering
		\scriptsize
		\caption{Model 1: \(\gamma\)-block of \(\hat\zeta_n\). Entries are Bias/SD/RMSE over \(R=300\) replications.}
		\label{tab:num-theta-model1-gamma}
		\setlength{\tabcolsep}{3pt}
		\resizebox{\textwidth}{!}{%
			\begin{tabular}{lccccccccc}
				\toprule
				& \multicolumn{3}{c}{$T=50$} & \multicolumn{3}{c}{$T=100$} & \multicolumn{3}{c}{$T=200$} \\
				\cmidrule(lr){2-4}\cmidrule(lr){5-7}\cmidrule(lr){8-10}
				Parameter & $h=0.02$ & $0.01$ & $0.005$ & $0.02$ & $0.01$ & $0.005$ & $0.02$ & $0.01$ & $0.005$ \\
				\midrule
				$\gamma_1$ & -0.03/0.15/0.15 & -0.02/0.15/0.15 & -0.02/0.15/0.15 & -0.02/0.11/0.11 & -0.02/0.11/0.11 & -0.02/0.11/0.11 & -0.01/0.08/0.08 & -0.01/0.08/0.08 & -0.01/0.08/0.08 \\
				$\gamma_2$ & -0.03/0.20/0.20 & -0.03/0.20/0.20 & -0.03/0.20/0.20 & -0.02/0.13/0.13 & -0.02/0.13/0.13 & -0.02/0.13/0.13 & -0.01/0.09/0.09 & -0.00/0.09/0.09 & -0.00/0.09/0.09 \\
				\bottomrule
			\end{tabular}
		}
	\end{table}
	
	\begin{table}[t]
		\centering
		\scriptsize
		\caption{Model 2: \(\alpha\)-block of \(\hat\zeta_n\). Entries are Bias/SD/RMSE over \(R=300\) replications.}
		\label{tab:num-theta-model2-alpha}
		\setlength{\tabcolsep}{3pt}
		\resizebox{\textwidth}{!}{%
			\begin{tabular}{lccccccccc}
				\toprule
				& \multicolumn{3}{c}{$T=50$} & \multicolumn{3}{c}{$T=100$} & \multicolumn{3}{c}{$T=200$} \\
				\cmidrule(lr){2-4}\cmidrule(lr){5-7}\cmidrule(lr){8-10}
				Parameter & $h=0.02$ & $0.01$ & $0.005$ & $0.02$ & $0.01$ & $0.005$ & $0.02$ & $0.01$ & $0.005$ \\
				\midrule
				$a$ & 0.23/0.67/0.71 & 0.23/0.67/0.71 & 0.24/0.67/0.71 & 0.15/0.48/0.50 & 0.15/0.48/0.51 & 0.15/0.48/0.51 & 0.09/0.31/0.32 & 0.09/0.32/0.33 & 0.10/0.32/0.33 \\
				$\mu_1$ & -0.02/0.29/0.29 & -0.02/0.30/0.30 & -0.03/0.30/0.30 & -0.01/0.18/0.18 & -0.02/0.18/0.18 & -0.02/0.18/0.18 & -0.00/0.11/0.11 & -0.01/0.11/0.11 & -0.01/0.11/0.11 \\
				$\mu_2$ & 0.05/0.35/0.35 & 0.06/0.35/0.36 & 0.06/0.36/0.36 & 0.02/0.23/0.23 & 0.03/0.23/0.24 & 0.03/0.24/0.24 & 0.01/0.15/0.15 & 0.02/0.15/0.15 & 0.02/0.15/0.15 \\
				$\rho_1$ & 0.16/0.95/0.96 & 0.15/0.95/0.96 & 0.16/0.95/0.96 & 0.17/0.72/0.74 & 0.16/0.73/0.74 & 0.17/0.73/0.75 & 0.11/0.48/0.49 & 0.11/0.49/0.50 & 0.11/0.49/0.50 \\
				$\rho_2$ & 0.14/1.10/1.11 & 0.13/1.10/1.11 & 0.13/1.11/1.12 & 0.11/0.81/0.82 & 0.10/0.81/0.82 & 0.10/0.81/0.82 & 0.10/0.57/0.57 & 0.09/0.57/0.58 & 0.09/0.57/0.58 \\
				\bottomrule
			\end{tabular}
		}
	\end{table}
	
	\begin{table}[t]
		\centering
		\scriptsize
		\caption{Model 2: \(\gamma\)-block of \(\hat\zeta_n\). Entries are Bias/SD/RMSE over \(R=300\) replications.}
		\label{tab:num-theta-model2-gamma}
		\setlength{\tabcolsep}{3pt}
		\resizebox{\textwidth}{!}{%
			\begin{tabular}{lccccccccc}
				\toprule
				& \multicolumn{3}{c}{$T=50$} & \multicolumn{3}{c}{$T=100$} & \multicolumn{3}{c}{$T=200$} \\
				\cmidrule(lr){2-4}\cmidrule(lr){5-7}\cmidrule(lr){8-10}
				Parameter & $h=0.02$ & $0.01$ & $0.005$ & $0.02$ & $0.01$ & $0.005$ & $0.02$ & $0.01$ & $0.005$ \\
				\midrule
				$\gamma_0$ & -0.04/0.19/0.19 & -0.04/0.19/0.19 & -0.04/0.19/0.19 & -0.03/0.12/0.13 & -0.03/0.12/0.13 & -0.03/0.12/0.13 & -0.02/0.09/0.09 & -0.01/0.09/0.09 & -0.01/0.09/0.09 \\
				$\gamma_1$ & 0.02/0.27/0.27 & 0.02/0.27/0.27 & 0.02/0.27/0.27 & 0.01/0.18/0.18 & 0.01/0.18/0.18 & 0.01/0.18/0.18 & 0.01/0.14/0.14 & 0.01/0.14/0.14 & 0.01/0.14/0.14 \\
				$\gamma_2$ & -0.03/0.22/0.22 & -0.04/0.22/0.22 & -0.04/0.22/0.22 & -0.02/0.15/0.15 & -0.02/0.15/0.15 & -0.02/0.16/0.16 & -0.01/0.11/0.11 & -0.01/0.11/0.11 & -0.01/0.11/0.11 \\
				\bottomrule
			\end{tabular}
		}
	\end{table}
	
	\begin{table}[t]
		\centering
		\scriptsize
		\caption{Model 1: \(\vartheta\)-block of \(\hat\zeta_n\). Entries are Bias/SD/RMSE over \(R=300\) replications.}
		\label{tab:num-q-model1}
		\setlength{\tabcolsep}{3pt}
		\resizebox{\textwidth}{!}{%
			\begin{tabular}{lccccccccc}
				\toprule
				& \multicolumn{3}{c}{$T=50$} & \multicolumn{3}{c}{$T=100$} & \multicolumn{3}{c}{$T=200$} \\
				\cmidrule(lr){2-4}\cmidrule(lr){5-7}\cmidrule(lr){8-10}
				Parameter & $h=0.02$ & $0.01$ & $0.005$ & $0.02$ & $0.01$ & $0.005$ & $0.02$ & $0.01$ & $0.005$ \\
				\midrule
				$\vartheta_{10}$ & -0.20/0.68/0.71 & -0.20/0.68/0.71 & -0.20/0.68/0.71 & -0.11/0.44/0.45 & -0.11/0.44/0.46 & -0.10/0.44/0.46 & -0.04/0.24/0.25 & -0.03/0.24/0.25 & -0.03/0.24/0.25 \\
				$\vartheta_{11}$ & -0.20/1.18/1.20 & -0.21/1.18/1.19 & -0.21/1.17/1.19 & -0.16/0.77/0.78 & -0.16/0.77/0.78 & -0.16/0.77/0.79 & -0.05/0.45/0.45 & -0.06/0.46/0.46 & -0.05/0.46/0.46 \\
				$\vartheta_{20}$ & -0.08/0.54/0.55 & -0.08/0.54/0.55 & -0.08/0.54/0.55 & -0.05/0.32/0.32 & -0.04/0.32/0.32 & -0.04/0.31/0.32 & -0.04/0.21/0.21 & -0.03/0.21/0.21 & -0.03/0.21/0.21 \\
				$\vartheta_{21}$ & 0.13/0.90/0.91 & 0.12/0.90/0.91 & 0.12/0.90/0.91 & 0.04/0.56/0.56 & 0.03/0.56/0.56 & 0.03/0.56/0.56 & 0.03/0.38/0.38 & 0.02/0.37/0.37 & 0.02/0.37/0.37 \\
				\bottomrule
			\end{tabular}
		}
	\end{table}
	
	\begin{table}[t]
		\centering
		\scriptsize
		\caption{Model 2: \(\vartheta\)-block of \(\hat\zeta_n\). Entries are Bias/SD/RMSE over \(R=300\) replications.}
		\label{tab:num-q-model2}
		\setlength{\tabcolsep}{3pt}
		\resizebox{\textwidth}{!}{%
			\begin{tabular}{lccccccccc}
				\toprule
				& \multicolumn{3}{c}{$T=50$} & \multicolumn{3}{c}{$T=100$} & \multicolumn{3}{c}{$T=200$} \\
				\cmidrule(lr){2-4}\cmidrule(lr){5-7}\cmidrule(lr){8-10}
				Parameter & $h=0.02$ & $0.01$ & $0.005$ & $0.02$ & $0.01$ & $0.005$ & $0.02$ & $0.01$ & $0.005$ \\
				\midrule
				$\vartheta_{10}$ & -0.17/0.65/0.67 & -0.17/0.66/0.68 & -0.17/0.66/0.68 & -0.08/0.42/0.43 & -0.08/0.42/0.42 & -0.07/0.42/0.42 & -0.04/0.24/0.24 & -0.04/0.23/0.24 & -0.04/0.23/0.24 \\
				$\vartheta_{11}$ & -0.18/1.02/1.04 & -0.18/1.03/1.05 & -0.18/1.03/1.05 & -0.08/0.65/0.65 & -0.07/0.65/0.66 & -0.07/0.65/0.66 & -0.06/0.39/0.40 & -0.05/0.39/0.39 & -0.05/0.39/0.39 \\
				$\vartheta_{20}$ & -0.12/0.55/0.57 & -0.12/0.56/0.57 & -0.12/0.55/0.57 & -0.06/0.32/0.32 & -0.05/0.32/0.32 & -0.05/0.32/0.32 & -0.02/0.21/0.21 & -0.02/0.20/0.20 & -0.02/0.20/0.20 \\
				$\vartheta_{21}$ & 0.17/1.05/1.06 & 0.16/1.05/1.06 & 0.16/1.04/1.05 & 0.06/0.63/0.63 & 0.04/0.63/0.63 & 0.04/0.63/0.63 & 0.07/0.40/0.41 & 0.06/0.41/0.41 & 0.06/0.41/0.41 \\
				\bottomrule
			\end{tabular}
		}
	\end{table}
	
	\begin{figure}[t]
		\centering
		\includegraphics[width=\textwidth]{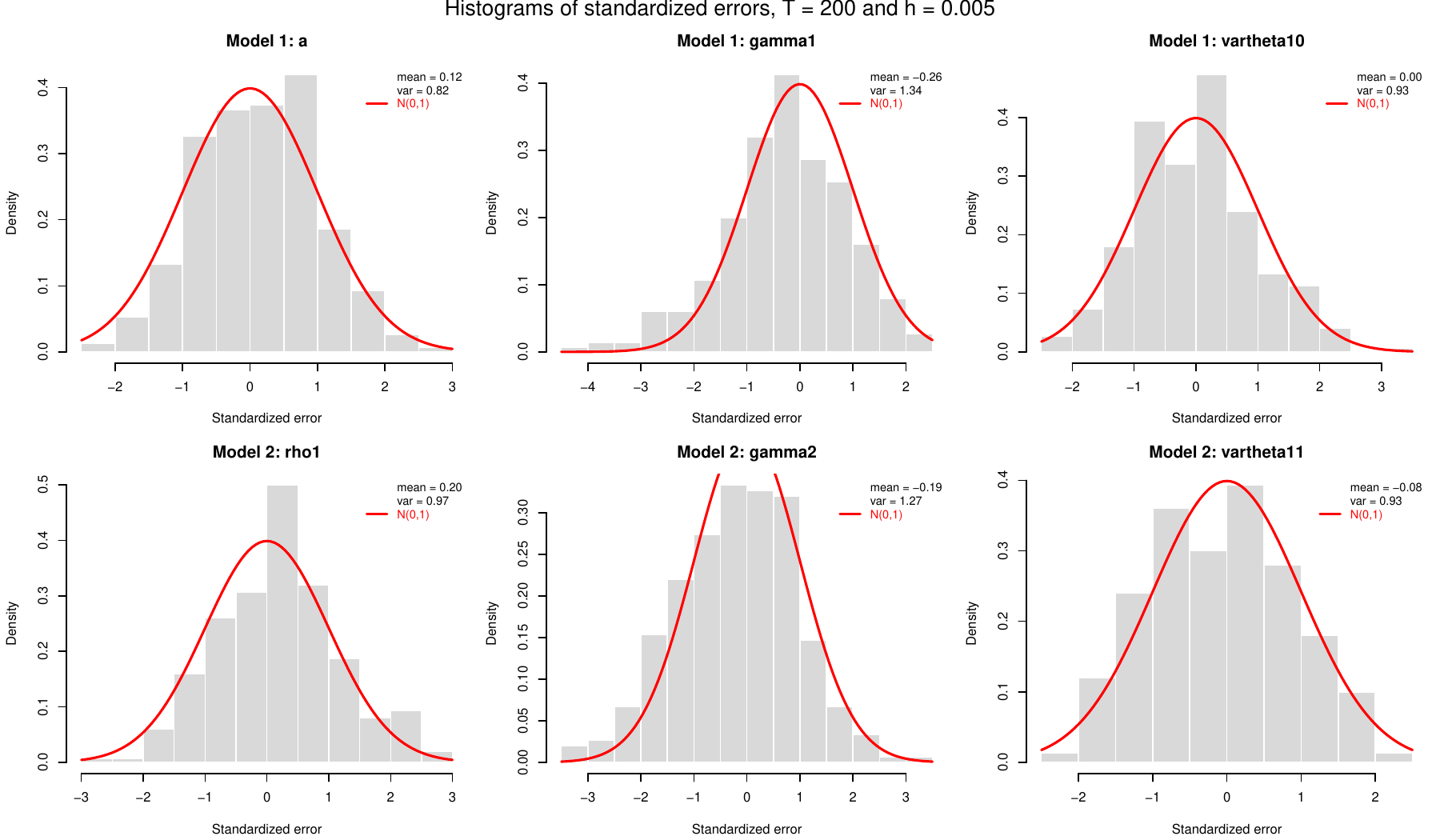}
		\caption{Histograms of selected standardized errors \(Z_\xi^{(r)}\) for representative components of \(\hat\zeta_n\) under \((T,h)=(200,0.005)\), together with the standard normal density.}
		\label{fig:num-hist}
	\end{figure}
	
	\begin{figure}[t]
		\centering
		\includegraphics[width=\textwidth]{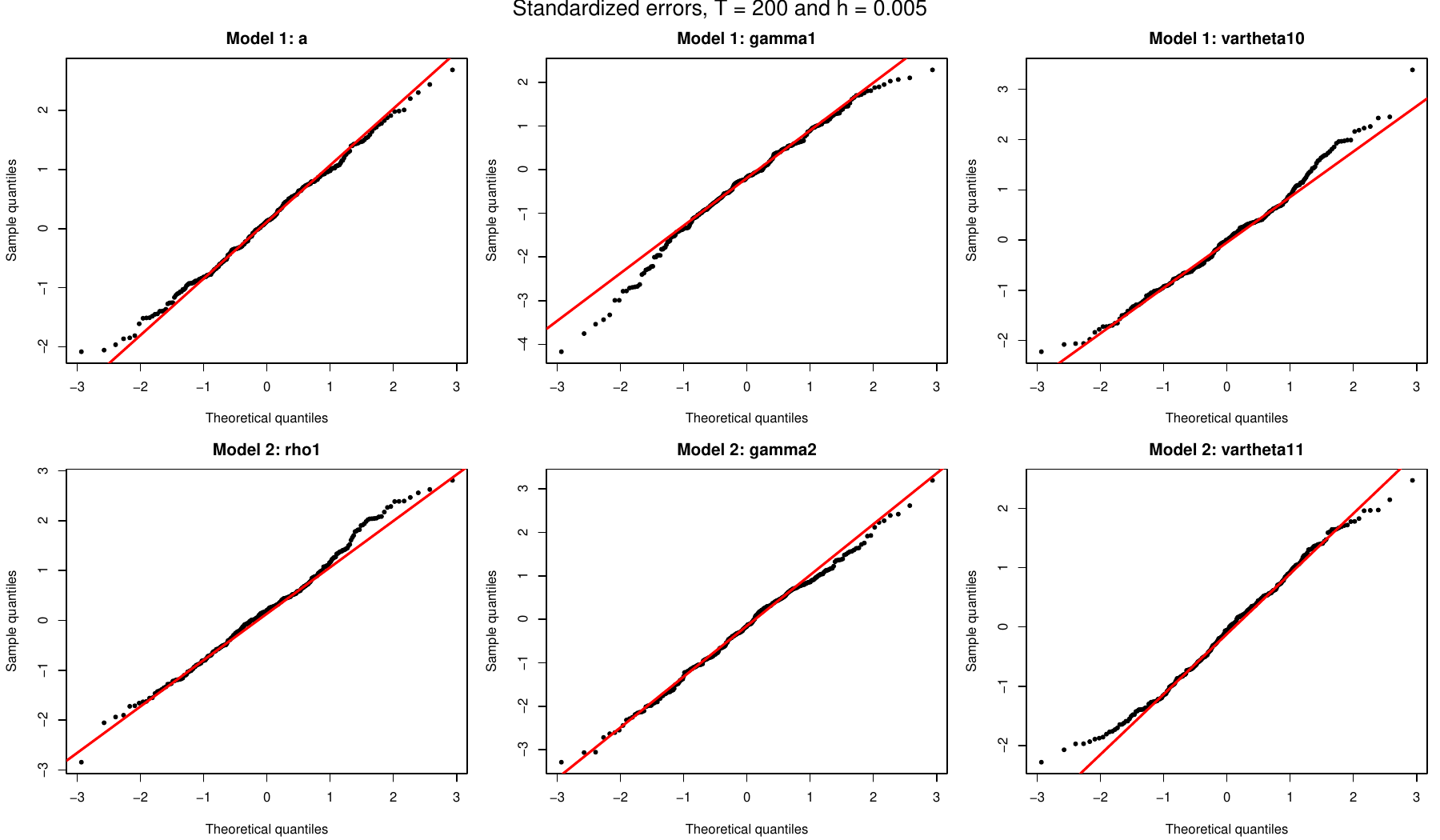}
		\caption{QQ plots of selected standardized errors \(Z_\xi^{(r)}\) for representative components of \(\hat\zeta_n\) under \((T,h)=(200,0.005)\).}
		\label{fig:num-qq}
	\end{figure}
	
	\begin{figure}[t]
		\centering
		\includegraphics[width=\textwidth]{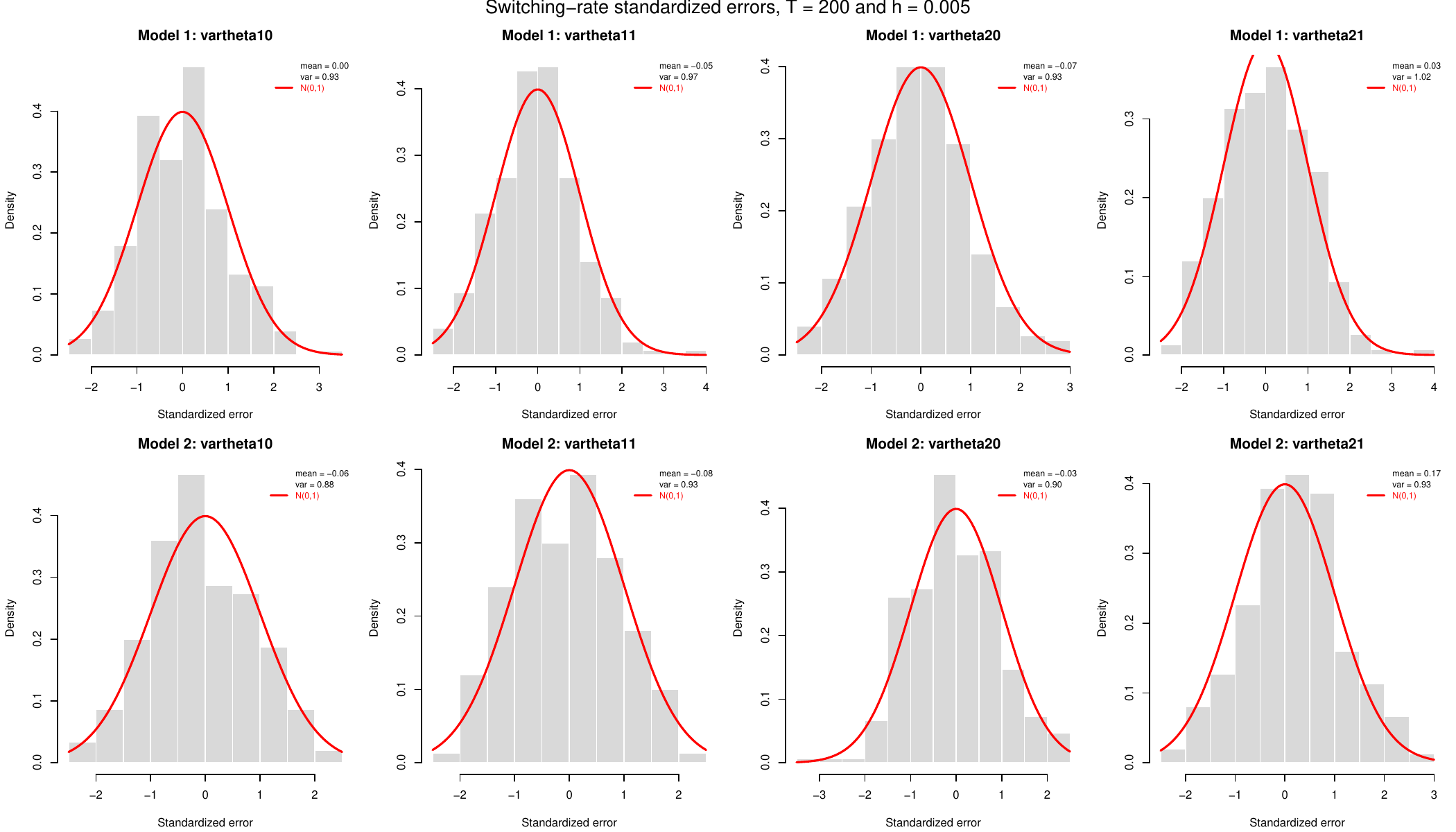}
		\caption{Histograms of standardized errors \(Z_\xi^{(r)}\) for all four components of the \(\vartheta\)-block of \(\hat\zeta_n\), shown for Models 1 and 2 under \((T,h)=(200,0.005)\), together with the standard normal density.}
		\label{fig:num-hist-vartheta}
	\end{figure}
	
	\begin{figure}[t]
		\centering
		\includegraphics[width=\textwidth]{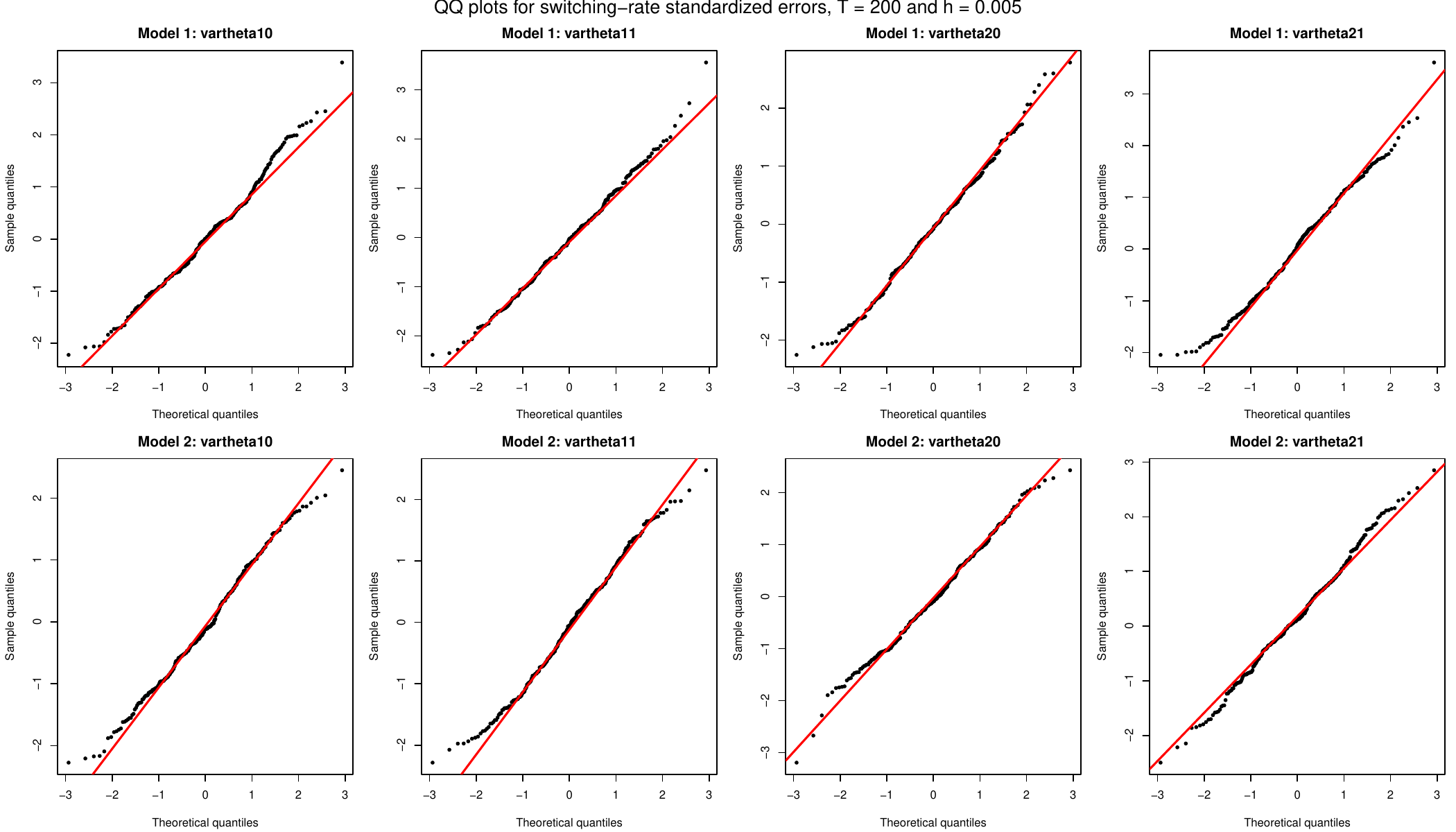}
		\caption{QQ plots of standardized errors \(Z_\xi^{(r)}\) for all four components of the \(\vartheta\)-block of \(\hat\zeta_n\), shown for Models 1 and 2 under \((T,h)=(200,0.005)\).}
		\label{fig:num-qq-vartheta}
	\end{figure}

	\section{Proofs}
	\label{sec:thmproof}

	\subsection{Proofs in Section~\ref{sec:exp-erg}}
	\label{sec:proof:Tchain}
	
	We use the following decomposition. By the small-jump condition in
	Assumption~\ref{ass:erg-coeff}, set
	$\nu^{\mathrm s}(dz):=\kappa_0\mathbf 1_{(-r_0,r_0)}(z)\,dz$,
	and
	$\nu^{\mathrm r}:=\nu-\nu^{\mathrm s}$.
	Then \(\nu^{\mathrm r}\) is a L\'evy measure, and the L\'evy--It\^o
	decomposition gives 
	$L=L^{\mathrm s}+L^{\mathrm r}$,
	where \(L^{\mathrm s}\) and \(L^{\mathrm r}\) are independent pure-jump
	L\'evy processes with jump measures \(\nu^{\mathrm s}\) and
	\(\nu^{\mathrm r}\), respectively. Since
	\[
	\rho:=\nu^{\mathrm s}(\mathbb R)=2\kappa_0r_0,
	\qquad
	\mu_{\mathrm s}(dz)
	:=
	\rho^{-1}\nu^{\mathrm s}(dz)
	=
	\frac1{2r_0}\mathbf 1_{(-r_0,r_0)}(z)\,dz,
	\]
	the process \(L^{\mathrm s}\) is compound Poisson
	$L_t^{\mathrm s}
	=
	\sum_{k=1}^{N_t^{\mathrm s}}Y_k$,
	where \(N^{\mathrm s}\) has rate \(\rho\), and \(Y_k\) are i.i.d. with law
	\(\mu_{\mathrm s}\), independent of \(N^{\mathrm s}\) and \(L^{\mathrm r}\).

	\subsubsection{Proof of Proposition~\ref{prop:fixed-h-skeleton}}
	The proof of Proposition~\ref{prop:fixed-h-skeleton} is based on two
	auxiliary estimates: a local minorization in a fixed regime and a one-step
	regime-switching estimate. The details are given in Appendix~\ref{app:erg}.
	
	Recall that $\Phi^{(h)}$ is a T-chain if there exists a substochastic
	kernel $T$ satisfying $T(\cdot,B)\le P_h(\cdot,B)$ for every Borel $B$,
	$\zeta\mapsto T(\zeta,B)$ is lower semicontinuous for every Borel $B$, and
	$T(\zeta,\mathbb{R}\times\mathbb{S})>0$ for every $\zeta$; and is $\varphi$-irreducible if
	\[
	\varphi(B)>0
	\implies
	\sum_{N\ge 1}P_h^N(\zeta,B)>0
	\qquad\text{for every }\zeta\in\mathbb{R}\times\mathbb{S}
	\]
	(see~\cite{meyn2012markov}).
	
	Fix a deterministic number $h_0>0$.
	
	\medskip
	\noindent\textbf{Step 1: T-chain property.}
	
	Fix $z=(x,i)\in\mathbb R\times\mathbb S$ and set $K_z:=[x-1,x+1]$.
	Apply Lemma~\ref{lem:fixed-h-local-minorization} to $K_z$ and regime $i$. Then there exists $h_{K_z}^i>0$ such that, for every fixed $h\in(0,h_{K_z}^i]$, the conclusion of Lemma~\ref{lem:fixed-h-local-minorization} holds.
	Choose $L_z\in\mathbb N$ so large that
	\[
	\Delta_z:=\frac{h_0}{L_z}\le h_{K_z}^i.
	\]
	For this fixed $\Delta_z$,
	Lemma~\ref{lem:fixed-h-local-minorization}
	gives constants $\delta_{K_z,\Delta_z}^i>0$, $\varepsilon_{K_z,\Delta_z}^i>0$,
	and, for the point $x\in K_z$, an open neighborhood $U_{x,\Delta_z}\ni x$
	such that \eqref{eq:fixed-h-local-minorization} holds.
	Choose an open interval $G_z \ni x$ such that $G_z \subset U_{x,\Delta_z}\cap(x-\delta_{K_z,\Delta_z}^i,x+\delta_{K_z,\Delta_z}^i)$.
	Then \eqref{eq:fixed-h-local-minorization} gives
	\[
	P_{\Delta_z}\bigl((u,i),A\times\{i\}\bigr)
	\ge
	\varepsilon_{K_z,\Delta_z}^i\,\lambda(A\cap G_z),
	\qquad u\in G_z,\ A\in\mathcal B(\mathbb R).
	\]
	Iterating this estimate $L_z$ times yields
	\[
	P_{h_0}\bigl((u,i),A\times\{i\}\bigr)
	=
	P_{\Delta_z}^{L_z}\bigl((u,i),A\times\{i\}\bigr)
	\ge
	a_z\,\lambda(A\cap G_z),
	\qquad u\in G_z,
	\]
	where $a_z:=(\varepsilon_{K_z,\Delta_z}^i)^{L_z}\lambda(G_z)^{L_z-1}>0$.
	Therefore we have, for $\zeta=(u,k)\in\mathbb R\times\mathbb S$,
	\[
	P_{h_0}(\zeta,B)
	\ge
	a_z\,\mathbf 1_{G_z \times\{i\}}(\zeta)\,
	\lambda(B_i\cap G_z),
	\]
	where $B_i:=\{y\in\mathbb R:(y,i)\in B\}$ for $B\in\mathcal{B}(\R)\otimes 2^\mbbs$.
	
	The family $\{G_z\times\{i\}:z=(x,i)\in\mathbb R\times\mathbb S\}$ is an
	open cover of $\mathbb R\times\mathbb S$. Since this space is second countable,
	choose a countable subcover $\{G_n\times\{i_n\}\}_{n\ge1}$.
	Let $a_n>0$ be the constant corresponding to $G_n\times\{i_n\}$. Define
	\[
	T(\zeta,B)
	:=
	\sum_{n=1}^\infty
	2^{-n}a_n
	\mathbf 1_{G_n\times\{i_n\}}(\zeta)
	\lambda(B_{i_n}\cap G_n).
	\]
	Then $T$ is a substochastic kernel. Moreover,
	\[
	T(\zeta,B)\le P_{h_0}(\zeta,B),
	\qquad
	\zeta\in\mathbb R\times\mathbb S.
	\]
	For each Borel $B$, the map $\zeta\mapsto T(\zeta,B)$ is lower
	semicontinuous because each $G_n\times\{i_n\}$ is open. Finally,
	$T(\zeta,\mathbb R\times\mathbb S)>0$ for every $\zeta$, since the
	sets $G_n\times\{i_n\}$ cover the state space. Hence the $h_0$-skeleton is a
	T-chain.

	\medskip
	\noindent\textbf{Step 2: $\varphi$-irreducibility.}
	
	Let $B\in\mathcal B(\mathbb R\times\mathbb S)$ satisfy $\varphi(B)>0$.
	Then there exist $j\in\mathbb S$ and $n\ge1$ such that $\lambda(B_j\cap I_n)>0$ where $B_j:=\{y\in\mathbb R:(y,j)\in B\}$.
	Choose a Lebesgue density point $y^*\in B_j\cap I_n$.
	
	Fix a point $z_0=(x_0,i_0)\in\mathbb R\times\mathbb S$.
	Fix a number $\eta_0>0$, and choose a compact interval
	$K\subset\mathbb R$ such that
	\[
	[x_0-\eta_0,x_0+\eta_0]\cup\{y^*\}
	\subset \operatorname{int}K.
	\]
	Apply Lemma~\ref{lem:fixed-h-local-minorization} to the compact interval $K$
	and the regime $j$. Then there exists $h_K^j>0$ such that, for every fixed
	$h\in(0,h_K^j]$, the conclusion of
	Lemma~\ref{lem:fixed-h-local-minorization} holds.
	
	If $i_0\neq j$, apply Lemma~\ref{lem:fixed-h-switch} to the compact interval
	$K$, the pair $(i_0,j)$, and the radius $\eta_0$. This gives a constant
	$h_{K,\eta_0}^{i_0j}>0$ such that, for every fixed
	$h\in(0,h_{K,\eta_0}^{i_0j}]$, the switching estimate \eqref{eq:fixed-h-switch} holds.
	
	Choose a large $L\in\mathbb N$ such that
	\[
	\Delta:=\frac{h_0}{L}\le h_K^j,
	\]
	and, if $i_0\neq j$, $\Delta\le h_{K,\eta_0}^{i_0j}$.
	
	For this fixed value of $\Delta$, Lemma~\ref{lem:fixed-h-local-minorization}
	gives constants $\delta_{K,\Delta}^j>0$, $\varepsilon_{K,\Delta}^j>0$,
	and, for every $x\in K$, an open neighborhood $U_{x,\Delta}\ni x$ such that
	\[
	P_\Delta\bigl((u,j),A\times\{j\}\bigr)
	\ge
	\varepsilon_{K,\Delta}^j\,
	\lambda\bigl(A\cap(x-\delta_{K,\Delta}^j,x+\delta_{K,\Delta}^j)\bigr),
	\qquad
	u\in U_{x,\Delta},
	\quad
	A\in\mathcal B(\mathbb R).
	\]
	
	We next prove the following auxiliary claim:
	Let $S\subset K$ be compact. For every residue class
	$r\in\{0,1,\dots,L-1\}$, there exist $M\in\mathbb N$ and $c>0$ such that $M\equiv r \pmod L$
	and
	\[
	P_\Delta^M\bigl((y,j),B\bigr)\ge c,
	\qquad y\in S.
	\]
	
	\smallskip
	\noindent
	To prove the claim, first choose, for every $a\in S$, a finite sequence
	\[
	p_0^a=a,\ p_1^a,\dots,p_{m_a}^a=y^*
	\]
	inside $\operatorname{int}K$, with $m_a\ge1$, such that
	\[
	|p_{\ell+1}^a-p_\ell^a|<\frac{\delta_{K,\Delta}^j}{4},
	\qquad \ell=0,\dots,m_a-1.
	\]
	The open sets $\{U_{a,\Delta}:a\in S\}$ cover $S$. Since $S$ is compact,
	choose finitely many points $a_1,\dots,a_R\in S$ such that
	\[
	S\subset\bigcup_{\rho=1}^R U_{a_\rho,\Delta}.
	\]
	For each $\rho$, write $p_\ell^\rho:=p_\ell^{a_\rho}$, $m_\rho:=m_{a_\rho}$.
	
	Choose an open interval $H_*$ containing $y^*$ such that
	\[
	H_*\subset I_n,
	\]
	\[
	H_*\subset U_{y^*,\Delta}
	\cap (y^*-\delta_{K,\Delta}^j,y^*+\delta_{K,\Delta}^j)\cap\operatorname{int}K
	,
	\]
	and
	\[
	H_*\subset
	(p_{m_\rho-1}^\rho-\delta_{K,\Delta}^j,p_{m_\rho-1}^\rho+\delta_{K,\Delta}^j),
	\qquad \rho=1,\dots,R.
	\]
	Since $y^*$ is a density point of $B_j\cap I_n$, $ \lambda(B_j\cap H_*)>0$.
	
	For each $\rho=1,\dots,R$ and each $\ell=1,\dots,m_\rho-1$, choose an open
	interval $H_\ell^\rho$ containing $p_\ell^\rho$ such that
	\[
	H_\ell^\rho
	\subset
	U_{p_\ell^\rho,\Delta}
	\cap
	(p_{\ell-1}^\rho-\delta_{K,\Delta}^j,p_{\ell-1}^\rho+\delta_{K,\Delta}^j)\cap\operatorname{int}K.
	\]
	Then repeated use of Lemma~\ref{lem:fixed-h-local-minorization} gives, for
	every $y\in U_{a_\rho,\Delta}$, the Chapman--Kolmogorov equation gives
	\[
	\begin{aligned}
		& P_\Delta^{m_\rho}
		\bigl((y,j),H_*\times\{j\}\bigr)
		\\
		&\quad \ge
		\int_{H_1^\rho}\cdots\int_{H_{m_\rho-1}^\rho}
		P_\Delta\bigl((v_{m_\rho-1},j),H_*\times\{j\}\bigr)
		\prod_{\ell=1}^{m_\rho-1}
		P_\Delta\bigl((v_{\ell-1},j),dv_\ell\times\{j\}\bigr),
		\\
		&\quad \geq c_\rho^*,
	\end{aligned}
	\]
	where
	\[
	c_\rho^*
	:=
	(\varepsilon_{K,\Delta}^j)^{m_\rho}
	\left(\prod_{\ell=1}^{m_\rho-1}\lambda(H_\ell^\rho)\right)
	\lambda(H_*)
	>0,
	\]
	with the empty product interpreted as one.
	
	Now choose $M\in\mathbb N$ such that
	\[
	M\equiv r \pmod L
	\quad\text{and}\quad
	M> \max_{\rho=1,\dots,R}m_\rho.
	\]
	For each $\rho$, set
	\[
	q_\rho:=M-m_\rho-1\ge0.
	\]
	Since $H_*\subset U_{y^*,\Delta}\cap (y^*-\delta_{K,\Delta}^j,y^*+\delta_{K,\Delta}^j)$,
	Lemma~\ref{lem:fixed-h-local-minorization} gives
	\[
	P_\Delta\bigl((u,j),H_*\times\{j\}\bigr)
	\ge
	\varepsilon_{K,\Delta}^j\lambda(H_*),
	\qquad u\in H_*,
	\]
	and
	\[
	P_\Delta\bigl((u,j),B\bigr)
	\ge
	\varepsilon_{K,\Delta}^j\lambda(B_j\cap H_*),
	\qquad u\in H_*.
	\]
	Therefore, for every $y\in U_{a_\rho,\Delta}$,
	\[
	\begin{aligned}
		P_\Delta^M\bigl((y,j),B\bigr)
		&=
		P_\Delta^{m_\rho+q_\rho+1}\bigl((y,j),B\bigr)        \\
		&\ge
		\int_{H_*}
		P_\Delta^{m_\rho}\bigl((y,j),du\times\{j\}\bigr)
		P_\Delta^{q_\rho+1}\bigl((u,j),B\bigr)            \\
		&\ge
		c_\rho^*
		\bigl(\varepsilon_{K,\Delta}^j\lambda(H_*)\bigr)^{q_\rho}
		\varepsilon_{K,\Delta}^j\lambda(B_j\cap H_*)                    \\
		&>0.
	\end{aligned}
	\]
	Taking the minimum over $\rho=1,\dots,R$ gives a constant $c>0$ satisfying
	\[
	P_\Delta^M\bigl((y,j),B\bigr)\ge c,
	\qquad y\in S.
	\]
	This proves the claim.

	We now apply the claim.
	First suppose $i_0=j$. Set $S:=\{x_0\}$.
	Apply the claim with residue class $r=0$. Then there exist $M = NL\in\mathbb N$ for $N \in \mathbb N$
	and $c>0$ such that $P_\Delta^M(z_0,B)\ge c>0$.
	Since $\Delta=h_0/L$, we have
	\[
	P_\Delta^M=P_\Delta^{NL}=P_{Nh_0}=P_{h_0}^N.
	\]
	Therefore $P_{h_0}^N(z_0,B)>0$.
	
	Next suppose $i_0\neq j$. Set $S:=[x_0-\eta_0,x_0+\eta_0]$.
	Apply the claim with residue class $r=L-1$. Then there exist
	$M\in\mathbb N$ and $c>0$ such that $M\equiv L-1\pmod L$
	and $P_\Delta^M\bigl((y,j),B\bigr)\ge c$, for $y\in S$.
	
	Since $x_0\in K$, Lemma~\ref{lem:fixed-h-switch} gives a constant $\xi_{K,\eta_0,\Delta}^{i_0j}>0$
	such that
	\[
	P_\Delta\bigl((x_0,i_0),(x-\eta_0,x+\eta_0)\times\{j\}\bigr)
	\ge \xi_{K,\eta_0,\Delta}^{i_0j},
	\qquad x\in K.
	\]
	Hence, by the Chapman--Kolmogorov equation,
	\[
	\begin{aligned}
		P_\Delta^{M+1}(z_0,B)
		&\ge
		\int_{(x_0-\eta_0,x_0+\eta_0)}
		P_\Delta\bigl(z_0,dy\times\{j\}\bigr)
		P_\Delta^M\bigl((y,j),B\bigr)  \\
		&\ge
		c\,
		P_\Delta\bigl((x_0,i_0),(x_0-\eta_0,x_0+\eta_0)\times\{j\}\bigr)\\
		&\ge c\, \xi_{K,\eta_0,\Delta}^{i_0j}>0.
	\end{aligned}
	\]
	Since $M\equiv L-1\pmod L$, there exists $N\in\mathbb N$ such that $M+1=NL$.
	Therefore
	\[
	P_\Delta^{M+1}=P_\Delta^{NL}=P_{Nh_0}=P_{h_0}^N,
	\]
	and consequently $P_{h_0}^N(z_0,B)>0$.
	
	In both cases, for the arbitrary starting point $z_0$ and every
	$B$ satisfying $\varphi(B)>0$, there exists $N\in\mathbb N$ such that
	\[
	P_{h_0}^N(z_0,B)>0.
	\]
	Thus the $h_0$-skeleton chain is $\varphi$-irreducible.

	\subsubsection{Proof of Theorem~\ref{thm:hybrid-exp-erg}}
	
	Fix \(r>0\) and set
	$\langle x\rangle:=(1+x^2)^{1/2}$,
	then 
	$V_r(x,i)=\langle x\rangle^r$.
	Then \(V_r\ge1\), \(V_r\) is norm-like, and
	\(V_r(x,i)\asymp 1+|x|^r\). Moreover,
	\[
	\partial_x V_r(x)
	=
	rx\langle x\rangle^{r-2},
	\qquad
	\partial_x^2 V_r(x)
	=
	r\langle x\rangle^{r-2}
	+
	r(r-2)x^2\langle x\rangle^{r-4}.
	\]
	Hence, with \(a_+:=\max(a,0)\),
	\begin{equation}
		\label{eq:Vr-derivative-bounds}
		|\partial_x V_r(x)|
		\le
		C_r\langle x\rangle^{(r-1)_+},
		\qquad
		|\partial_x^2 V_r(x)|
		\le
		C_r\langle x\rangle^{(r-2)_+}.
	\end{equation}
	
	We first check that \(V_r\in\mathcal D(\mathcal A_0)\). By Taylor's
	formula, Assumption~\ref{ass:erg-coeff}\textnormal{(E4)}, and
	\eqref{eq:Vr-derivative-bounds},
	\[
	\int_{\{|z|\le1\}}
	\left|
	V_r(x+c(x,i,\gamma_0)z,i)-V_r(x,i)
	-\partial_xV_r(x)c(x,i,\gamma_0)z
	\right|
	\nu(dz)
	\le
	C_r\{1+\langle x\rangle^{(r-2)_+}\}.
	\]
	For the large-jump part, if \(r\ge1\), then
	\[
	|V_r(x+y)-V_r(x)|
	\le
	C_r\{1+\langle x\rangle^{r-1}+|y|^{r-1}\}|y|,
	\]
	and Assumption~\ref{ass:Levy}, together with the boundedness of
	\(c(\cdot,\cdot,\gamma_0)\), gives
	\[
	\int_{\{|z|>1\}}
	|V_r(x+c(x,i,\gamma_0)z,i)-V_r(x,i)|\,\nu(dz)
	\le
	C_r\{1+\langle x\rangle^{r-1}\}.
	\]
	If \(0<r<1\), then
	$|\langle x+y\rangle^r-\langle x\rangle^r|
	\le |y|^r$,
	and hence the same integral is finite by Assumption~\ref{ass:Levy}. Thus
	\(V_r\in\mathcal D(\mathcal A_0)\).
	
	Since \(V_r\) does not depend on the regime,
	\[
	\sum_{j\neq i}q_{ij}(x,\vartheta_0)
	\{V_r(x,j)-V_r(x,i)\}=0 .
	\]
	By Assumption~\ref{ass:erg-coeff}\textnormal{(E5)},
	\[
	b(x,i,\alpha_0)\partial_xV_r(x)
	=
	r\langle x\rangle^{r-2}x\,b(x,i,\alpha_0)
	\le
	-r\lambda_0x^2\langle x\rangle^{r-2}
	+
	rK_0\langle x\rangle^{r-2}.
	\]
	Since
	$x^2\langle x\rangle^{r-2}
	=
	\langle x\rangle^r-\langle x\rangle^{r-2}$,
	we obtain
	\[
	b(x,i,\alpha_0)\partial_xV_r(x)
	\le
	-r\lambda_0V_r(x,i)
	+
	C_r\{1+\langle x\rangle^{(r-2)_+}\}.
	\]
	Combining this estimate with the small- and large-jump bounds yields
	\[
	\mathcal A_0V_r(x,i)
	\le
	-r\lambda_0V_r(x,i)
	+
	C_r\{1+\langle x\rangle^{(r-1)_+}\}.
	\]
	Hence, for suitable constants \(a_r,b_r>0\),
	\begin{equation}
		\label{eq:Vr-drift-final}
		\mathcal A_0V_r(x,i)
		\le
		-a_rV_r(x,i)+b_r,
		\qquad (x,i)\in E .
	\end{equation}
	Since \(V_r\in\mathcal D(\mathcal A_0)\), we also have
	\(V_r\in\mathcal D(\mathcal A_0^{\mathrm{ext}})\) and 
	$\mathcal A_0^{\mathrm{ext}}V_r
	=
	\mathcal A_0V_r$.
	
	By Proposition~\ref{prop:fixed-h-skeleton}, there exists \(h_0>0\) such
	that the \(h_0\)-skeleton chain is a \(\varphi\)-irreducible T-chain.
	Let
	\[
	C_r:=\{z\in E: V_r(z)\le 2b_r/a_r\}.
	\]
	Since \(V_r\) is norm-like, \(C_r\) is compact, hence every compact subset
	of $\R\times\mbbs$ is petite for the $h_0$-skeleton chain.
	Therefore \cite[Theorem~6.1]{meyn1993stability}
	gives a unique invariant probability measure \(\pi_0\), with
	\(\pi_0(V_r)<\infty\), and constants \(B_r<\infty\), \(\rho_r>0\) such that
	\[
	\sup_{|f|\le V_r}
	\left|
	P_tf(x,i)-\pi_0(f)
	\right|
	\le
	B_rV_r(x,i)e^{-\rho_rt},
	\qquad
	t\ge0,\quad (x,i)\in E .
	\]
	This proves the theorem.

	\subsection{Proofs in Section~\ref{sec:estimation}}
	
	\subsubsection{Proof of Theorem \ref{thm:full-consistency}}
	\label{sec:proof:consis}

	We first establish $(\hat\al_n,\hat\gam_n)\cip(\al_0,\gam_0)$.
	
	Define the auxiliary scale contrast and its remainder by
	\[
	\mathcal Y_{n,\gam}(\gam)
	:=\frac1n\sum_{j=1}^n
	\left\{
	\log\frac{c_{j-1}(\gam)^2}{c_{j-1}(\gam_0)^2}
	+\rho_{j-1}(\gam)
	\right\},
	\qquad
	r_{n,\gam}(\gam)
	:=-\frac1{2n}\sum_{j=1}^n
	\left(
	\frac{\eta_j^2}{h_n c_{j-1}(\gam_0)^2}-1
	\right)\rho_{j-1}(\gam),
	\]
	and the auxiliary drift contrast and its remainder by
	\[
	\mathcal Y_{n,\al}(\al;\gam)
	:=\frac1n\sum_{j=1}^n
	\frac{\{b_{j-1}(\al)-b_{j-1}(\al_0)\}^2}{c_{j-1}(\gam)^2},
	\qquad
	r_{n,\al}(\al,\gam)
	:=\frac1{T_n}\sum_{j=1}^n
	\frac{\{b_{j-1}(\al)-b_{j-1}(\al_0)\}\,\eta_j}{c_{j-1}(\gam)^2},
	\]
	and write $\mathcal Y_{n,\al}(\al):=\mathcal Y_{n,\al}(\al;\gam_0)$.
	
	It suffices to show that, for every $\Theta_\gam$-valued sequence $\bar\gam_n$ with $\bar\gam_n\cip\gam_0$,
	\begin{equation}\label{eq:unif-gamma}
		\sup_{\gam\in\Theta_\gam}
		\bigl|\{\mathbb G_{1,n}(\gam)-\mathbb G_{1,n}(\gam_0)\}-\Y_{\gam}(\gam)\bigr|\cip0,
	\end{equation}
	\begin{equation}\label{eq:unif-alpha}
		\sup_{\al\in\Theta_\al}
		\bigl|\{\mathbb G_{2,n}(\al;\bar\gam_n)-\mathbb G_{2,n}(\al_0;\bar\gam_n)\}-\Y_{\al}(\al)\bigr|\cip0.
	\end{equation}
	
	\smallskip
	\noindent\textit{Step 1: Uniform convergence of the scale block.}
	A direct expansion together with Proposition \ref{prop:hybrid-one-step} yields
	\begin{align*}
		\sup_{\gam\in\Theta_\gam}
		\bigl|\{\mathbb G_{1,n}(\gam)-\mathbb G_{1,n}(\gam_0)\}-\Y_{\gam}(\gam)\bigr|
		&\le \frac12\sup_{\gam\in\Theta_\gam}|\mathcal Y_{n,\gam}(\gam)+2\Y_{\gam}(\gam)|
		+\sup_{\gam\in\Theta_\gam}|r_{n,\gam}(\gam)|\\
		&\quad+\frac{C}{n}\sum_{j=1}^n(1+|X_{t_{j-1}}|^C)|\eta_j|
		+\frac{C h_n}{n}\sum_{j=1}^n(1+|X_{t_{j-1}}|^C)\\
		&= \frac12\sup_{\gam\in\Theta_\gam}|\mathcal Y_{n,\gam}(\gam)+2\Y_{\gam}(\gam)|
		+\sup_{\gam\in\Theta_\gam}|r_{n,\gam}(\gam)|+o_p(1).
	\end{align*}
	Lemma~\ref{lem:ergodic-rate-unif} gives
	$\sup_{\gam\in\Theta_\gam}|\mathcal Y_{n,\gam}(\gam)+2\Y_{\gam}(\gam)|\cip0$,
	so it remains to control the remainder.
	
	Fix $q>p_\gam$. For $|\beta|\le 1$, decompose
	\[
	\partial_\gam^\beta r_{n,\gam}(\gam)
	=M_{n,\gam}^{(\beta)}(\gam)+B_{n,\gam}^{(\beta)}(\gam),
	\]
	with
	\[
	M_{n,\gam}^{(\beta)}(\gam)
	:=-\frac1{2n}\sum_{j=1}^n \partial_\gam^\beta\rho_{j-1}(\gam)\,\bar\xi_j,
	\qquad
	B_{n,\gam}^{(\beta)}(\gam)
	:=-\frac1{2n}\sum_{j=1}^n \partial_\gam^\beta\rho_{j-1}(\gam)\,\E_{j-1}[\xi_j],
	\]
	where $\bar\xi_j:=\xi_j-\E_{j-1}[\xi_j]$. By Proposition \ref{prop:hybrid-one-step},
	\[
	|B_{n,\gam}^{(\beta)}(\gam)|
	\le C h_n^{1/2}\cdot\frac1n\sum_{j=1}^n(1+|X_{t_{j-1}}|^C),
	\qquad
	\sup_{\gam\in\Theta_\gam}\E\bigl[|B_{n,\gam}^{(\beta)}(\gam)|^q\bigr]\le C h_n^{q/2}.
	\]
	Combining the Burkholder--Davis--Gundy(BDG) inequality with the bound
	$\E_{j-1}[|\bar\xi_j|^2]\le\E_{j-1}[|\xi_j|^2]\le C h_n^{-1}(1+|X_{t_{j-1}}|^C)$
	from Proposition \ref{prop:hybrid-one-step},
	\[
	\E\bigl[|M_{n,\gam}^{(\beta)}(\gam)|^q\bigr]
	\le \frac{C}{n^q}\left\{(n h_n^{-1})^{q/2}+n h_n^{-q/2}\right\}
	\le C T_n^{-q/2}.
	\]
	Hence
	\begin{equation}\label{eq:pointwise-rn-gamma}
		\sup_{\gam\in\Theta_\gam}\E\bigl[|\partial_\gam^\beta r_{n,\gam}(\gam)|^q\bigr]
		\le C\bigl(T_n^{-q/2}+h_n^{q/2}\bigr)\to0
		\qquad(|\beta|\le 1).
	\end{equation}
	Since $q>p_\gam$, the Sobolev embedding
	$W^{1,q}(\Theta_\gam)\hookrightarrow C(\Theta_\gam)$
	(see, e.g., \cite[Section 10.2]{friedman2006stochastic}) gives
	\[
	\E\Bigl[\sup_{\gam\in\Theta_\gam}|r_{n,\gam}(\gam)|^q\Bigr]
	\le C\sum_{|\beta|\le 1}\int_{\Theta_\gam}\E\bigl[|\partial_\gam^\beta r_{n,\gam}(\gam)|^q\bigr]\,d\gam
	\le C\bigl(T_n^{-q/2}+h_n^{q/2}\bigr)\to0,
	\]
	so that $\sup_{\gam\in\Theta_\gam}|r_{n,\gam}(\gam)|\cip0$. We shall invoke this Sobolev argument repeatedly in what follows without further comment. This proves \eqref{eq:unif-gamma}.
	
	\smallskip
	\noindent\textit{Step 2: Uniform convergence of the drift block.}
	A similar expansion gives
	\begin{align*}
		&\sup_{\al\in\Theta_\al}
		\bigl|\{\mathbb G_{2,n}(\al;\bar\gam_n)-\mathbb G_{2,n}(\al_0;\bar\gam_n)\}-\Y_{\al}(\al)\bigr|\\
		&\quad\le \frac12\sup_{\al\in\Theta_\al}|\mathcal Y_{n,\al}(\al;\bar\gam_n)-\mathcal Y_{n,\al}(\al)|
		+\frac12\sup_{\al\in\Theta_\al}|\mathcal Y_{n,\al}(\al)+2\Y_{\al}(\al)|
		+\sup_{\theta\in\Theta}|r_{n,\al}(\al,\gam)|\\
		&\quad\le \frac12\sup_{\al\in\Theta_\al}|\mathcal Y_{n,\al}(\al)+2\Y_{\al}(\al)|
		+\sup_{\theta\in\Theta}|r_{n,\al}(\al,\gam)|
		+|\bar\gam_n-\gam_0|\cdot\frac{C}{n}\sum_{j=1}^n(1+|X_{t_{j-1}}|^C)\\
		&\quad=\frac12\sup_{\al\in\Theta_\al}|\mathcal Y_{n,\al}(\al)+2\Y_{\al}(\al)|
		+\sup_{\theta\in\Theta}|r_{n,\al}(\al,\gam)|+o_p(1).
	\end{align*}
	Lemma~\ref{lem:ergodic-rate-unif} ensures
	$\sup_{\al\in\Theta_\al}|\mathcal Y_{n,\al}(\al)+2\Y_{\al}(\al)|\cip0$, and the
	Sobolev argument used in Step~1 gives
	$\sup_{\theta\in\Theta}|r_{n,\al}(\al,\gam)|\cip0$. This proves \eqref{eq:unif-alpha}.
	
	\smallskip
	\noindent\textit{Step 3: Argmax conclusion.}
	Since
	$\hat\gam_n\in\argmax_{\gam\in\Theta_\gam}\mathbb G_{1,n}(\gam)$ and
	$\hat\al_n\in\argmax_{\al\in\Theta_\al}
	\{\mathbb G_{2,n}(\al;\hat\gam_n)-\mathbb G_{2,n}(\al_0;\hat\gam_n)\}$,
	Assumption~\ref{ass:Y-global} together with the standard argmax theorem
	yields $\hat\gam_n\cip\gam_0$ and $\hat\al_n\cip\al_0$.
	
	\medskip
	We now turn to $\hat\vartheta_n\cip\vartheta_0$. It is convenient to work with the centered contrast
	\[
	\bar{\mathbb G}_{3,n}(\vartheta):=\mathbb G_{3,n}(\vartheta)-\mathbb G_{3,n}(\vartheta_0).
	\]
	Set $\chi_j^{ik}:=\Delta_j N_{ik}^{\mathrm o}-\E_{j-1}[\Delta_j N_{ik}^{\mathrm o}]$.
	Substituting the one-step expansion in Lemma~\ref{lem:q-obs-direct} produces the decomposition
	\[
	T_n\bar{\mathbb G}_{3,n}(\vartheta)
	=\mathbb M_n^Q(\vartheta)+\mathbb G_n^Q(\vartheta)+\mathbb R_n^Q(\vartheta),
	\]
	where
	\begin{align*}
		\mathbb M_n^Q(\vartheta)
		&:=\sum_{j=1}^n\sum_{i=1}^m\sum_{k\ne i}
		\chi_j^{ik}\log\frac{q_{ik}(X_{t_{j-1}},\vartheta)}{q_{ik}(X_{t_{j-1}},\vartheta_0)},\\
		\mathbb G_n^Q(\vartheta)
		&:=\sum_{j=1}^n h_n\,F_Q(X_{t_{j-1}},\Lambda_{t_{j-1}},\vartheta),\\
		\mathbb R_n^Q(\vartheta)
		&:=\sum_{j=1}^n\sum_{i=1}^m\sum_{k\ne i}
		r_{j,n}^{ik}\log\frac{q_{ik}(X_{t_{j-1}},\vartheta)}{q_{ik}(X_{t_{j-1}},\vartheta_0)}.
	\end{align*}
	We treat the three terms in turn.
	
	\emph{Remainder term.} Assumption~\ref{ass:coeff} and Proposition \ref{prop:hybrid-switching} give
	\[
	\sup_{\vartheta\in\Theta_\vartheta}\frac1{T_n}|\mathbb R_n^Q(\vartheta)|
	\le \frac{C}{T_n}\sum_{j=1}^n\sum_{i=1}^m\sum_{k\ne i}
	h_n^{3/2}(1+|X_{t_{j-1}}|^C)
	=O_p(\sqrt{h_n})\to0.
	\]
	
	\emph{Drift term.} Assumption~\ref{ass:coeff} together with Lemma~\ref{lem:ergodic-rate-unif} yields
	\[
	\sup_{\vartheta\in\Theta_\vartheta}
	\Bigl|\frac1{T_n}\mathbb G_n^Q(\vartheta)-\Y_Q(\vartheta)\Bigr|\cip0.
	\]
	
	\emph{Martingale term.} Since $\Delta_j N_{ik}^{\mathrm o}\in\{0,1\}$, Proposition \ref{prop:hybrid-switching} gives
	\[
	\E_{j-1}\bigl[(\chi_j^{ik})^2\bigr]
	\le\E_{j-1}\bigl[\Delta_j N_{ik}^{\mathrm o}\bigr]
	\le C h_n(1+|X_{t_{j-1}}|^C).
	\]
	The Burkholder--Davis--Gundy inequality therefore yields, for every fixed $\vartheta$ and every $K>0$,
	\[
	\sup_{n\in\mathbb N}\E\Bigl[\Bigl|\tfrac1{\sqrt{T_n}}\mathbb M_n^Q(\vartheta)\Bigr|^K\Bigr]<\infty,
	\]
	and the Sobolev embedding argument upgrades this to
	\[
	\sup_{\vartheta\in\Theta_\vartheta}\Bigl|\tfrac1{T_n}\mathbb M_n^Q(\vartheta)\Bigr|\cip0.
	\]
	
	Combining the three estimates,
	\[
	\sup_{\vartheta\in\Theta_\vartheta}|\bar{\mathbb G}_{3,n}(\vartheta)-\Y_Q(\vartheta)|\cip0.
	\]
	Since Assumption~\ref{ass:Y-global} guarantees that $\Y_Q(\vartheta)=0$ if and only if $\vartheta=\vartheta_0$, the argmax theorem yields $\hat\vartheta_n\cip\vartheta_0$.

	\subsubsection{Proof of Theorem \ref{thm:joint-an}}
	\label{sec:proof:joint-an}
	We prove the theorem by three steps.
	
	Write $\ell_{j-1}(\gam):=\log c_{j-1}(\gam)^2$,
	$\psi_{j-1}:=\p_\gam\ell_{j-1}(\gam_0)$,
	$H_{j-1}:=\p_\gam^2\ell_{j-1}(\gam_0)$,
	and abbreviate $b_{j-1}:=b_{j-1}(\al_0)$, $c_{j-1}:=c_{j-1}(\gam_0)$.
	
	\medskip
	\noindent
	\textbf{Step 1: Limit $\Sigma_\zeta$.}

	Let $(\bar\al_n,\bar\gam_n,\bar\vartheta_n)$ be any $\Theta$-valued sequence with
	$(\bar\al_n,\bar\gam_n,\bar\vartheta_n)\cip\zeta_0$. We claim that
	\begin{equation}\label{eq:hessian-claim}
		-\p_\gam^2\mathbb G_{1,n}(\bar\gam_n)\cip\Gamma_\gam,
		\qquad
		-\p_\al^2\mathbb G_{2,n}(\bar\al_n;\bar\gam_n)\cip\Gamma_\al,
		\qquad
		-\p_\vartheta^2\mathbb G_{3,n}(\bar\vartheta_n)\cip\Gamma_Q.
	\end{equation}
	
	\smallskip
	\noindent\textit{Step 1: Convergence of $-\p_\gam^2\mathbb G_{1,n}(\bar\gam_n)$.}
	Differentiating \eqref{eq:G1-stage} twice gives
	\begin{align}
		-\p_\gam^2\mathbb G_{1,n}(\gam_0)
		&=\frac1{2T_n}\sum_{j=1}^n
		\left[h_n H_{j-1}+\frac{(\Del_j X)^2}{c_{j-1}^2}\bigl(\psi_{j-1}^{\otimes 2}-H_{j-1}\bigr)\right]\notag\\
		&=\frac1{2n}\sum_{j=1}^n\psi_{j-1}^{\otimes 2}
		+\frac1{2T_n}\sum_{j=1}^n
		\left(\frac{(\Del_j X)^2}{c_{j-1}^2}-h_n\right)\bigl(\psi_{j-1}^{\otimes 2}-H_{j-1}\bigr).
		\label{eq:hessian-gamma-at-true}
	\end{align}
	The first term converges in probability to $\Gamma_\gam$ by Lemma~\ref{lem:ergodic-rate-unif}, so it suffices to show that the second term is $o_p(1)$.
	
	Set $U_{j-1}:=\psi_{j-1}^{\otimes 2}-H_{j-1}$ and write the second term in \eqref{eq:hessian-gamma-at-true} as $R_{n,1}+R_{n,2}+R_{n,3}$, where
	\begin{align*}
		R_{n,1}&:=\frac1{2n}\sum_{j=1}^n U_{j-1}\,\xi_j,
		&R_{n,2}&:=\frac1{T_n}\sum_{j=1}^n U_{j-1}\,\frac{h_n b_{j-1}(\al_0)\eta_j}{c_{j-1}^2},\\
		R_{n,3}&:=\frac1{2T_n}\sum_{j=1}^n U_{j-1}\,\frac{h_n^2 b_{j-1}(\al_0)^2}{c_{j-1}^2}.
	\end{align*}
	
	\emph{Bound for $R_{n,1}$.} Decompose $\xi_j=\tilde\xi_j+\E_{j-1}[\xi_j]$ with $\tilde\xi_j:=\xi_j-\E_{j-1}[\xi_j]$. Proposition \ref{prop:hybrid-one-step} yields
	$\bigl|\E_{j-1}[\xi_j]\bigr|\le C h_n^{1/2}(1+|X_{t_{j-1}}|^C)$, whence
	\[
	\left|\frac1{2n}\sum_{j=1}^n U_{j-1}\,\E_{j-1}[\xi_j]\right|
	\le C h_n^{1/2}\cdot\frac1n\sum_{j=1}^n(1+|X_{t_{j-1}}|^C)\cip 0.
	\]
	Since $\{U_{j-1}\tilde\xi_j\}$ is a martingale-difference array, the orthogonality of martingale differences together with Proposition \ref{prop:hybrid-one-step} gives
	\[
	\E\left[\left|\frac1{2n}\sum_{j=1}^n U_{j-1}\tilde\xi_j\right|^2\right]
	=\frac1{4n^2}\sum_{j=1}^n\E\bigl[\|U_{j-1}\|^2\,\E_{j-1}[\tilde\xi_j^2]\bigr]
	\le\frac{C}{n^2}\sum_{j=1}^n h_n^{-1}
	=\frac{C}{T_n}\to 0.
	\]
	Hence $R_{n,1}\cip 0$. We shall invoke this orthogonality estimate for martingale-difference arrays repeatedly without further comment.
	
	\emph{Bound for $R_{n,2}$.} By Proposition \ref{prop:hybrid-one-step},
	\[
	\E[|R_{n,2}|]
	\le \frac{C h_n}{T_n}\sum_{j=1}^n \E\bigl[(1+|X_{t_{j-1}}|^C)|\eta_j|\bigr]
	\le \frac{C h_n}{T_n}\sum_{j=1}^n h_n^{1/2}
	\le C h_n^{1/2}\to 0.
	\]
	
	\emph{Bound for $R_{n,3}$.} Directly,
	\[
	\E[|R_{n,3}|]
	\le C h_n\cdot\frac1n\sum_{j=1}^n\E\bigl[1+|X_{t_{j-1}}|^C\bigr]\to 0.
	\]
	
	Combining these three bounds yields $-\p_\gam^2\mathbb G_{1,n}(\gam_0)\cip\Gamma_\gam$.
	
	It remains to transfer this convergence from $\gam_0$ to the consistent sequence $\bar\gam_n$. Assumption~\ref{ass:coeff} gives
	\[
	\sup_{\gam\in\Theta_\gam}\bigl\|\p_\gam^3\mathbb G_{1,n}(\gam)\bigr\|
	\le \frac{C}{T_n}\sum_{j=1}^n(1+|X_{t_{j-1}}|^C)\bigl\{h_n+(\Del_j X)^2\bigr\},
	\]
	and the right-hand side is $O_p(1)$ by Lemma~\ref{lem:short-time-moment}. The mean-value theorem therefore yields
	\[
	\bigl\|{-\p_\gam^2}\mathbb G_{1,n}(\bar\gam_n)+\p_\gam^2\mathbb G_{1,n}(\gam_0)\bigr\|
	\le |\bar\gam_n-\gam_0|\,\sup_{\gam\in\Theta_\gam}\bigl\|\p_\gam^3\mathbb G_{1,n}(\gam)\bigr\|\cip 0,
	\]
	which proves the first convergence in \eqref{eq:hessian-claim}.
	
	\smallskip
	\noindent\textit{Step 2: Convergence of $-\p_\al^2\mathbb G_{2,n}(\bar\al_n;\bar\gam_n)$.}
	The argument is structurally identical to that of Step~1; we omit the details.
	
	\smallskip
	\noindent\textit{Step 3: Convergence of $-\p_\vartheta^2\mathbb G_{3,n}(\bar\vartheta_n)$.}
	Applying the decomposition of $T_n\bar{\mathbb G}_{3,n}(\vartheta)$ used in the consistency proof of Theorem~\ref{thm:full-consistency} to the second derivative and arguing term by term as above yields the claim.

	\medskip
	\noindent
	\textbf{Step 2: $\Delta_{n,\zeta}\cil N(0,\Sigma_\zeta)$.}

	Propositions \ref{prop:hybrid-one-step} and \ref{prop:hybrid-switching} together with a direct computation yield the decompositions
	\[
	\Del_{n,\gam}=M_{n,\gam}+R_{n,\gam},
	\qquad
	\Del_{n,\al}=M_{n,\al}+R_{n,\al},
	\qquad
	\Del_{n,Q}=M_{n,\vartheta}+R_{n,\vartheta},
	\]
	where
	\begin{align*}
		M_{n,\gam}&:=\frac1{2\sqrt{T_n}}\sum_{j=1}^n \p_\gam\ell_{j-1}(\gam_0)\bigl\{(\Del_j L)^2-h_n\bigr\},
		&
		R_{n,\gam}&:=\frac1{2\sqrt{T_n}}\sum_{j=1}^n \p_\gam\ell_{j-1}(\gam_0)\,\zeta_{j,n},\\[2pt]
		M_{n,\al}&:=\frac1{\sqrt{T_n}}\sum_{j=1}^n \frac{\p_\al b_{j-1}(\al_0)}{c_{j-1}(\gam_0)}\,\Del_j L,
		&
		R_{n,\al}&:=\frac1{\sqrt{T_n}}\sum_{j=1}^n G_\al^{(2)}(X_{t_{j-1}},\Lam_{t_{j-1}},\al_0,\gam_0)\,r_{j,n},\\[2pt]
		M_{n,\vartheta}&:=\frac1{\sqrt{T_n}}\sum_{j=1}^n\sum_{i=1}^m\sum_{k\neq i}
		\chi_j^{ik}\,\p_\vartheta\log q_{ik}(X_{t_{j-1}},\vartheta_0),
		&
		R_{n,\vartheta}&:=\frac1{\sqrt{T_n}}\sum_{j=1}^n\sum_{i=1}^m\sum_{k\neq i}
		r_{j,n}^{ik}\,\p_\vartheta\log q_{ik}(X_{t_{j-1}},\vartheta_0),
	\end{align*}
	and
	\[
	\zeta_{j,n}
	:=\frac{2 h_n b_{j-1}}{c_{j-1}}\Del_j L
	+\frac{2\Del_j L\, r_{j,n}}{c_{j-1}}
	+\frac{2 h_n b_{j-1} r_{j,n}}{c_{j-1}^2}
	+\frac{r_{j,n}^2}{c_{j-1}^2}
	+\frac{h_n^2 b_{j-1}^2}{c_{j-1}^2}.
	\]
	The proof proceeds in two steps: we first show that the three remainders are $o_p(1)$, and then verify the conditions of the martingale triangular array CLT for $(M_{n,\gam},M_{n,\al},M_{n,\vartheta})$.
	
	\smallskip
	\noindent\textit{Step 1: The remainders are negligible.}
	
	\emph{Bound for $R_{n,\gam}$.} Assumption~\ref{ass:coeff} and Proposition \ref{prop:hybrid-one-step} give
	\begin{align*}
		\E\bigl[|\p_\gam\ell_{j-1}(\gam_0)\,\zeta_{j,n}|\bigr]
		&\le C\,\E\Bigl[(1+|X_{t_{j-1}}|^q)
		\bigl(h_n|\Del_j L|+|\Del_j L|\,|r_{j,n}|+h_n|r_{j,n}|+r_{j,n}^2+h_n^2\bigr)\Bigr]
		\le C h_n^{3/2}.
	\end{align*}
	Hence $\E|R_{n,\gam}|\le \frac{C}{\sqrt{T_n}}\sum_{j=1}^n h_n^{3/2}=C\sqrt{n h_n^2}\to 0$, so $R_{n,\gam}\cip 0$.
	
	\emph{Bound for $R_{n,\al}$.} Set $\bar r_{j,n}:=r_{j,n}-\E_{j-1}[r_{j,n}]$ and split
	$R_{n,\al}=R_{n,\al}^{(m)}+R_{n,\al}^{(d)}$ with
	\[
	R_{n,\al}^{(m)}:=\frac1{\sqrt{T_n}}\sum_{j=1}^n G_\al^{(2)}(X_{t_{j-1}},\Lam_{t_{j-1}},\theta_0)\,\bar r_{j,n},
	\quad
	R_{n,\al}^{(d)}:=\frac1{\sqrt{T_n}}\sum_{j=1}^n G_\al^{(2)}(X_{t_{j-1}},\Lam_{t_{j-1}},\theta_0)\,\E_{j-1}[r_{j,n}].
	\]
	Since $\{G_\al^{(2)}(X_{t_{j-1}},\Lam_{t_{j-1}},\theta_0)\,\bar r_{j,n}\}$ is a martingale-difference array, the orthogonality estimate together with Proposition \ref{prop:hybrid-one-step} yields
	\[
	\E\bigl[\|R_{n,\al}^{(m)}\|^2\bigr]
	=\frac{1}{T_n}\sum_{j=1}^n \E\bigl[\|G_\al^{(2)}(X_{t_{j-1}},\Lam_{t_{j-1}},\theta_0)\bar r_{j,n}\|^2\bigr]
	\le \frac{C}{T_n}\sum_{j=1}^n h_n^2
	= C h_n\to 0.
	\]
	For $R_{n,\al}^{(d)}$, note that $\E_{j-1}[r_{j,n}]=\E_{j-1}[\eta_j]$, so Proposition \ref{prop:hybrid-one-step} gives $|\E_{j-1}[r_{j,n}]|\le C h_n^2(1+|X_{t_{j-1}}|^C)$ and therefore
	\[
	\E\bigl[\|R_{n,\al}^{(d)}\|\bigr]
	\le \frac{C}{\sqrt{T_n}}\sum_{j=1}^n h_n^2
	= C\sqrt{n h_n^3}\to 0.
	\]
	Combining the two bounds, $R_{n,\al}\cip 0$.
	
	\emph{Bound for $R_{n,\vartheta}$.} By Assumption~\ref{ass:coeff} and Proposition \ref{prop:hybrid-switching},
	\[
	\frac1{\sqrt{T_n}}\left|\sum_{j=1}^n\sum_{i=1}^m\sum_{k\neq i}
	r_{j,n}^{ik}\,\p_\vartheta\log q_{ik}(X_{t_{j-1}},\vartheta_0)\right|
	\le \frac{C}{\sqrt{T_n}}\sum_{j=1}^n h_n^{3/2}(1+|X_{t_{j-1}}|^C)=o_p(1).
	\]
	
	\smallskip
	\noindent\textit{Step 2: Triangular-array CLT for $(M_{n,\gam},M_{n,\al},M_{n,\vartheta})$.}
	Define the per-step martingale increments
	\[
	m^{\gam}_{j,n}:=\frac1{2\sqrt{T_n}}\p_\gam\ell_{j-1}(\gam_0)\bigl\{(\Del_j L)^2-h_n\bigr\},
	\qquad
	m^{\al}_{j,n}:=\frac1{\sqrt{T_n}}\frac{\p_\al b_{j-1}(\al_0)}{c_{j-1}(\gam_0)}\,\Del_j L,
	\]
	\[
	m^{\vartheta}_{j,n}:=\frac1{\sqrt{T_n}}\sum_{i=1}^m\sum_{k\neq i}\p_\vartheta\log q_{ik}(X_{t_{j-1}},\vartheta_0)\,\chi_j^{ik},
	\]
	so that
	$\E_{j-1}[m^{\gam}_{j,n}]=\E_{j-1}[m^{\al}_{j,n}]=\E_{j-1}[m^{\vartheta}_{j,n}]=0$. We verify the two standard hypotheses of the triangular array CLT: convergence of the conditional covariances and a Lyapunov-type fourth moment condition.
	
	\emph{Diagonal covariance blocks.} By Lemma~\ref{lem:ergodic-rate-unif},
	\begin{align*}
		\sum_{j=1}^n \E_{j-1}[(m^{\gam}_{j,n})^{\otimes 2}]
		&=\frac{\kappa_4 h_n+O(h_n^2)}{4 T_n}\sum_{j=1}^n (\p_\gam\ell_{j-1}(\gam_0))^{\otimes 2}
		=\frac{\kappa_4+O(h_n)}{4 n}\sum_{j=1}^n (\p_\gam\ell_{j-1}(\gam_0))^{\otimes 2}\cip \Sigma_\gam,\\
		\sum_{j=1}^n \E_{j-1}[(m^{\al}_{j,n})^{\otimes 2}]
		&=\frac1{T_n}\sum_{j=1}^n \frac{(\p_\al b_{j-1}(\al_0))^{\otimes 2}}{c_{j-1}(\gam_0)^2}\,\E[(\Del_j L)^2]
		=\frac1{n}\sum_{j=1}^n \frac{(\p_\al b_{j-1}(\al_0))^{\otimes 2}}{c_{j-1}(\gam_0)^2}\cip \Gamma_\al.
	\end{align*}
	For the $\vartheta$-block, since $\Del_j N_{ik}^{\mathrm o}\,\Del_j N_{i'k'}^{\mathrm o}=0$ whenever $(i,k)\neq(i',k')$,
	\[
	\E_{j-1}[\chi_j^{ik}\chi_j^{i'k'}]
	=\begin{cases}
		\E_{j-1}\bigl[(\chi_j^{ik})^2\bigr], & (i,k)=(i',k'),\\[1mm]
		-\E_{j-1}[\Del_j N_{ik}^{\mathrm o}]\,\E_{j-1}[\Del_j N_{i'k'}^{\mathrm o}], & (i,k)\neq(i',k').
	\end{cases}
	\]
	Proposition \ref{prop:hybrid-switching} and Assumption~\ref{ass:coeff} therefore give
	\begin{align*}
		\E_{j-1}\bigl[(m^{\vartheta}_{j,n})^{\otimes 2}\bigr]
		&=\frac1{T_n}\sum_{i=1}^m\sum_{k\neq i}
		\bigl(\p_\vartheta\log q_{ik}(X_{t_{j-1}},\vartheta_0)\bigr)^{\otimes 2}\,
		\E_{j-1}[\Del_j N_{ik}^{\mathrm o}]+\rho_{j,n}\\
		&=\frac{h_n}{T_n}\sum_{i=1}^m\sum_{k\neq i}\mathbf 1_{\{\Lam_{t_{j-1}}=i\}}
		\bigl(\p_\vartheta\log q_{ik}(X_{t_{j-1}},\vartheta_0)\bigr)^{\otimes 2}
		q_{ik}(X_{t_{j-1}},\vartheta_0)+\tilde\rho_{j,n},
	\end{align*}
	with $\sum_{j=1}^n|\tilde\rho_{j,n}|\le C(\sqrt{h_n}+h_n)\cdot\frac1n\sum_{j=1}^n(1+|X_{t_{j-1}}|^C)=o_p(1)$. Lemma~\ref{lem:ergodic-rate-unif} then yields
	\[
	\sum_{j=1}^n\E_{j-1}\bigl[(m^{\vartheta}_{j,n})^{\otimes 2}\bigr]\cip \Gamma_Q.
	\]
	
	\emph{Cross covariance blocks.} A direct computation gives
	\begin{align*}
		\sum_{j=1}^n\E_{j-1}[m^{\al}_{j,n}(m^{\gam}_{j,n})^\top]
		&=\frac1{2T_n}\sum_{j=1}^n A_\al(X_{t_{j-1}},\Lam_{t_{j-1}})\Psi_\gam(X_{t_{j-1}},\Lam_{t_{j-1}})^\top
		\E\bigl[\Del_j L\bigl\{(\Del_j L)^2-h_n\bigr\}\bigr]\\
		&=\frac{\kappa_3+O(h_n)}{2n}\sum_{j=1}^n A_\al(X_{t_{j-1}},\Lam_{t_{j-1}})\Psi_\gam(X_{t_{j-1}},\Lam_{t_{j-1}})^\top
		\cip \Sigma_{\al\gam},
	\end{align*}
	and the transpose block converges to $\Sigma_{\al\gam}^\top$. By Proposition \ref{prop:hybrid-switching},
	\[
	\sum_{j=1}^n\E_{j-1}[m^{\al}_{j,n}(m^{\vartheta}_{j,n})^\top]
	=O_p\!\left(\frac{n h_n^{3/2}}{T_n}\right)=o_p(1),
	\]
	and analogously $\sum_{j=1}^n\E_{j-1}[m^{\gam}_{j,n}(m^{\vartheta}_{j,n})^\top]\cip 0$.
	
	\emph{Lyapunov condition.} We have
	\begin{align*}
		\sum_{j=1}^n \E_{j-1}\bigl[|m^{\gam}_{j,n}|^4\bigr]
		&\le \frac{C}{T_n^2}\sum_{j=1}^n \E_{j-1}\bigl[(1+|X_{t_{j-1}}|^C)\bigl|(\Del_j L)^2-h_n\bigr|^4\bigr]
		\le \frac{1}{T_n^2}\,O_p(T_n)=o_p(1),\\
		\sum_{j=1}^n \E\bigl[|m^{\al}_{j,n}|^4\bigr]
		&\le \frac{C}{T_n^2}\sum_{j=1}^n \E\bigl[(1+|X_{t_{j-1}}|^C)|\Del_j L|^4\bigr]
		\le \frac{C}{T_n^2}\sum_{j=1}^n h_n=o_p(1),\\
		\sum_{j=1}^n \E_{j-1}\bigl[|m^{\vartheta}_{j,n}|^4\bigr]
		&\le \frac{C}{T_n^2}\sum_{j=1}^n\sum_{i=1}^m\sum_{k\neq i}
		\bigl|\p_\vartheta\log q_{ik}(X_{t_{j-1}},\vartheta_0)\bigr|^4
		\E_{j-1}[\Del_j N_{ik}^{\mathrm o}]
		\notag\\
		&\le \frac{C}{T_n^2}\sum_{j=1}^n
		h_n(1+|X_{t_{j-1}}|^C)=o_p(1).
	\end{align*}
	
	The triangular-array martingale CLT (see, e.g., \cite[Lemma~3.6]{kessler2012statistical}) therefore yields
	$\Del_{n,\zeta}\cil N(0,\Sigma_\zeta)$.

	\medskip
	\noindent
	\textbf{Step 3: Joint CLT.}

	A Taylor expansion of the estimating equations yields
	\begin{equation}\label{eq:taylor-gamma-Q}
		\sqrt{T_n}(\hat\gam_n-\gam_0)=\Gamma_\gam^{-1}\Del_{n,\gam}+o_p(1),
		\qquad
		\sqrt{T_n}(\hat\vartheta_n-\vartheta_0)=\Gamma_Q^{-1}\Del_{n,Q}+o_p(1).
	\end{equation}
	To obtain the analogous expansion for $\hat\al_n$, namely
	\begin{equation}\label{eq:taylor-alpha}
		\sqrt{T_n}(\hat\al_n-\al_0)=\Gamma_\al^{-1}\Del_{n,\al}+o_p(1),
	\end{equation}
	it suffices to establish the asymptotic decoupling identity
	\begin{equation}\label{eq:alpha-gamma-decoupling}
		\sqrt{T_n}\bigl\{\p_\al\mathbb G_{2,n}(\al_0;\hat\gam_n)-\p_\al\mathbb G_{2,n}(\al_0;\gam_0)\bigr\}=o_p(1).
	\end{equation}
	Define
	$M_n(\gam):=\sqrt{T_n}\,\p_\gam\p_\al\mathbb G_{2,n}(\al_0;\gam)$.
	The martingale difference moment bound used in Step~1 of the CLT proof, combined with the Sobolev embedding argument employed throughout, yields
	\[
	\sup_{\gam\in\Theta_\gam}\|M_n(\gam)\|=O_p(1).
	\]
	By the mean-value theorem, there exists $\tilde\gam_n$ on the segment between $\hat\gam_n$ and $\gam_0$ such that
	\[
	\sqrt{T_n}\bigl\{\p_\al\mathbb G_{2,n}(\al_0;\hat\gam_n)-\p_\al\mathbb G_{2,n}(\al_0;\gam_0)\bigr\}
	=M_n(\tilde\gam_n)\,(\hat\gam_n-\gam_0)=O_p(1)\cdot o_p(1)=o_p(1),
	\]
	where the consistency $\hat\gam_n\cip\gam_0$ is used in the last step. This proves \eqref{eq:alpha-gamma-decoupling}, and hence \eqref{eq:taylor-alpha}.
	
	Combining \eqref{eq:taylor-gamma-Q} and \eqref{eq:taylor-alpha} and setting
	\[
	\Gamma_\zeta:=\begin{pmatrix}\Gamma_\al & 0 & 0\\ 0 & \Gamma_\gam & 0\\ 0 & 0 & \Gamma_Q\end{pmatrix},
	\qquad
	\Del_{n,\zeta}:=\begin{pmatrix}\Del_{n,\al}\\ \Del_{n,\gam}\\ \Del_{n,Q}\end{pmatrix},
	\]
	we obtain the joint stochastic expansion
	\[
	\sqrt{T_n}(\hat\zeta_n-\zeta_0)=\Gamma_\zeta^{-1}\Del_{n,\zeta}+o_p(1).
	\]
	Slutsky's theorem, applied together with the joint CLT $\Del_{n,\zeta}\cil N(0,\Sigma_\zeta)$ established above, yields
	\[
	\sqrt{T_n}(\hat\zeta_n-\zeta_0)\cil N\bigl(0,\,\Gamma_\zeta^{-1}\Sigma_\zeta\Gamma_\zeta^{-1}\bigr).
	\]

	\subsubsection{Proof of Theorem \ref{thm:pldi} and Corollary \ref{cor:moments}}
	\label{sec:proof:pldi}

	Define the following
	\[\chi_{0,\gam}:=\inf_{\gam\neq\gam_0}
	\frac{-\Y_{\gam}(\gam)}{\abs{\gam-\gam_0}^2}, \quad 
	\chi_{0,\al}:=\inf_{\al\neq\al_0}
	\frac{-\Y_{\al}(\al)}{\abs{\al-\al_0}^2},
	\quad
	\chi_{0,Q} := \inf_{\vartheta\neq\vartheta_0} \frac{-\mathbb Y_Q(\vartheta)}{|\vartheta-\vartheta_0|^2}.
	\]
	Also define the contrast differences, for $\bar\gam\in\Theta_\gam$,
	\[
	\Y_{n,\gam}(\gam):=\mathbb G_{1,n}(\gam)-\mathbb G_{1,n}(\gam_0),
	\quad
	\Y_{n,\al}(\al;\bar\gam):=\mathbb G_{2,n}(\al;\bar\gam)-\mathbb G_{2,n}(\al_0;\bar\gam).
	\]
	\[ \mathbb Y_{n,Q}(\vartheta) := \mathbb G_{3,n}(\vartheta)-\mathbb G_{3,n}(\vartheta_0). 
	\]
	
	We first give the following Proposition \ref{prop:pldi-verify}.
	
	\begin{prop}
		\label{prop:pldi-verify}
		Suppose Assumptions \ref{ass:Levy}--\ref{ass:ergodic} and \ref{ass:Y-global} hold.
		Then the following hold.
		
		\medskip
		\noindent
		For the local field $\{\mathbb Z_{n,\gam}(v)\}$, the following conditions are satisfied for every
		$K>0$:
		\begin{enumerate}[label=\textup{(G\arabic*)},leftmargin=3.2em]
			\item
			$\sup_{n\in\mbbn}\E\bigl[|\Delta_{n,\gam}|^K\bigr]<\infty$;
			\medskip
			\item
			$\sup_{n\in\mbbn}
			\E\left[
			\left(
			\sup_{\gam\in\Theta_\gam}
			\|\p_\gam^3\mathbb G_{1,n}(\gam)\|
			\right)^K
			\right]
			<\infty$;
			\medskip
			\item writing $\Gamma_{n,\gam}:=-\p_\gam^2\mathbb G_{1,n}(\gam_0)$,
			$\sup_{n\in\mbbn}
			\E\left[
			\|\sqrt{T_n}(\Gamma_{n,\gam}-\Gamma_\gam)\|^K
			\right]
			<\infty$;
			\medskip
			\item
			$\sup_{n\in\mbbn}
			\E\left[
			\sup_{\gam\in\Theta_\gam}
			\left|
			\sqrt{T_n}\bigl\{\Y_{n,\gam}(\gam)-\Y_\gam(\gam)\bigr\}
			\right|^K
			\right]
			<\infty$.
		\end{enumerate}
		
		\medskip
		\noindent
		Let $\bar\gam_n$ be any $\Theta_\gam$-valued sequence such that, for every $K>0$,
		\begin{equation}
			\label{eq:plug-rate-ass}
			\sup_{n\in\mbbn}\E\bigl[|\sqrt{T_n}(\bar\gam_n-\gam_0)|^K\bigr]<\infty.
		\end{equation}
		Then, for the plug-in local field $\{\mathbb Z_{n,\al}(u;\bar\gam_n)\}$, the following conditions
		are satisfied for every $K>0$:
		\begin{enumerate}[label=\textup{(A\arabic*)},leftmargin=3.2em]
			\item writing $S_{n,\al}(\gam):=\sqrt{T_n}\,\p_\al\mathbb G_{2,n}(\al_0;\gam)$,
			$\sup_{n\in\mbbn}\E\bigl[|S_{n,\al}(\bar\gam_n)|^K\bigr]<\infty$;
			\medskip
			\item
			$\sup_{n\in\mbbn}
			\E\left[
			\left(
			\sup_{\al\in\Theta_\al}
			\|\p_\al^3\mathbb G_{2,n}(\al;\bar\gam_n)\|
			\right)^K
			\right]
			<\infty;$
			\medskip
			\item writing $\Gamma_{n,\al}(\al;\gam):=-\p_\al^2\mathbb G_{2,n}(\al;\gam)$,
			$\sup_{n\in\mbbn}
			\E\left[
			\|\sqrt{T_n}\bigl(\Gamma_{n,\al}(\al_0;\bar\gam_n)-\Gamma_\al\bigr)\|^K
			\right]
			<\infty$;
			\medskip
			\item
			$\sup_{n\in\mbbn}
			\E\left[
			\sup_{\al\in\Theta_\al}
			\left|
			\sqrt{T_n}\bigl\{\Y_{n,\al}(\al;\bar\gam_n)-\Y_\al(\al)\bigr\}
			\right|^K
			\right]
			<\infty$.
		\end{enumerate}
		For the switching field \(\{\mathbb Z_{n,Q}(w)\}\), the following conditions are satisfied for every \(K>0\): 
		\begin{enumerate}[label=\textup{(Q\arabic*)},leftmargin=3.2em] \item $\sup_{n\in\mathbb N} \mathbb E[|\Delta_{n,Q}|^K]<\infty$. 
			\medskip
			\item $\sup_{n\in\mathbb N} \mathbb E\left[ \left( \sup_{\vartheta\in\Theta_\vartheta} \|\partial_\vartheta^3\mathbb G_{3,n}(\vartheta)\| \right)^K \right] <\infty$.
			\medskip
			\item Writing $\Gamma_{n,Q}:=-\partial_\vartheta^2\mathbb G_{3,n}(\vartheta_0)$, $\sup_{n\in\mathbb N} \mathbb E\left[ \|\sqrt{T_n}(\Gamma_{n,Q}-\Gamma_Q)\|^K \right] <\infty$.
			\medskip
			\item $\sup_{n\in\mathbb N} \mathbb E\left[ \sup_{\vartheta\in\Theta_\vartheta} \left| \sqrt{T_n}\{\mathbb Y_{n,Q}(\vartheta)-\mathbb Y_Q(\vartheta)\} \right|^K \right] <\infty$. 
		\end{enumerate}
	\end{prop}
	
	\begin{proof}
		We prove it separately.
		
		\medskip
		\noindent
		\textbf{Step 1: proof of \textup{(G1)}--\textup{(G4)}.}
		
		\emph{Verification of (G1): moments of the score.}
		Recall the decomposition $\Del_{n,\gam}=M_{n,\gam}+R_{n,\gam}$ from the proof of Theorem~\ref{thm:joint-an}. Since $\{M_{n,\gam}\}$ is a martingale array, the Burkholder--Davis--Gundy inequality gives
		\[
		\sup_{n\in\mbbn}\E[|M_{n,\gam}|^K]<\infty \qquad \text{for every } K>0.
		\]
		Combining Proposition \ref{prop:hybrid-one-step} with the orthogonality estimate for martingale-difference arrays yields, in the same way, $\sup_{n\in\mbbn}\E[|R_{n,\gam}|^K]<\infty$ for every $K>0$. Consequently,
		\[
		\sup_{n\in\mbbn}\E[|\Del_{n,\gam}|^K]<\infty,
		\]
		which is (G1).
		
		\smallskip
		\emph{Verification of (G2): moments of higher derivatives.}
		For $|\beta|\le 1$, direct differentiation of \eqref{eq:G1-stage} yields
		\[
		\p_\gam^\beta\p_\gam^3\mathbb G_{1,n}(\gam)
		=\frac1{T_n}\sum_{j=1}^n A_\beta(X_{t_{j-1}},\Lam_{t_{j-1}},\gam)\,(\Del_j X)^2
		+\frac{h_n}{T_n}\sum_{j=1}^n B_\beta(X_{t_{j-1}},\Lam_{t_{j-1}},\gam),
		\]
		where $A_\beta$ and $B_\beta$ are continuous in $\gam$ and satisfy
		\[
		\sup_{|\beta|\le 1}\sup_{\gam\in\Theta_\gam}
		\bigl(|A_\beta(x,i,\gam)|+|B_\beta(x,i,\gam)|\bigr)
		\le C(1+|x|^C).
		\]
		Decomposing $(\Del_j X)^2$ as in Proposition \ref{prop:hybrid-one-step} separates the right-hand side into a martingale-difference component and a predictable remainder. The Burkholder--Davis--Gundy inequality, together with Proposition \ref{prop:hybrid-one-step} and Lemmas \ref{lem:short-time-moment}, then yields
		\[
		\sup_{n\in\mbbn}\sup_{\gam\in\Theta_\gam}
		\E\bigl[\|\p_\gam^\beta\p_\gam^3\mathbb G_{1,n}(\gam)\|^q\bigr]<\infty
		\qquad (|\beta|\le 1)
		\]
		for every $q>0$. A Sobolev embedding and Jensen's inequality therefore give
		\[
		\sup_{n\in\mbbn}\E\!\left[
		\Bigl(\sup_{\gam\in\Theta_\gam}\|\p_\gam^3\mathbb G_{1,n}(\gam)\|\Bigr)^{\!K}
		\right]<\infty
		\qquad \text{for every } K>0,
		\]
		which is (G2).
		
		\smallskip
		\emph{Verification of (G3): moments of the Hessian deviation.}
		Using the decomposition \eqref{eq:hessian-gamma-at-true},
		\[
		\Gamma_{n,\gam}-\Gamma_\gam
		=\left\{\frac1{2n}\sum_{j=1}^n\psi_{j-1}^{\otimes 2}-\Gamma_\gam\right\}+R_{n,1}+R_{n,2}+R_{n,3}.
		\]
		For the leading term, Lemma~\ref{lem:ergodic-rate-unif} applied to the kernel $f(x,i)=\tfrac12\Psi_\gam(x,i)^{\otimes 2}$, together with the identity
		\[
		\sqrt{T_n}\left\{\frac1{2n}\sum_{j=1}^n\psi_{j-1}^{\otimes 2}-\Gamma_\gam\right\}
		=\frac1{\sqrt{T_n}}\sum_{j=1}^n h_n
		\left\{\frac12\Psi_\gam(X_{t_{j-1}},\Lam_{t_{j-1}})^{\otimes 2}-\Gamma_\gam\right\},
		\]
		gives uniformly bounded $K$-th moments. For $R_{n,1}$, the martingale-difference argument used in the proof of Theorem~\ref{thm:joint-an}, now multiplied by $\sqrt{T_n}$, gives
		\[
		\sup_{n\in\mbbn}\E\bigl[|\sqrt{T_n}\,R_{n,1}|^K\bigr]<\infty.
		\]
		For $R_{n,2}$ and $R_{n,3}$, Proposition \ref{prop:hybrid-one-step} implies that, for every $K\ge 2$,
		\[
		\E\bigl[|\sqrt{T_n}\,R_{n,2}|^K\bigr]+\E\bigl[|\sqrt{T_n}\,R_{n,3}|^K\bigr]
		\le C\bigl((n h_n^2)^{K/2}+h_n^{K/2}(n h_n^2)^{K/2}\bigr)
		\le C.
		\]
		Combining these bounds,
		\[
		\sup_{n\in\mbbn}\E\bigl[\|\sqrt{T_n}(\Gamma_{n,\gam}-\Gamma_\gam)\|^K\bigr]<\infty,
		\]
		which is (G3).
		
		\smallskip
		\emph{Verification of (G4): moments of the contrast deviation.}
		The proof of Theorem~\ref{thm:full-consistency} establishes
		\[
		\sup_{\gam\in\Theta_\gam}|\Y_{n,\gam}(\gam)-\Y_\gam(\gam)|
		\le \frac12\sup_{\gam\in\Theta_\gam}|\mathcal Y_{n,\gam}(\gam)+2\Y_\gam(\gam)|
		+\sup_{\gam\in\Theta_\gam}|r_{n,\gam}(\gam)|+R_n,
		\]
		where
		\[
		|R_n|\le \frac{C}{n}\sum_{j=1}^n(1+|X_{t_{j-1}}|^C)|\eta_j|
		+\frac{C h_n}{n}\sum_{j=1}^n(1+|X_{t_{j-1}}|^C).
		\]
		Applying Lemma~\ref{lem:ergodic-rate-unif} with $f(x,i,\gam)=G_\gam(x,i,\gam)$ to the first term gives
		\[
		\sup_{n\in\mbbn}
		\E\!\left[\Bigl(\sqrt{T_n}\sup_{\gam\in\Theta_\gam}|\mathcal Y_{n,\gam}(\gam)+2\Y_\gam(\gam)|\Bigr)^{\!K}\right]<\infty.
		\]
		For the remainder $r_{n,\gam}$, the proof of Theorem~\ref{thm:full-consistency} already establishes that, for some $q>p_\gam$,
		\[
		\E\!\left[\sup_{\gam\in\Theta_\gam}|r_{n,\gam}(\gam)|^q\right]
		\le C\bigl(T_n^{-q/2}+h_n^{q/2}\bigr),
		\]
		and hence
		\[
		\E\!\left[\Bigl(\sqrt{T_n}\sup_{\gam\in\Theta_\gam}|r_{n,\gam}(\gam)|\Bigr)^{\!q}\right]
		\le C\bigl(1+(n h_n^2)^{q/2}\bigr)\le C.
		\]
		Finally, Proposition \ref{prop:hybrid-one-step} yields $\sup_{n\in\mbbn}\E[|\sqrt{T_n}\,R_n|^K]<\infty$. This establishes (G4).

		\medskip
		\noindent
		\textbf{Step 2: proof of \textup{(A1)}--\textup{(A4)}.}
		
		\emph{Verification of (A1).}
		Recall the decomposition $\Del_{n,\al}=M_{n,\al}+R_{n,\al}$ from the proof of Theorem~\ref{thm:joint-an}. The martingale-difference/BDG argument used for (G1), applied to $M_{n,\al}$ and $R_{n,\al}$, gives
		\[
		\sup_{n\in\mbbn}
		\E\!\left[\Bigl(\sup_{\gam\in\Theta_\gam}\bigl\|\sqrt{T_n}\,\p_\gam\p_\al\mathbb G_{2,n}(\al_0;\gam)\bigr\|\Bigr)^{\!K}\right]<\infty.
		\]
		Together with the plug-in rate \eqref{eq:plug-rate-ass}, this yields (A1).
		
		\smallskip
		\emph{Verification of (A2).}
		For $|\nu|\le 1$, direct differentiation of \eqref{eq:G2-stage} gives
		\[
		\p_\al^\nu\p_\al^3\mathbb G_{2,n}(\al;\gam)
		=\frac1{T_n}\sum_{j=1}^n A_\nu(X_{t_{j-1}},\Lam_{t_{j-1}},\al,\gam)\,\eta_j
		+\frac{h_n}{T_n}\sum_{j=1}^n B_\nu(X_{t_{j-1}},\Lam_{t_{j-1}},\al,\gam),
		\]
		where $A_\nu,B_\nu$ are continuous on $\Theta$ and satisfy
		$\sup_{|\nu|\le 1}\sup_{(\al,\gam)\in\Theta}(|A_\nu|+|B_\nu|)\le C(1+|x|^C)$.
		Repeating the BDG/Sobolev argument of (G2), with Proposition \ref{prop:hybrid-one-step}, yields
		\[
		\sup_{n\in\mbbn}
		\E\!\left[\Bigl(\sup_{\al\in\Theta_\al}\|\p_\al^3\mathbb G_{2,n}(\al;\bar\gam_n)\|\Bigr)^{\!K}\right]<\infty,
		\]
		which is (A2).
		
		\smallskip
		\emph{Verification of (A3).}
		Setting $H(x,i,\gam):=G_\al^{(2)}(x,i,\al_0,\gam)$ and $A(x,i,\gam):=\p_\al^2 b(x,i,\al_0)/c(x,i,\gam)^2$,
		\[
		\Gamma_{n,\al}(\al_0;\gam)
		=\frac1n\sum_{j=1}^n H(X_{t_{j-1}},\Lam_{t_{j-1}},\gam)
		-\frac1{T_n}\sum_{j=1}^n A(X_{t_{j-1}},\Lam_{t_{j-1}},\gam)\,\eta_j,
		\]
		so that
		\[
		\sqrt{T_n}\bigl(\Gamma_{n,\al}(\al_0;\bar\gam_n)-\Gamma_\al\bigr)
		=\sqrt{T_n}\bigl(\Gamma_{n,\al}(\al_0;\gam_0)-\Gamma_\al\bigr)+U_n+V_n,
		\]
		where
		$U_n$ is the plug-in difference of the empirical-average part, and $V_n$ is the plug-in difference of the
		$\eta_j$-part. The term $\sqrt{T_n}\bigl(\Gamma_{n,\al}(\al_0;\bar\gam_n)-\Gamma_\al\bigr)$ is controlled by
		Lemma \ref{lem:ergodic-rate-unif}.
		For $U_n$, the mean-value theorem yields a factor
		$|\sqrt{T_n}(\bar\gam_n-\gam_0)|$ times an empirical average of polynomial-growth functions of
		$(X_{t_{j-1}},\Lam_{t_{j-1}})$, so the required moment bound is given by \eqref{eq:plug-rate-ass}. For $V_n$, we separate it into a martingale difference part and a predictable remainder. Then, by BDG inequality, Proposition \ref{prop:hybrid-one-step} and Sobolev inequality, the required moment bound holds.
		Combining these bounds with \eqref{eq:plug-rate-ass}, we obtain
		\[
		\sup_{n\in\mbbn}
		\E\!\left[\bigl\|\sqrt{T_n}\bigl(\Gamma_{n,\al}(\al_0;\bar\gam_n)-\Gamma_\al\bigr)\bigr\|^K\right]<\infty,
		\]
		which is (A3).
		
		\smallskip
		\emph{Verification of (A4).}
		By Theorem~\ref{thm:full-consistency},
		\[
		\Y_{n,\al}(\al;\gam_0)-\Y_\al(\al)
		=-\tfrac12\bigl\{\mathcal Y_{n,\al}(\al)+2\Y_\al(\al)\bigr\}+r_{n,\al}(\al,\gam_0).
		\]
		The leading bracket has uniformly bounded $K$-th moments after multiplication by $\sqrt{T_n}$ by Lemma~\ref{lem:ergodic-rate-unif}, and the proof of Theorem~\ref{thm:full-consistency} shows that, for some $q>p_\al\vee K$,
		\[
		\E\!\left[\sup_{\theta\in\Theta}|r_{n,\al}(\al,\gam)|^q\right]\le C(T_n^{-q/2}+h_n^{q/2}),
		\]
		so that $\sup_n\E[\sup_{\al}|\sqrt{T_n}\,r_{n,\al}(\al,\gam_0)|^K]<\infty$.
		For the plug-in difference,
		\[
		\Y_{n,\al}(\al;\bar\gam_n)-\Y_{n,\al}(\al;\gam_0)
		=-\tfrac12\bigl\{\mathcal Y_{n,\al}(\al;\bar\gam_n)-\mathcal Y_{n,\al}(\al)\bigr\}
		+\bigl\{r_{n,\al}(\al,\bar\gam_n)-r_{n,\al}(\al,\gam_0)\bigr\},
		\]
		where the first bracket is dominated by $|\bar\gam_n-\gam_0|\cdot\tfrac{C}{n}\sum_{j=1}^n(1+|X_{t_{j-1}}|^C)$ as in Theorem~\ref{thm:full-consistency}, and the second is handled by the same single-remainder argument. Combined with \eqref{eq:plug-rate-ass}, this controls $\sqrt{T_n}\sup_\al|\Y_{n,\al}(\al;\bar\gam_n)-\Y_{n,\al}(\al;\gam_0)|$ in $L^K$, and the triangle inequality yields
		\[
		\sup_{n\in\mbbn}
		\E\!\left[\sup_{\al\in\Theta_\al}\bigl|\sqrt{T_n}\bigl\{\Y_{n,\al}(\al;\bar\gam_n)-\Y_\al(\al)\bigr\}\bigr|^K\right]<\infty,
		\]
		which is (A4).
		
		\medskip
		\noindent
		\textbf{Step 3: proof of \textup{(Q1)}--\textup{(Q4)}.}
		
		\emph{Verification of (Q1).}
		Set $D_{ik}(x,\vartheta):=\p_\vartheta\log q_{ik}(x,\vartheta)$ and $D_{ik}^{(r)}(x,\vartheta):=\p_\vartheta^r\log q_{ik}(x,\vartheta)$. The proof of Theorem~\ref{thm:joint-an} gives
		\[
		\Del_{n,Q}
		=\frac1{\sqrt{T_n}}\sum_{j=1}^n\sum_{i=1}^m\sum_{k\ne i}
		D_{ik}(X_{t_{j-1}},\vartheta_0)\,\chi_j^{ik}+o_p(1),
		\]
		with the leading term a martingale sum. By the polynomial-growth bounds on $D_{ik}$, the boundedness of the switching intensities, and the Burkholder--Davis--Gundy inequality,
		\[
		\sup_{n\in\mbbn}\E[|\Del_{n,Q}|^K]<\infty,
		\]
		which is (Q1).
		
		\smallskip
		\emph{Verification of (Q2).}
		For $r=3$, direct differentiation of $\mathbb G_{3,n}$ yields
		\[
		\p_\vartheta^r\mathbb G_{3,n}(\vartheta)
		=\frac1{T_n}\sum_{j=1}^n\sum_{i=1}^m\sum_{k\ne i}\Del_j N_{ik}^{\mathrm o}\,D_{ik}^{(r)}(X_{t_{j-1}},\vartheta)
		-\frac1{T_n}\sum_{j=1}^n h_n\sum_{i=1}^m\mathbf 1_{\{\Lam_{t_{j-1}}=i\}}\p_\vartheta^r q_i(X_{t_{j-1}},\vartheta),
		\]
		where $D_{ik}^{(r)}$ has polynomial growth uniformly in $\vartheta\in\Theta_\vartheta$. Combining Proposition \ref{prop:hybrid-switching}, the BDG inequality, the sampled ergodic moment bound and applying the Sobolev embedding argument used in (G2) yields
		\[
		\sup_{n\in\mbbn}
		\E\!\left[\Bigl(\sup_{\vartheta\in\Theta_\vartheta}\|\p_\vartheta^3\mathbb G_{3,n}(\vartheta)\|\Bigr)^{\!K}\right]<\infty,
		\]
		which is (Q2).
		
		\smallskip
		\emph{Verification of (Q3).}
		The identity
		$-q_{ik}\p_\vartheta^2\log q_{ik}+\p_\vartheta^2 q_{ik}=q_{ik}(\p_\vartheta\log q_{ik})^{\otimes 2}$ gives the decomposition
		\[
		\Gamma_{n,Q}-\Gamma_Q
		=\frac1{T_n}\sum_{j=1}^n h_n\bigl\{H_Q(X_{t_{j-1}},\Lam_{t_{j-1}})-\pi_0(H_Q)\bigr\}+M_{n,Q}+R_{n,Q},
		\]
		where $H_Q(x,i):=\sum_{k\ne i}q_{ik}(x,\vartheta_0)D_{ik}(x,\vartheta_0)^{\otimes 2}$, $M_{n,Q}$ is a martingale-difference sum, and $R_{n,Q}$ is the endpoint-count remainder. Lemma~\ref{lem:ergodic-rate-unif} controls the first term at the $T_n^{-1/2}$ scale, the BDG inequality controls $\sqrt{T_n}\,M_{n,Q}$, and Proposition \ref{prop:hybrid-switching} gives $\sqrt{T_n}\,R_{n,Q}=O_p(\sqrt{n h_n^2})=o_p(1)$ with bounded polynomial moments. Hence
		\[
		\sup_{n\in\mbbn}\E\bigl[\|\sqrt{T_n}(\Gamma_{n,Q}-\Gamma_Q)\|^K\bigr]<\infty,
		\]
		which is (Q3).
		
		\smallskip
		\emph{Verification of (Q4).}
		Set $L_{ik}(x,\vartheta):=\log\{q_{ik}(x,\vartheta)/q_{ik}(x,\vartheta_0)\}$. Then
		\[
		\mathbb Y_{n,Q}(\vartheta)-\mathbb Y_Q(\vartheta)
		=\frac1{T_n}\sum_{j=1}^n\sum_{i=1}^m\sum_{k\ne i}\chi_j^{ik}L_{ik}(X_{t_{j-1}},\vartheta)
		+\frac1{T_n}\sum_{j=1}^n h_n\bigl\{F_Q(X_{t_{j-1}},\Lam_{t_{j-1}},\vartheta)-\mathbb Y_Q(\vartheta)\bigr\}
		+R_{n,Q}(\vartheta),
		\]
		with $R_{n,Q}(\vartheta)$ the endpoint remainder. Lemma~\ref{lem:ergodic-rate-unif} controls the second term uniformly in $\vartheta$. The Sobolev embedding on $\Theta_\vartheta$, together with the BDG inequality and the polynomial-growth bounds on $\p_\vartheta^\ell L_{ik}$ for $\ell\le 1$, gives
		\[
		\sup_{n\in\mbbn}
		\E\!\left[\sup_{\vartheta\in\Theta_\vartheta}\left|\frac1{\sqrt{T_n}}\sum_{j=1}^n\sum_{i=1}^m\sum_{k\ne i}\chi_j^{ik}L_{ik}(X_{t_{j-1}},\vartheta)\right|^K\right]<\infty,
		\]
		and Proposition \ref{prop:hybrid-switching} bounds the endpoint remainder. Therefore
		\[
		\sup_{n\in\mbbn}
		\E\!\left[\sup_{\vartheta\in\Theta_\vartheta}\bigl|\sqrt{T_n}\bigl\{\mathbb Y_{n,Q}(\vartheta)-\mathbb Y_Q(\vartheta)\bigr\}\bigr|^K\right]<\infty,
		\]
		which is (Q4).
	\end{proof}

	By Proposition~\ref{prop:pldi-verify}, the local random fields $\mathbb Z_{n,\gam}(v)$ and $\mathbb Z_{n,Q}(w)$ satisfy the hypotheses of \cite[Theorem~3(c)]{Yos11}, so that, for every $L>0$,
	\[
	\sup_{n\in\mbbn}\Prob\!\left(\sup_{v\in\mathbb U_{n,\gam}(r)}\mathbb Z_{n,\gam}(v)\ge e^{-r^2/C_L}\right)\le \frac{C_L}{r^L},
	\qquad r>0,
	\]
	and similarly for $\mathbb Z_{n,Q}$; see \cite[Section~4]{Mas13as}.
	
	\smallskip
	\noindent\textit{Verification of \eqref{eq:plug-rate-ass} for $\bar\gam_n=\hat\gam_n$.}
	Let $\hat v_n:=\sqrt{T_n}(\hat\gam_n-\gam_0)$. Since $\hat\gam_n$ maximizes $\mathbb G_{1,n}$, we have $\mathbb Z_{n,\gam}(\hat v_n)\ge 1$, and hence on the event $\{|\hat v_n|>r\}$,
	\[
	\hat v_n\in\mathbb U_{n,\gam}(r)
	\quad\text{and}\quad
	\sup_{v\in\mathbb U_{n,\gam}(r)}\mathbb Z_{n,\gam}(v)\ge \mathbb Z_{n,\gam}(\hat v_n)\ge 1\ge e^{-r^2/C_L}.
	\]
	The PLDI above therefore yields
	\[
	\sup_{n\in\mbbn}\Prob(|\hat v_n|>r)\le \frac{C_L}{r^L}, \qquad r>0,
	\]
	and the tail-integral formula gives
	\[
	\sup_{n\in\mbbn}\E\bigl[|\sqrt{T_n}(\hat\gam_n-\gam_0)|^K\bigr]<\infty
	\qquad\text{for every } K>0,
	\]
	which is \eqref{eq:plug-rate-ass}.
	
	\smallskip
	\noindent\textit{Joint tail bound.}
	With \eqref{eq:plug-rate-ass} verified, Proposition~\ref{prop:pldi-verify} delivers the polynomial tail bound for $\mathbb Z_{n,\al}(u;\hat\gam_n)$, and the argument above yields
	$\sup_{n\in\mbbn}\Prob(\sqrt{T_n}|\hat\al_n-\al_0|>r)\le C/r^L$. Combining the three component bounds,
	\[
	\sup_{n\in\mbbn}\Prob\bigl(\sqrt{T_n}|\hat\zeta_n-\zeta_0|>r\bigr)\le \frac{C_L'}{r^L},
	\qquad r>0,
	\]
	which is Theorem~\ref{thm:pldi}.
	
	\smallskip
	\noindent\textit{Proof of Corollary~\ref{cor:moments}.}
	Set $Y_n:=\sqrt{T_n}(\hat\zeta_n-\zeta_0)$. By Theorem~\ref{thm:pldi} and the tail-integral formula,
	\[
	\sup_{n\in\mbbn}\E[|Y_n|^q]
	=\sup_{n\in\mbbn}\int_0^\infty q r^{q-1}\Prob(|Y_n|>r)\,dr<\infty
	\qquad\text{for every } q>0.
	\]
	Let $f:\R^p\to\R$ be continuous with $|f(u)|\le C(1+|u|^m)$ for some $C,m>0$. Theorem~\ref{thm:joint-an} gives $Y_n\cil Y$, and the continuous mapping theorem then yields $f(Y_n)\cil f(Y)$. The moment bound above, applied with any $q>m$, implies that $\{f(Y_n)\}_{n\in\mbbn}$ is uniformly integrable, and weak convergence combined with uniform integrability gives
	\[
	\lim_{n\to\infty}\E[f(Y_n)]=\E[f(Y)],
	\]
	which is Corollary~\ref{cor:moments}.

	\section{Appendix}
	\label{sec:tech}
	
	This appendix collects the technical lemmas used in the proofs of the main results.
	
	We begin with an analogue of \cite[Lemma~4.3]{Mas13as}, established by the same Sobolev-embedding and exponential-ergodicity argument as in that reference; see also \cite{Yos11}.
	
	\begin{lem}
		\label{lem:ergodic-rate-unif}
		Suppose Assumptions \ref{ass:Levy}--\ref{ass:ergodic} hold. Let $\Xi\subset\R^d$ be compact and let
		$f:\R\times\mbbs\times\Xi\to\R$ be continuously differentiable in the parameter variable with
		\[
		\sup_{\vartheta\in\Xi}\sum_{|\beta|\le 1}
		\abs{\p_\vartheta^\beta f(x,i,\vartheta)}
		\le C(1+\abs{x}^C),
		\qquad (x,i)\in\R\times\mbbs.
		\]
		Then, for every $K>0$,
		\[
		\sup_{n\in\mbbn}
		\E\left[
		\left(
		\sup_{\vartheta\in\Xi}
		\left|
		\frac1{\sqrt{T_n}}
		\sum_{j=1}^n h_n
		\Bigl(
		f(X_{t_{j-1}},\Lam_{t_{j-1}},\vartheta)
		-
		\int f(x,i,\vartheta)\,\pi_0(dx,di)
		\Bigr)
		\right|
		\right)^K
		\right]
		<\infty.
		\]
	\end{lem}
	
	\begin{proof}
		Under Assumption \ref{ass:ergodic}, the skeleton chain
		$Y_j:=(X_{t_j},\Lam_{t_j})$ is strictly stationary and exponentially $\beta$-mixing. More precisely,
		its $\beta$-mixing coefficient satisfies
		\[
		\beta_Y(k)\le C e^{-a k h_n},
		\qquad k\in\mbbn,
		\]
		for some constants $C,a>0$ independent of $n$.
		Fix $|\beta|\le1$ and $\vartheta\in\Xi$, and set
		\[
		\bar f_{j-1}^{(\beta)}(\vartheta)
		:=
		\p_\vartheta^\beta f(X_{t_{j-1}},\Lam_{t_{j-1}},\vartheta)
		-
		\int \p_\vartheta^\beta f(x,i,\vartheta)\,\pi_0(dx,di).
		\]
		By Assumption \ref{ass:ergodic},
		\[
		\sup_{\vartheta\in\Xi}
		\int \abs{\p_\vartheta^\beta f(x,i,\vartheta)}^m\,\pi_0(dx,di)<\infty
		\qquad (m>0).
		\]
		Hence the Rosenthal inequality for exponentially $\beta$-mixing sequences (See \cite{Yos11}[Lemma 4]), applied to
		$\{h_n\bar f_{j-1}^{(\beta)}(\vartheta)\}_{j\ge1}$, yields for every $K>0$,
		\[
		\sup_{n\in\mbbn}\sup_{\vartheta\in\Xi}
		\E\left[
		\left|
		\frac1{\sqrt{T_n}}
		\sum_{j=1}^n h_n\bar f_{j-1}^{(\beta)}(\vartheta)
		\right|^K
		\right]
		\le C_K.
		\]
		Now choose $q>d$. By Sobolev embedding $W^{1,q}(\Xi)\hookrightarrow C(\Xi)$ and the preceding
		moment bounds for derivatives up to order one,
		\begin{align*}
			&\E\left[
			\left(
			\sup_{\vartheta\in\Xi}
			\left|
			\frac1{\sqrt{T_n}}
			\sum_{j=1}^n h_n
			\Bigl(
			f(X_{t_{j-1}},\Lam_{t_{j-1}},\vartheta)
			-
			\int f(x,i,\vartheta)\,\pi_0(dx,di)
			\Bigr)
			\right|
			\right)^K
			\right]
			\\
			&\le
			C
			\sum_{|\beta|\le1}
			\int_\Xi
			\E\left[
			\left|
			\frac1{\sqrt{T_n}}
			\sum_{j=1}^n h_n\bar f_{j-1}^{(\beta)}(\vartheta)
			\right|^K
			\right]d\vartheta
			\le C_K.
		\end{align*}
		This proves the lemma.
	\end{proof}
	
	\subsection{Auxiliary Lemmas for Propositions in Section \ref{sec:estimation}}
	
	We split the proofs of Propositions~\ref{prop:hybrid-one-step} and~\ref{prop:hybrid-switching} into a sequence of lemmas. Recall the increment decomposition
	\[
	\Del_j X=h_n b_{j-1}(\al_0)+c_{j-1}(\gam_0)\,\Del_j L+r_{j,n}.
	\]

	\begin{lem}
		\label{lem:increment}
		Suppose Assumptions \ref{ass:Levy}--\ref{ass:ergodic} hold.
		Then, for every $q>0$,
		\[
		\E\left[\abs{r_{j,n}}^q\mid\mathcal{F}_{t_{j-1}}\right]
		\le C_q h_n^{q\wedge2}\bigl(1+\abs{X_{t_{j-1}}}^{C_q}\bigr)
		\]
		uniformly in $j,n$.
	\end{lem}
	
	\begin{proof}
		Set $A_j:=\int_{t_{j-1}}^{t_j}\{b(X_s,\Lam_s,\al_0)-b_{j-1}(\al_0)\}\,ds$ and $B_j:=\int_{t_{j-1}}^{t_j}\{c(X_{s-},\Lam_{s-},\gam_0)-c_{j-1}(\gam_0)\}\,dL_s$, so that $r_{j,n}=A_j+B_j$.
		
		\emph{Bound on $A_j$.} For $q\ge 1$, Hölder's inequality combined with the pathwise bound
		\[
		|b(X_s,\Lam_s,\al_0)-b_{j-1}(\al_0)|
		\le C|X_s-X_{t_{j-1}}|+C(1+|X_s|)\mathbf 1_{\{\Lam_s\ne\Lam_{t_{j-1}}\}}
		\]
		from Assumption~\ref{ass:coeff}, together with Lemma~\ref{lem:short-time-moment} (applying the Cauchy--Schwarz inequality on the indicator term), gives
		\[
		\E[|A_j|^q\mid\mathcal F_{t_{j-1}}]\le C_q h_n^q\bigl(1+|X_{t_{j-1}}|^{C_q}\bigr).
		\]
		The conditional Lyapunov inequality extends this to $0<q<1$.
		
		\emph{Bound on $B_j$.} Set $\Delta c_s:=c(X_{s-},\Lam_{s-},\gam_0)-c_{j-1}(\gam_0)$. The Burkholder--Davis--Gundy and Rosenthal inequalities yield, for $q\ge 2$,
		\begin{equation}\label{eq:Bj-BDG}
			\E[|B_j|^q\mid\mathcal F_{t_{j-1}}]
			\le C_q\!\left\{
			\E\!\left[\!\left(\int_{t_{j-1}}^{t_j}|\Delta c_s|^2\,ds\right)^{\!q/2}\!\Big|\mathcal F_{t_{j-1}}\right]
			+\E\!\left[\int_{t_{j-1}}^{t_j}|\Delta c_s|^q\,ds\,\Big|\mathcal F_{t_{j-1}}\right]\!
			\right\}.
		\end{equation}
		For $r\ge 2$ and $(x_0,i_0)\in\R\times\mbbs$, set $\Phi_r^{x_0,i_0}(x,i):=|c(x,i,\gam_0)-c(x_0,i_0,\gam_0)|^r$. Assumptions~\ref{ass:Levy} and~\ref{ass:coeff} yield the generator bound
		\[
		|\mathcal A_{\theta_0}\Phi_r^{x_0,i_0}(x,i)|\le C_r\bigl(1+|x|^{C_r}+|x_0|^{C_r}\bigr).
		\]
		Taking $(x_0,i_0)=(X_{t_{j-1}},\Lam_{t_{j-1}})$ and applying Dynkin's formula for hybrid-switching jump processes (\cite{ZhuYinBaran2017,yin2009hybrid}) along $X_{\cdot\wedge\tau_m}$ with $\tau_m:=\inf\{u\ge t_{j-1}:|X_u|\ge m\}$, then letting $m\to\infty$ via Fatou's lemma, gives
		\[
		\E\bigl[|\Delta c_s|^r\mid\mathcal F_{t_{j-1}}\bigr]
		\le C_r(s-t_{j-1})\bigl(1+|X_{t_{j-1}}|^{C_r}\bigr),\qquad s\in[t_{j-1},t_j].
		\]
		Integrating this bound and applying Jensen's inequality to the first term of \eqref{eq:Bj-BDG} gives $\E[|B_j|^q\mid\mathcal F_{t_{j-1}}]\le C_q h_n^2(1+|X_{t_{j-1}}|^{C_q})$ for $q\ge 2$; the conditional Lyapunov inequality extends this to $C_q h_n^q(1+|X_{t_{j-1}}|^{C_q})$ for $0<q<2$.
		
		Since $h_n\in(0,1]$, combining the two bounds yields
		\[
		\E[|r_{j,n}|^q\mid\mathcal F_{t_{j-1}}]\le C_q h_n^{q\wedge 2}\bigl(1+|X_{t_{j-1}}|^{C_q}\bigr).\qedhere
		\]
	\end{proof}

	\begin{lem}
		\label{lem:centered-increment}
		Under Assumptions \ref{ass:Levy}--\ref{ass:ergodic}, there exist remainders
		$\mathcal R_{j,n}^{(1)},\mathcal R_{j,n}^{(2)},\mathcal R_{j,n}^{(3)}$ and constants
		$C>0$, $q_0>0$ such that
		\[
		\E[\eta_j\mid\mathcal F_{t_{j-1}}]=\mathcal R_{j,n}^{(1)},
		\qquad
		\abs{\mathcal R_{j,n}^{(1)}}
		\le C h_n^2\bigl(1+\abs{X_{t_{j-1}}}^{q_0}\bigr),
		\]
		\[
		\E[\eta_j^2\mid\mathcal F_{t_{j-1}}]=h_n c_{j-1}(\gam_0)^2+\mathcal R_{j,n}^{(2)},
		\qquad
		\abs{\mathcal R_{j,n}^{(2)}}
		\le C h_n^{3/2}\bigl(1+\abs{X_{t_{j-1}}}^{q_0}\bigr),
		\]
		\[
		\E[\eta_j^3\mid\mathcal F_{t_{j-1}}]=h_n m_3 c_{j-1}(\gam_0)^3+\mathcal R_{j,n}^{(3)},
		\qquad
		\abs{\mathcal R_{j,n}^{(3)}}
		\le C h_n^{3/2}\bigl(1+\abs{X_{t_{j-1}}}^{q_0}\bigr).
		\]
		Moreover, for every $q>0$, there exists $C_q>0$ such that
		\[
		\E\left[\abs{\eta_j}^q\mid\mathcal F_{t_{j-1}}\right]
		\le C_q h_n^{q/2\wedge1}\bigl(1+\abs{X_{t_{j-1}}}^{C_q}\bigr).
		\]
	\end{lem}
	
	\begin{proof}
		Recall from Lemma~\ref{lem:increment} that $\eta_j=c_{j-1}(\gam_0)\Del_j L+r_{j,n}$ with $r_{j,n}=A_j+B_j$.
		
		\emph{Conditional $q$-th moment.}
		By Lemma~\ref{lem:increment}, Assumption~\ref{ass:coeff}, the moment bound $\E[|\Del_j L|^q]\le C_q h_n^{q/2\wedge 1}$, and the BDG inequality,
		\[
		\E\bigl[|\eta_j|^q\mid\mathcal F_{t_{j-1}}\bigr]\le C_q\,h_n^{q/2\wedge 1}\bigl(1+|X_{t_{j-1}}|^{C_q}\bigr).
		\]
		
		\emph{Conditional mean.}
		Since $\E[L_1]=0$, $\E_{j-1}[B_j]=0$, and hence $\E[\eta_j\mid\mathcal F_{t_{j-1}}]=\E_{j-1}[A_j]=:\mathcal R_{j,n}^{(1)}$. Setting $f(x,i):=b(x,i,\al_0)$, Assumption~\ref{ass:coeff} yields $|\mathcal A_{\theta_0}f(x,i)|\le C(1+|x|^{q_0})$ for some $C,q_0>0$. Dynkin's formula on $[t_{j-1},s]$ together with Lemma~\ref{lem:short-time-moment} gives
		\[
		\bigl|\E_{j-1}[f(X_s,\Lam_s)-f(X_{t_{j-1}},\Lam_{t_{j-1}})]\bigr|
		\le C(s-t_{j-1})\bigl(1+|X_{t_{j-1}}|^{q_0}\bigr),
		\]
		and integrating over $[t_{j-1},t_j]$ yields
		$|\mathcal R_{j,n}^{(1)}|\le C h_n^2\bigl(1+|X_{t_{j-1}}|^{q_0}\bigr)$.
		
		\emph{Conditional second moment.}
		Expanding $\eta_j^2=c_{j-1}(\gam_0)^2(\Del_j L)^2+2c_{j-1}(\gam_0)\Del_j L\,r_{j,n}+r_{j,n}^2$ and using $\E[(\Del_j L)^2]=h_n$,
		\[
		\E[\eta_j^2\mid\mathcal F_{t_{j-1}}]=h_n c_{j-1}(\gam_0)^2+\mathcal R_{j,n}^{(2)},
		\]
		where Cauchy--Schwarz on the cross term together with Lemma~\ref{lem:increment} gives
		$|\mathcal R_{j,n}^{(2)}|\le C h_n^{3/2}\bigl(1+|X_{t_{j-1}}|^{C}\bigr)$.
		
		\emph{Conditional third moment.}
		Similarly, expanding $\eta_j^3$ and using $\E[(\Del_j L)^3]=h_n m_3$,
		\[
		\E[\eta_j^3\mid\mathcal F_{t_{j-1}}]=h_n m_3\,c_{j-1}(\gam_0)^3+\mathcal R_{j,n}^{(3)},
		\]
		where the same Cauchy--Schwarz argument applied to each of the three cross terms, combined with Assumption~\ref{ass:Levy} and Lemma~\ref{lem:increment}, gives
		$|\mathcal R_{j,n}^{(3)}|\le C h_n^{3/2}\bigl(1+|X_{t_{j-1}}|^{C}\bigr)$.
	\end{proof}

	\begin{lem}
		\label{lem:q-obs-direct}
		Suppose Assumptions \ref{ass:Levy}, \ref{ass:coeff} and
		\ref{ass:ergodic} hold.
		Then, for each \(i\neq k\),
		\[
		\E\!\left[
		\Delta_j N_{ik}^{\mathrm{o}}
		\Bigm| \mathcal F_{t_{j-1}}
		\right]
		=
		\mathbf 1_{\{\Lambda_{t_{j-1}}=i\}}
		q_{ik}(X_{t_{j-1}},\vartheta_0)\,h_n
		+
		r_{j,n}^{ik},
		\]
		where \(r_{j,n}^{ik}\) is \(\mathcal F_{t_{j-1}}\)-measurable and satisfies
		$|r_{j,n}^{ik}|
		\le
		C h_n^{3/2}\bigl(1+|X_{t_{j-1}}|^{C}\bigr)$
		for some constant \(C>0\) independent of \(j\) and \(n\).
	\end{lem}
	
	\begin{proof}
		Fix $i\ne k$ and let $f_k(x,\ell):=\mathbf 1_{\{\ell=k\}}$. Since $f_k$ depends only on $\ell$, its generator at $(x,\ell)$ equals $q_{\ell k}(x,\vartheta_0)$ for $\ell\ne k$. Dynkin's formula, together with $f_k(X_{t_{j-1}},\Lam_{t_{j-1}})=0$ on $\{\Lam_{t_{j-1}}=i\}$, gives
		\[
		\E\bigl[\Del_j N_{ik}^{\mathrm o}\mid\mathcal F_{t_{j-1}}\bigr]
		=\mathbf 1_{\{\Lam_{t_{j-1}}=i\}}\,\E\!\left[\int_{t_{j-1}}^{t_j}\mathcal A_{\theta_0}f_k(X_s,\Lam_s)\,ds\,\bigg|\,\mathcal F_{t_{j-1}}\right].
		\]
		Adding and subtracting $q_{ik}(X_s,\vartheta_0)$ inside the integral and applying the triangle inequality,
		\[
		\Bigl|\E_{j-1}\!\bigl[\Del_j N_{ik}^{\mathrm o}\bigr]
		-\mathbf 1_{\{\Lam_{t_{j-1}}=i\}}q_{ik}(X_{t_{j-1}},\vartheta_0)\,h_n\Bigr|
		\le \mathbf 1_{\{\Lam_{t_{j-1}}=i\}}\int_{t_{j-1}}^{t_j}\bigl(\mathrm I_s+\mathrm{II}_s\bigr)\,ds,
		\]
		where
		\[
		\mathrm I_s:=\E\!\bigl[|\mathcal A_{\theta_0}f_k(X_s,\Lam_s)-q_{ik}(X_s,\vartheta_0)|\,\big|\,\mathcal F_{t_{j-1}}\bigr],
		\qquad
		\mathrm{II}_s:=\E\!\bigl[|q_{ik}(X_s,\vartheta_0)-q_{ik}(X_{t_{j-1}},\vartheta_0)|\,\big|\,\mathcal F_{t_{j-1}}\bigr].
		\]
		On $\{\Lam_{t_{j-1}}=i\}$, the integrand defining $\mathrm I_s$ vanishes when $\Lam_s=i$ and is otherwise dominated by $C(1+|X_s|^C)$, so Lemma~\ref{lem:short-time-moment} gives $\mathbf 1_{\{\Lam_{t_{j-1}}=i\}}\mathrm I_s\le C(s-t_{j-1})(1+|X_{t_{j-1}}|^C)$. Assumption~\ref{ass:coeff} and Lemma~\ref{lem:short-time-moment} also give $\mathrm{II}_s\le C(s-t_{j-1})^{1/2}(1+|X_{t_{j-1}}|^C)$. Integrating these bounds over $[t_{j-1},t_j]$ produces $C h_n^2$ and $C h_n^{3/2}$ contributions, respectively, so
		\[
		\E\bigl[\Del_j N_{ik}^{\mathrm o}\mid\mathcal F_{t_{j-1}}\bigr]
		=\mathbf 1_{\{\Lam_{t_{j-1}}=i\}}\,q_{ik}(X_{t_{j-1}},\vartheta_0)\,h_n+r_{j,n}^{ik},
		\qquad |r_{j,n}^{ik}|\le C h_n^{3/2}(1+|X_{t_{j-1}}|^C).\qedhere
		\]
	\end{proof}

	\begin{lem}
		\label{lem:cont-switch-cross}
		Suppose Assumptions~\ref{ass:Levy}, \ref{ass:coeff}, and
		\ref{ass:ergodic} hold. Let \(F_{j-1}\) be
		\(\mathcal F_{t_{j-1}}\)-measurable and satisfy
		$|F_{j-1}|\le C(1+|X_{t_{j-1}}|^C)$.
		Then, for every \(i\neq k\),
		\[
		\left|
		\mathbb E_{j-1}
		\left[
		F_{j-1}\Delta_jL\,\chi_j^{ik}
		\right]
		\right|
		\le
		Ch_n^{3/2}(1+|X_{t_{j-1}}|^C),
		\]
		\[
		\left|
		\mathbb E_{j-1}
		\left[
		F_{j-1}\{(\Delta_jL)^2-h_n\}\chi_j^{ik}
		\right]
		\right|
		\le
		Ch_n^{3/2}(1+|X_{t_{j-1}}|^C).
		\]
	\end{lem}
	
	\begin{proof}
		For $s\in(t_{j-1},t_j]$ set $\lambda_s^{ik}:=\mathbf 1_{\{\Lam_{s-}=i\}}q_{ik}(X_{s-},\vartheta_0)$ and $\lambda_{j-1}^{ik}:=\mathbf 1_{\{\Lam_{t_{j-1}}=i\}}q_{ik}(X_{t_{j-1}},\vartheta_0)$. Let
		\[
		\Del_j\mathcal N_{ik}
		:=\int_{t_{j-1}}^{t_j}\int_{\R_+}\mathbf 1_{\{\Lam_{s-}=i\}}\mathbf 1_{\Gamma_{ik}(X_{s-},\vartheta_0)}(u)\,N(ds,du)
		\]
		denote the actual number of $i\to k$ jumps of $\Lam$ on $(t_{j-1},t_j]$, and let $R_j$ be the total number of switches on the same interval. Fix $Y_j\in\{\Del_j L,(\Del_j L)^2-h_n\}$; for either choice, $\E_{j-1}|Y_j|^2\le C h_n$.
		
		Since $\{\Del_j N_{ik}^{\mathrm o}\ne\Del_j\mathcal N_{ik}\}\subset\{R_j\ge 2\}$ and $|\Del_j N_{ik}^{\mathrm o}|\vee|\Del_j\mathcal N_{ik}|\le 1+R_j$,
		\[
		|\Del_j N_{ik}^{\mathrm o}-\Del_j\mathcal N_{ik}|^2\le (1+R_j)^2\mathbf 1_{\{R_j\ge 2\}}.
		\]
		The bounded total switching rate from Assumption~\ref{ass:coeff} dominates $R_j$, conditionally on $\mathcal F_{t_{j-1}}$, by a Poisson random variable with mean $Ch_n$, hence
		\begin{equation}\label{eq:obs-actual-count-error}
			\E_{j-1}\bigl|\Del_j N_{ik}^{\mathrm o}-\Del_j\mathcal N_{ik}\bigr|^2\le C h_n^2,
		\end{equation}
		and Cauchy--Schwarz yields
		\begin{equation}\label{eq:cross-endpoint-error}
			\bigl|\E_{j-1}\!\bigl[Y_j\bigl(\Del_j N_{ik}^{\mathrm o}-\Del_j\mathcal N_{ik}\bigr)\bigr]\bigr|\le C h_n^{3/2}.
		\end{equation}
		
		Write $\Del_j\mathcal N_{ik}=\int_{t_{j-1}}^{t_j}\lambda_s^{ik}\,ds+\Del_j M_{ik}$ for the corresponding compensated martingale increment. Since the Poisson random measure driving $\Lam$ is independent of $L$, conditioning on the $L$-path over $[t_{j-1},t_j]$ gives $\E_{j-1}[Y_j\Del_j M_{ik}]=0$. Using also $\E_{j-1}[Y_j]=0$ to eliminate the frozen-rate contribution,
		\[
		\E_{j-1}[Y_j\Del_j\mathcal N_{ik}]
		=\E_{j-1}\!\left[Y_j\int_{t_{j-1}}^{t_j}(\lambda_s^{ik}-\lambda_{j-1}^{ik})\,ds\right].
		\]
		Assumption~\ref{ass:coeff}, Lemma~\ref{lem:increment}, and Proposition~\ref{prop:hybrid-one-step} give $\E_{j-1}|\lambda_s^{ik}-\lambda_{j-1}^{ik}|^2\le C(s-t_{j-1})(1+|X_{t_{j-1}}|^C)$, and Cauchy--Schwarz then yields
		\[
		|\E_{j-1}[Y_j\Del_j\mathcal N_{ik}]|
		\le C h_n^{1/2}\!\int_{t_{j-1}}^{t_j}(s-t_{j-1})^{1/2}ds\,(1+|X_{t_{j-1}}|^C)
		\le C h_n^2(1+|X_{t_{j-1}}|^C).
		\]
		
		Combining this with \eqref{eq:cross-endpoint-error} gives $|\E_{j-1}[Y_j\chi_j^{ik}]|\le C h_n^{3/2}(1+|X_{t_{j-1}}|^C)$, and multiplying by the polynomially bounded, $\mathcal F_{t_{j-1}}$-measurable factor $F_{j-1}$ delivers the two stated estimates.
	\end{proof}

	\subsection{Auxiliary Lemmas for Proposition~\ref{prop:fixed-h-skeleton}}
	\label{app:erg}
	Recall \(L^{\mathrm s}\) and \(L^{\mathrm r}\) from Section~\ref{sec:proof:Tchain}.

	\begin{lem}
		\label{lem:one-jump-geometry}
		Suppose Assumptions~\ref{ass:Levy} and
		Assumption~\ref{ass:erg-coeff} hold.
		Fix a compact interval $K\subset\mathbb R$, a regime $i\in\mathbb S$.
		Then there exist
		\begin{itemize}
			\item a compact interval $K_0$ such that
			$K\subset\operatorname{int}(K_0)\subset K_0$;
			\item constants
			$h_K>0$,
			$\delta_{z,K}>0$,
			$0<\kappa_{1,K}\le \kappa_{2,K}<\infty$;
		\end{itemize}
		such that, for every fixed $h\in(0,h_K]$ and every $y\in K$, there exist
		an open interval $U_{y,h}^0\ni y$, an open interval $J_{y,h}^0\ni y$, and an
		event
		$\mathcal R_{K,h}\in\sigma(L^{\mathrm r})$ with 
		$\mathbb P(\mathcal R_{K,h})\ge\frac34$,
		such that, setting
		\[
		I_{\tau,h}:=\left[\frac h4,\frac{3h}{4}\right],
		\qquad
		I_{z,K}:=[-\delta_{z,K},\delta_{z,K}]\subset(-r_0,r_0),
		\]
		the following hold for every $(u,\omega,\tau,z)\in \overline{U_{y,h}^0} \times \mathcal R_{K,h} \times I_{\tau,h} \times I_{z,K}$.
		\begin{enumerate}[label=\textup{(\roman*)},leftmargin=3em]
			\item The equation
			\begin{align}
				\label{eq:frozen-one-jump-sde}
				\bar X_t^{u,\tau,z}
				&=
				u
				+
				\int_0^t b(\bar X_s^{u,\tau,z},i,\alpha_0)\,ds
				+
				\int_0^t c(\bar X_{s-}^{u,\tau,z},i,\gamma_0)\,dL_s^{\mathrm r}
				\notag\\
				&\qquad
				+
				\mathbf 1_{\{t\ge\tau\}}\,
				c(\bar X_{\tau-}^{u,\tau,z},i,\gamma_0)z,
				\qquad 0\le t\le h,
			\end{align}
			admits a unique strong solution satisfying
			$\bar X_t^{u,\tau,z}(\omega)\in K_0$,
			for $0\le t\le h$.
			
			\item The endpoint map
			$G_{u,\tau,\omega}^{(h)}(z)
			:=
			\bar X_h^{u,\tau,z}(\omega)
			\colon I_{z,K}\to\mathbb R$
			belongs to $C^1(I_{z,K})$, is strictly increasing, and satisfies
			\[
			\kappa_{1,K}
			\le
			\partial_zG_{u,\tau,\omega}^{(h)}(z)
			\le
			\kappa_{2,K},
			\qquad z\in I_{z,K}.
			\]
			
			\item The image of $I_{z,K}$ contains the target interval:
			$J_{y,h}^0
			\subset
			G_{u,\tau,\omega}^{(h)}(I_{z,K})$.
		\end{enumerate}
	\end{lem}

	\begin{proof}
		\medskip
		\noindent\textbf{Step 1: deterministic construction on a small time window.}
		
		Choose compact intervals $K^\flat$, $K^\dagger$, $K_1$, and $K_0$ such that
		\[
		K
		\subset \operatorname{int}(K^\flat)
		\subset K^\flat
		\subset \operatorname{int}(K^\dagger)
		\subset K^\dagger
		\subset \operatorname{int}(K_1)
		\subset K_1
		\subset \operatorname{int}(K_0)
		\subset K_0.
		\]
		Set
		\[
		L_b:=\sup_{v\in K_0}|\partial_x b(v,i,\alpha_0)|,
		\qquad
		M_b:=\sup_{v\in K_0}|b(v,i,\alpha_0)|,
		\]
		\[
		L_c:=\sup_{v\in K_0}|\partial_x c(v,i,\gamma_0)|,
		\qquad
		M_c:=\sup_{v\in K_0}|c(v,i,\gamma_0)|,
		\]
		and 
		\[
		c_{*,K}:=\inf_{v\in K_0}c(v,i,\gamma_0)>0 .
		\]
		Let $\phi_t^{x,i}$ denote the flow of the ODE $\dot y_t=b(y_t,i,\alpha_0)$ with $y_0=x$ and $d_{\dagger,1}:=\operatorname{dist}(K^\dagger,K_1^c)>0$.
		Choose $\delta_{z,K}>0$ such that
		\[
		I_{z,K}:=[-\delta_{z,K},\delta_{z,K}]\subset(-r_0,r_0),
		\qquad
		M_c\delta_{z,K}<\frac{d_{\dagger,1}}4 .
		\]
		Define $\mathcal C_0
		:=
		K^\dagger
		\cup
		\left\{
		v+c(v,i,\gamma_0)z:
		v\in K^\dagger,\ z\in I_{z,K}
		\right\}$.
		Since, for $v\in K^\dagger$ and $z\in I_{z,K}$, we have $|c(v,i,\gamma_0)z|
		\le M_c\delta_{z,K}
		<
		\frac{d_{\dagger,1}}4$.
		Hence $\mathcal C_0\subset \operatorname{int}(K_1)$.

		Since $K^\flat\subset\operatorname{int}(K^\dagger)$ and
		$\phi_0^{x,i}=x$, uniform continuity of the ODE flow on compact sets gives $T_{\mathrm{pre}}>0$ such that
		\[
		\phi_s^{x,i}\in K^\dagger,
		\qquad \text{for}\quad
		x\in K^\flat,\quad 0\le s\le T_{\mathrm{pre}}.
		\]
		Since $\mathcal C_0\subset\operatorname{int}(K_1)$ is compact and
		$\phi_0^{v,i}=v$, there exists $T_{\mathrm{post}}>0$ such that
		\[
		\phi_r^{v,i}\in K_1,
		\qquad \text{for}\quad
		v\in\mathcal C_0,\quad 0\le r\le T_{\mathrm{post}}.
		\]
		Set 
		\[
		T_1:=T_{\mathrm{pre}}\wedge T_{\mathrm{post}}, \qquad 
		I_\tau(t):=\left[\frac t4,\frac{3t}{4}\right].
		\] for $0<t\le T_1$.
		For $(x,\tau,z) \in K^\flat\times I_\tau(t) \times I_{z,K}$,
		define
		\[
		y_-(x,\tau):=\phi_\tau^{x,i},
		\qquad
		y_+(x,\tau,z)
		:=
		y_-(x,\tau)+c(y_-(x,\tau),i,\gamma_0)z,
		\]
		and
		\[
		Y_s^{x,\tau,z;t}
		:=
		\begin{cases}
			\phi_s^{x,i}, & 0\le s<\tau,\\[1mm]
			\phi_{s-\tau}^{\,y_+(x,\tau,z),i}, & \tau\le s\le t.
		\end{cases}
		\]
		We claim that
		\begin{equation}
			\label{eq:compact-Y-in-K1-step1}
			Y_s^{x,\tau,z;t}\in K_1,
			\qquad \text{for}\quad0\le s\le t.
		\end{equation}
		For $0\le s<\tau$, then $s\le T_1\le T_{\mathrm{pre}}$, and hence
		$Y_s^{x,\tau,z;t}
		=
		\phi_s^{x,i}
		\in K^\dagger
		\subset K_1$.
		For $\tau\le s\le t$, since $\tau\le T_1\le T_{\mathrm{pre}}$, so $y_-(x,\tau)=\phi_\tau^{x,i}\in K^\dagger$.
		Therefore $y_+(x,\tau,z)\in\mathcal C_0$.
		Since $s-\tau\le t\le T_1\le T_{\mathrm{post}}$, we obtain
		$Y_s^{x,\tau,z;t}
		=
		\phi_{s-\tau}^{\,y_+(x,\tau,z),i}
		\in K_1$.
		This proves \eqref{eq:compact-Y-in-K1-step1}.
		
		Finally, we set
		$ d_*:=\operatorname{dist}(K_1,K_0^c)>0$.

		\medskip
		\noindent\textbf{Step 2: deterministic target intervals and deterministic margins.}
		
		For $t\in(0,T_1]$, $(x,\tau,z) \in K^\flat\times I_\tau(t) \times I_{z,K}$, define the deterministic endpoint map
		\[
		H_{x,\tau}^{(t)}(z):=Y_t^{x,\tau,z;t}
		=
		\phi_{t-\tau}^{\,y_+(x,\tau,z),i}.
		\]
		The map $(x,\tau,z)\mapsto H_{x,\tau}^{(t)}(z)$ is $C^1$. The variational
		equation for the ODE flow gives
		\[
		\partial_x\phi_r^{x,i}
		=
		\exp\left\{
		\int_0^r
		\partial_xb(\phi_s^{x,i},i,\alpha_0)\,ds
		\right\},
		\]
		and hence
		\[
		e^{-L_bT_*}
		\le
		\partial_x\phi_r^{x,i}
		\le
		e^{L_bT_*},
		\qquad 0\le r\le T_1,\quad x\in K_0 .
		\]
		Therefore, we have
		\begin{equation}
			\label{eq:compact-uniform-det-derivative-window}
			\bar\kappa_K
			\le
			\partial_zH_{x,\tau}^{(t)}(z)
			=
			\partial_x\phi_{t-\tau}^{y,i}\big|_{y=y_+(x,\tau,z)}
			c(y_-(x,\tau),i,\gamma_0)
			\le
			\bar\kappa_K',
		\end{equation}
		where
		\[
		\bar\kappa_K:=e^{-L_bT_1}c_{*,K}>0,
		\qquad
		\bar\kappa_K':=e^{L_bT_1}M_c<\infty .
		\]
		Note that the bounds in \eqref{eq:compact-uniform-det-derivative-window} hold uniformly
		over $x\in K^\flat$,
		$0<t\le T_1$,
		$\tau\in I_\tau(t)$,
		$z\in I_{z,K}$.
		In particular, $z\mapsto H_{x,\tau}^{(t)}(z)$ is strictly increasing on
		$I_{z,K}$.
		
		Let $\bar m_K:=\frac{\bar\kappa_K\delta_{z,K}}{16}$.
		There exists
		$T_2\in(0,T_1]$ such that
		\begin{equation}
			\label{eq:compact-flow-close-to-id}
			\sup_{y\in K}|\phi_h^{y,i}-y|<\bar m_K,
			\qquad 0<h\le T_2 .
		\end{equation}
		For $h\in(0,T_2]$ and $y\in K$, define
		\[
		\tau_h:=\frac h2,
		\qquad
		J_{y,h}^0
		:=
		H_{y,\tau_h}^{(h)}
		\left(
		\left(-\frac{\delta_{z,K}}4,\frac{\delta_{z,K}}4\right)
		\right).
		\]
		Since $z\mapsto H_{y,\tau_h}^{(h)}(z)$ is continuous and strictly increasing,
		\[
		J_{y,h}^0
		=
		\left(
		H_{y,\tau_h}^{(h)}\left(-\frac{\delta_{z,K}}4\right),
		H_{y,\tau_h}^{(h)}\left(\frac{\delta_{z,K}}4\right)
		\right),
		\]
		and hence $J_{y,h}^0$ is an open interval.
		
		By the derivative lower bound
		\eqref{eq:compact-uniform-det-derivative-window}, we have
		\[
		\begin{aligned}
			H_{y,\tau_h}^{(h)}\left(\frac{\delta_{z,K}}4\right)
			-
			H_{y,\tau_h}^{(h)}(0)
			=
			\int_0^{\delta_{z,K}/4}
			\partial_zH_{y,\tau_h}^{(h)}(z)\,dz                \ge
			\bar\kappa_K\frac{\delta_{z,K}}4
			=
			4\bar m_K,
		\end{aligned}
		\]
		and similarly
		\[
		\begin{aligned}
			H_{y,\tau_h}^{(h)}(0)
			-
			H_{y,\tau_h}^{(h)}\left(-\frac{\delta_{z,K}}4\right)
			\ge
			\bar\kappa_K\frac{\delta_{z,K}}4
			=
			4\bar m_K.
		\end{aligned}
		\]
		Therefore
		\[
		\left(
		H_{y,\tau_h}^{(h)}(0)-4\bar m_K,
		H_{y,\tau_h}^{(h)}(0)+4\bar m_K
		\right)
		\subset J_{y,h}^0.
		\]
		Using $H_{y,\tau_h}^{(h)}(0)=
		\phi_{h-\tau_h}^{\phi_{\tau_h}^{y,i},i}=\phi_h^{y,i}$ and
		\eqref{eq:compact-flow-close-to-id}, we get
		\[
		|H_{y,\tau_h}^{(h)}(0)-y|
		=
		|\phi_h^{y,i}-y|
		<
		\bar m_K.
		\]
		Hence $y\in J_{y,h}^0$.
		
		By the derivative lower bound \eqref{eq:compact-uniform-det-derivative-window}, we have
		\begin{equation}
			\label{eq:compact-base-endpoint-margin-left}
			H_{y,\tau_h}^{(h)}
			\left(-\frac{\delta_{z,K}}2\right)
			\le
			\inf J_{y,h}^0-4\bar m_K,
		\end{equation}
		and
		\begin{equation}
			\label{eq:compact-base-endpoint-margin-right}
			H_{y,\tau_h}^{(h)}
			\left(\frac{\delta_{z,K}}2\right)
			\ge
			\sup J_{y,h}^0+4\bar m_K .
		\end{equation}
		Note that Grönwall's inequality gives the Lipschitz estimate
		\begin{equation}
			\label{eq:flow-lipschitz-rev}
			|\phi_r^{y,i_0}-\phi_r^{y',i_0}|
			\le C_\varphi|y-y'|,
			\qquad 0\le r\le T_2,\quad y,y'\in K_0.
		\end{equation}
		Let $C_H:=M_b\bigl(1+C_\varphi(1+L_c r_0)\bigr)$.
		We choose $T_3\in(0,T_2]$ such that
		\begin{equation}
			\label{eq:compact-TH-choice}
			C_HT_3\le \bar m_K .
		\end{equation}
		Then, for every $h\in(0,T_3]$, every $\tau\in I_{\tau,h}$, and every
		$z\in I_{z,K}$, we have
		\begin{align}
			\left|
			H_{y,\tau}^{(h)}(z)-H_{y,\tau_h}^{(h)}(z)
			\right|                                             
			&\le
			\left|
			\phi_{h-\tau}^{\,y_+(y,\tau,z),i}
			-
			\phi_{h-\tau_h}^{\,y_+(y,\tau,z),i}
			\right|
			+
			\left|
			\phi_{h-\tau_h}^{\,y_+(y,\tau,z),i}
			-
			\phi_{h-\tau_h}^{\,y_+(y,\tau_h,z),i}
			\right|                                                    \notag\\
			&\le
			M_b|\tau-\tau_h|
			+
			C_\varphi
			\left|y_+(y,\tau,z)-y_+(y,\tau_h,z)\right|
			\notag\\
			&\le
			M_b|\tau-\tau_h|
			+
			C_\varphi\left|\phi_\tau^{y,i}-\phi_{\tau_h}^{y,i}\right|
			+
			C_\varphi\left|
			c(\phi_\tau^{y,i},i,\gamma_0)
			-
			c(\phi_{\tau_h}^{y,i},i,\gamma_0)
			\right|
			|z|                                                        \notag\\
			&\le
			M_b|\tau-\tau_h|
			+
			C_\varphi M_b|\tau-\tau_h|
			+
			C_\varphi L_c r_0 M_b|\tau-\tau_h|                                  \notag\\
			&=
			C_H |\tau-\tau_h| \le C_Hh
			\le
			\bar m_K. \label{eq:compact-H-tau-midpoint-bound}
		\end{align}
		
		Now fix $h\in(0,T_3]$ and $y\in K$. By continuity of
		$u\mapsto H_{u,\tau}^{(h)}(z)$
		uniformly over the compact set 
		$I_{\tau,h}\times
		\left\{-\frac{\delta_{z,K}}2,\frac{\delta_{z,K}}2\right\}$,
		there exists an open interval $U_{y,h}^0\ni y$ such that
		$\overline{U_{y,h}^0}\subset K^\flat$
		and, for every $u\in\overline{U_{y,h}^0}$, every $(\tau,z)\in I_{\tau,h}\times
		\left\{-\frac{\delta_{z,K}}2,\frac{\delta_{z,K}}2\right\}$, and
		$z\in\{-\delta_{z,K}/2,\delta_{z,K}/2\}$,
		\begin{equation}
			\label{eq:compact-u-perturb-bound}
			\left|
			H_{u,\tau}^{(h)}(z)-H_{y,\tau}^{(h)}(z)
			\right|
			\le
			\bar m_K .
		\end{equation}
		Therefore, by
		\eqref{eq:compact-base-endpoint-margin-left},
		\eqref{eq:compact-H-tau-midpoint-bound}, and
		\eqref{eq:compact-u-perturb-bound}, for every $u\in\overline{U_{y,h}^0}$ and every $\tau\in I_{\tau,h}$,
		\begin{align}
			H_{u,\tau}^{(h)}
			\left(-\frac{\delta_{z,K}}2\right)
			&\le
			H_{y,\tau_h}^{(h)}
			\left(-\frac{\delta_{z,K}}2\right)
			+
			\left|
			H_{y,\tau}^{(h)}
			\left(-\frac{\delta_{z,K}}2\right)
			-
			H_{y,\tau_h}^{(h)}
			\left(-\frac{\delta_{z,K}}2\right)
			\right|                                                    \notag\\
			&\quad+
			\left|
			H_{u,\tau}^{(h)}
			\left(-\frac{\delta_{z,K}}2\right)
			-
			H_{y,\tau}^{(h)}
			\left(-\frac{\delta_{z,K}}2\right)
			\right|                                                    \notag\\
			&\le
			\inf J_{y,h}^0
			-4\bar m_K
			+\bar m_K
			+\bar m_K                                                   \notag\\
			&=
			\inf J_{y,h}^0-2\bar m_K . \label{eq:compact-deterministic-left-margin}
		\end{align}
		Similarly, we have
		\begin{equation}
			\label{eq:compact-deterministic-right-margin}
			H_{u,\tau}^{(h)}
			\left(\frac{\delta_{z,K}}2\right)
			\ge
			\sup J_{y,h}^0+2\bar m_K .
		\end{equation}

		\medskip
		\noindent\textbf{Step 3: good residual event.}
		
		Choose $\widetilde{b}, \widetilde{c} \in C_b^2(\mathbb{R})$ satisfying
		\[
		\widetilde{b}(u) = b(u, i, \alpha_0),
		\qquad
		\widetilde{c}(u) = c(u, i, \gamma_0),
		\qquad u \in K_0.
		\]
		For $0 \le s \le t \le T_3$, let $\Phi_{s,t}(y)$ denote the stochastic flow of the modified equation
		\[
		Z_t = y + \int_s^t \widetilde{b}(Z_u) \, du
		+ \int_s^t \widetilde{c}(Z_{u-}) \, dL_u^{\mathrm{r}},
		\qquad t \in [s, T_3].
		\]
		Since $\widetilde{b}, \widetilde{c} \in C_b^2(\mathbb{R})$, standard variational arguments guarantee that the solution field $\Phi_{s,t}(y)$ is of class \(C^1\) in the initial value
		\(y\). We 
		write $\Upsilon_{s,t}(y):=\partial_y\Phi_{s,t}(y)$.
		Standard moment estimates for stochastic flows with bounded $C^2$ coefficients
		yield a constant $C_*>0$ such that, for all $0<h\le T_3$,
		\[
		\E\left[
		\sup_{\substack{0\le s\le t\le h\\ x\in K_1}}
		|\Phi_{s,t}(x)-\phi_{t-s}^{x,i}|^2
		\right]
		+
		\E\left[
		\sup_{\substack{0\le s\le t\le h\\ x\in K_1}}
		|\Upsilon_{s,t}(x)-\partial_x\phi_{t-s}^{x,i}|^2
		\right]
		\le
		C_*h .
		\]
		
		Since $(r,x)\mapsto \partial_x\phi_{r}^{x,i}$ is uniformly continuous on
		$[0,T_3]\times K_0$, there exists $\rho_{K}>0$ such that
		\begin{equation}
			\label{eq:compact-rho-u-choice}
			|\partial_x\phi_{r}^{x,i}-\partial_x\phi_{r}^{x',i}|
			\le
			\frac{\bar\kappa_K}{6M_c},
			\qquad
			0\le r\le T_3,
		\end{equation}
		whenever $x,x'\in K_0$ and $|x-x'|<\rho_{K}$.
		
		Let $\delta_{+,K}
		:=
		\operatorname{dist}(\mathcal C_0,K_1^c)$.
		Choose $\varepsilon_*>0$
		satisfying
		\begin{equation}
			\label{eq:compact-eps-star-conditions}
			\varepsilon_*
			\le
			\min\left(
			\frac{d_*}{4},
			\frac{\bar m_K}{4},
			1
			\right),
			\qquad
			(1+L_cr_0)\varepsilon_*
			\le
			\min\left(
			\frac{\delta_{+,K}}{2},
			\rho_{K}
			\right),
		\end{equation}
		and
		\begin{equation}
			\label{eq:compact-eps-star-conditions-2}
			\varepsilon_*+C_\varphi(1+L_cr_0)\varepsilon_*
			\le
			\min\left(\frac{d_*}{2},\bar m_K\right),
			\qquad
			M_c\varepsilon_*
			\le
			\frac{\bar\kappa_K}{6},
			\qquad
			M_uL_c\varepsilon_*
			\le
			\frac{\bar\kappa_K}{6},
		\end{equation}
		where
		$M_u:=\sup_{\substack{0\le t\le T_3\\ x\in K_0}}\partial_x\phi_{t}^{x,i}<\infty$.

		Now define
		\[
		h_K
		:=
		\min\left(
		T_3,
		\frac{\varepsilon_*^2}{4C_*}
		\right), \quad I_{\tau,h}:=\left[\frac h4,\frac{3h}{4}\right],
		\]
		for $h\in(0,h_K]$.
		Define
		\[
		\mathcal R_{K,h}
		:=
		\left\{
		\sup_{\substack{0\le s\le t\le h\\ x\in K_1}}
		|\Phi_{s,t}(x)-\phi_{t-s}^{x,i}|
		\le \varepsilon_*,
		\quad
		\sup_{\substack{0\le s\le t\le h\\ x\in K_1}}
		|\Upsilon_{s,t}(x)-\partial_x\phi_{t-s}^{x,i}|
		\le \varepsilon_*
		\right\}.
		\]
		By Markov's inequality,
		\[
		\mathbb P(\mathcal R_{K,h}^c)
		\le
		\frac{C_*h}{\varepsilon_*^2}
		\le
		\frac14,
		\]
		and hence $\mathbb P(\mathcal R_{K,h})\ge\frac34$.

		\medskip
		\noindent\textbf{Step 4: confinement in $K_0$.}
		
		Fix $h\in(0,h_K]$,
		$y\in K$,
		$u\in\overline{U_{y,h}^0}$,
		$\omega\in\mathcal R_{K,h}$,
		$\tau\in I_{\tau,h}$,
		$z\in I_{z,K}$.
		Define
		\[
		\widehat y_-(u,\tau,\omega)
		:=
		\Phi_{0,\tau-}(u)(\omega),
		\qquad
		\widehat y(u,\tau,\omega)
		:=
		\Phi_{0,\tau}(u)(\omega),
		\]
		and
		\[
		\widehat y_+(u,\tau,z,\omega)
		:=
		\widehat y(u,\tau,\omega)
		+
		c(\widehat y_-(u,\tau,\omega),i,\gamma_0)z .
		\]
		On $\mathcal R_{K,h}$, we have
		$|\widehat y(u,\tau,\omega)-y_-(u,\tau)|\le\varepsilon_*$,
		and
		$|\widehat y_-(u,\tau,\omega)-y_-(u,\tau)|\le\varepsilon_*$.
		Therefore
		\begin{align}
			\label{eq:compact-yplus-difference}
			|\widehat y_+(u,\tau,z,\omega)-y_+(u,\tau,z)|
			&\le
			|\widehat y-y_-|
			+
			|c(\widehat y_-,i,\gamma_0)-c(y_-,i,\gamma_0)|\,|z|
			\notag\\
			&\le
			\varepsilon_*+L_c r_0\varepsilon_*
			=
			(1+L_c r_0)\varepsilon_* .
		\end{align}
		By the choice of $\varepsilon_*$, the point $\widehat y_+(u,\tau,z,\omega)$ lies in $K_1$.

		Define the candidate solution
		\[
		\widetilde{X}_t^{u,\tau,z}(\omega)
		:=
		\begin{cases}
			\Phi_{0,t}(u)(\omega), & 0\le t<\tau,\\[1mm]
			\Phi_{\tau,t}\bigl(\widehat{y}_+(u,\tau,z,\omega)\bigr)(\omega),
			& \tau\le t\le h.
		\end{cases}
		\]
		This path is càdlàg with $\widetilde{X}_{\tau-}^{u,\tau,z}(\omega)
		=\widehat{y}_-(u,\tau,\omega)$, and one verifies directly that it satisfies
		the cutoff SDE:
		\begin{equation}
			\label{eq:cutoff-sde-check}
			\widetilde{X}_t^{u,\tau,z}
			=
			u+\int_0^t\widetilde{b}(\widetilde{X}_s^{u,\tau,z})\,ds
			+\int_0^t\widetilde{c}(\widetilde{X}_{s-}^{u,\tau,z})\,dL_s^{\mathrm{r}}
			+\mathbf{1}_{\{t\ge\tau\}}\,
			\widetilde{c}(\widetilde{X}_{\tau-}^{u,\tau,z})\,z,
			\quad 0\le t\le h.
		\end{equation}
		
		We now show $\widetilde{X}_t^{u,\tau,z}(\omega)\in K_0$ for all $t\in[0,h]$.
		For $0\le t<\tau$, on $\mathcal{R}_{K,h}$,
		\[
		|\widetilde{X}_t^{u,\tau,z}(\omega)-Y_t^{u,\tau,z;h}|
		=|\Phi_{0,t}(u)(\omega)-\phi_t^{u,i}|\le\varepsilon_*.
		\]
		For $\tau\le t\le h$, using $\mathcal{R}_{K,h}$,
		\eqref{eq:flow-lipschitz-rev}, and~\eqref{eq:compact-yplus-difference},
		\begin{align*}
			|\widetilde{X}_t^{u,\tau,z}(\omega)-Y_t^{u,\tau,z;h}|
			&\le
			|\Phi_{\tau,t}(\widehat{y}_+)(\omega)-\phi_{t-\tau}^{\widehat{y}_+,i}|
			+|\phi_{t-\tau}^{\widehat{y}_+,i}-\phi_{t-\tau}^{y_+,i}|
			\\
			&\le
			\varepsilon_*+C_\varphi|\widehat{y}_+-y_+|
			\le
			\varepsilon_*+C_\varphi(1+L_cr_0)\varepsilon_* \le d_*.
		\end{align*}
		By~\eqref{eq:compact-Y-in-K1-step1}, we have
		\[
		\widetilde{X}_t^{u,\tau,z}(\omega)\in K_0,\qquad 0\le t\le h.
		\]
		Since $\widetilde{X}^{u,\tau,z}(\omega)$ stays in $K_0$ and $\widetilde{b},
		\widetilde{c}$ agree with $b(\cdot,i,\alpha_0)$, $c(\cdot,i,\gamma_0)$
		on $K_0$, the process $\widetilde{X}^{u,\tau,z}(\omega)$ solves the original
		SDE~\eqref{eq:frozen-one-jump-sde}. By pathwise uniqueness,
		$\bar{X}_t^{u,\tau,z}(\omega)=\widetilde{X}_t^{u,\tau,z}(\omega)\in K_0$ for
		all $t\in[0,h]$, establishing conclusion~(1).

		\medskip
		\noindent\textbf{Step 5: endpoint derivative.}
		
		For $\omega\in\mathcal R_{K,h}$ define
		\[
		G_{u,\tau,\omega}^{(h)}(z):=
		\bar X_h^{u,\tau,z}(\omega)=
		\Phi_{\tau,h}
		\left(
		\widehat y_+(u,\tau,z,\omega)
		\right)(\omega).
		\]
		Hence $z\mapsto G_{u,\tau,\omega}^{(h)}(z)$ is $C^1$ on $I_{z,K}$, and
		\[
		\partial_zG_{u,\tau,\omega}^{(h)}(z)
		=
		\Upsilon_{\tau,h}(\widehat y_+(u,\tau,z,\omega))\,
		c(\widehat y_-(u,\tau,\omega),i,\gamma_0).
		\]
		On $\mathcal R_{K,h}$, comparing this derivative with
		\[
		\partial_zH_{u,\tau}^{(h)}(z)
		=
		\partial_x\phi_{h-\tau}^{y_+(u,\tau,z),i}
		\,
		c(y_-(u,\tau),i,\gamma_0)
		\]
		and using \eqref{eq:compact-yplus-difference}, \eqref{eq:compact-rho-u-choice}, \eqref{eq:compact-eps-star-conditions}, \eqref{eq:compact-eps-star-conditions-2}, we obtain
		\begin{align*}
			|\partial_z G_{u,\tau,\omega}^{(h)}(z)-\partial_z H_{u,\tau}^{(h)}(z)|
			&\le
			|\Upsilon_{\tau,h}(\widehat{y}_+)-\partial_x\phi_{h-\tau}^{\widehat{y}_+,i}|
			\cdot|c(\widehat{y}_-,i,\gamma_0)|\\
			&\quad+|\partial_x\phi_{h-\tau}^{\widehat{y}_+,i}-\partial_x\phi_{h-\tau}^{y_+,i}|
			\cdot|c(\widehat{y}_-,i,\gamma_0)|\\
			&\quad+|\partial_x\phi_{h-\tau}^{y_+,i}|
			\cdot|c(\widehat{y}_-,i,\gamma_0)-c(y_-,i,\gamma_0)|
			\\
			&\le M_c\varepsilon_* + \frac{\bar\kappa_K}{6M_c}\,M_c + M_u L_c \varepsilon_*
			\\
			&\le \frac{\bar\kappa_K}{2}.
		\end{align*}
		Combining this with
		\eqref{eq:compact-uniform-det-derivative-window}, we get
		\[
		\frac{\bar\kappa_K}{2}
		\le
		\partial_zG_{u,\tau,\omega}^{(h)}(z)
		\le
		\bar\kappa_K'+\frac{\bar\kappa_K}{2},
		\qquad z\in I_{z,K}.
		\]
		Set
		\[
		\kappa_{1,K}:=\frac{\bar\kappa_K}{2},
		\qquad
		\kappa_{2,K}:=\bar\kappa_K'+\frac{\bar\kappa_K}{2}.
		\]
		Then conclusion~\textup{(ii)} follows. In particular,
		$z\mapsto G_{u,\tau,\omega}^{(h)}(z)$ is strictly increasing on $I_{z,K}$.
		
		\medskip
		\noindent\textbf{Step 6: $J_{y,h}^0\subset G_{u,\tau,\omega}^{(h)}(I_{z,K})$.}
		
		Since $z\mapsto G_{u,\tau,\omega}^{(h)}(z)$ is strictly increasing, it is
		enough to show
		\[
		G_{u,\tau,\omega}^{(h)}
		\left(-\frac{\delta_{z,K}}2\right)
		<
		\inf J_{y,h}^0,
		\qquad
		G_{u,\tau,\omega}^{(h)}
		\left(\frac{\delta_{z,K}}2\right)
		>
		\sup J_{y,h}^0 .
		\]
		The same triangle-inequality estimate as in Step~5 gives
		\begin{equation}
			\label{eq:compact-G-H-bound}
			\left|
			G_{u,\tau,\omega}^{(h)}(z)
			-
			H_{u,\tau}^{(h)}(z)
			\right|
			\le
			\varepsilon_*+
			C_\varphi(1+L_c r_0)\varepsilon_*
			<
			\bar m_K,
			\qquad z\in I_{z,K}.
		\end{equation}
		By 
		\eqref{eq:compact-deterministic-left-margin} and
		\eqref{eq:compact-deterministic-right-margin}, we have
		\[
		G_{u,\tau,\omega}^{(h)}
		\left(-\frac{\delta_{z,K}}2\right)
		<
		\inf J_{y,h}^0, \quad 
		G_{u,\tau,\omega}^{(h)}
		\left(\frac{\delta_{z,K}}2\right)
		>
		\sup J_{y,h}^0.
		\]
		Therefore
		$J_{y,h}^0
		\subset
		G_{u,\tau,\omega}^{(h)}(I_{z,K})$,
		which proves conclusion~\textup{(iii)}.
		
		The proof is complete.
	\end{proof}

	The next lemma is the fixed-time local minorization.
	
	\begin{lem}
		\label{lem:fixed-h-local-minorization}
		Suppose Assumptions~\ref{ass:Levy} and
		Assumption~\ref{ass:erg-coeff} hold.
		Fix a regime $i\in\mathbb{S}$ and a compact interval $K\subset\mathbb{R}$.
		Then there exists $h_K>0$ such that, for every fixed
		$h\in(0,h_K]$, there exist constants
		$\delta_{K,h}>0$ and $\varepsilon_{K,h}>0$ with the following property:
		for every $x\in K$, there is an open interval $U_{x,h}\ni x$ such that
		\begin{equation}
			\label{eq:fixed-h-local-minorization}
			P_h\bigl((x',i),\,B\times\{i\}\bigr)
			\ge
			\varepsilon_{K,h}\,
			\lambda\!\bigl(B\cap J_{x,h}\bigr)
			\qquad
			\text{for all }x'\in U_{x,h},\;B\in\mathcal{B}(\mathbb{R}),
		\end{equation}
		where
		\[
		J_{x,h}:=(x-\delta_{K,h},\,x+\delta_{K,h}).
		\]
	\end{lem}
	
	\begin{proof}
		Fix a compact interval $K\subset\mathbb R$ and a regime $i\in\mathbb S$.
		
		Recall the decomposition $L=L^{\mathrm s}+L^{\mathrm r}$, and $\rho:=\nu^{\mathrm s}(\mathbb R)=2\kappa_0r_0$.
		We use Lemma~\ref{lem:one-jump-geometry}.
		There exists $h_K>0$ such that, for every fixed
		$h\in(0,h_K]$ and every $y\in K$, there exist open intervals $U_{y,h}^0\ni y$ and
		$J_{y,h}^0\ni y$, a compact interval $K_{y,h}^0$, constants
		$\delta_{z,y,h}>0$ and $0<\kappa_{1,y,h}\le\kappa_{2,y,h}<\infty$, and an
		event
		\[
		\mathcal R_{y,h}\in\sigma(L^{\mathrm r}),
		\qquad
		\mathbb P(\mathcal R_{y,h})\ge\frac34,
		\]
		such that, with
		\[
		I_{\tau,h}:=\left[\frac h4,\frac{3h}{4}\right],
		\qquad
		I_{z,y,h}:=[-\delta_{z,y,h},\delta_{z,y,h}]\subset(-r_0,r_0),
		\]
		the following hold for every $u\in \overline{U_{y,h}^0}$,
		$\omega\in\mathcal R_{y,h}$,
		$\tau\in I_{\tau,h}$,
		$z\in I_{z,y,h}$:
		The frozen one-jump equation in regime $i$ has a unique solution
		$\bar X^{u,\tau,z}$ on $[0,h]$, the path remains in $K_{y,h}^0$, and the
		endpoint map $G_{u,\tau,\omega}(z):=\bar X_h^{u,\tau,z}(\omega)$
		is of class $C^1$, strictly increasing on $I_{z,y,h}$, and satisfies $\kappa_{1,y,h}\le\partial_zG_{u,\tau,\omega}(z)\le\kappa_{2,y,h}$.
		Moreover, $J_{y,h}^0\subset G_{u,\tau,\omega}(I_{z,y,h})$.

		We first prove a pointwise local minorization. Fix $y\in K$, fix
		$u\in U_{y,h}^0$, and let $A\subset J_{y,h}^0$ be Borel. Define
		\[
		\mathcal N_h
		:=
		\{\Lambda_s=i\ \text{for all }0\le s\le h\}.
		\]
		Let $N^{\mathrm s}_{(0,h]}$ be the number of jumps of $L^{\mathrm s}$ on
		$(0,h]$, and on the event $\{N^{\mathrm s}_{(0,h]}=1\}$ let
		$\tau^{\mathrm s}$ and $Z$ denote the unique jump time and jump size. Then
		\[
		\begin{aligned}
			P_h\bigl((u,i),A\times\{i\}\bigr)
			&\ge
			\Prob_{(u,i)}
			\Bigl(
			\mathcal R_{y,h},\
			N^{\mathrm s}_{(0,h]}=1,\
			\tau^{\mathrm s}\in I_{\tau,h},\
			Z\in I_{z,y,h},                                      \\
			&\hspace{7em}
			\mathcal N_h,\
			G_{u,\tau^{\mathrm s},\omega}(Z)\in A
			\Bigr).
		\end{aligned}
		\]
		Since $\mathcal R_{y,h}\in\sigma(L^{\mathrm r})$, the tower property gives
		\[
		\begin{aligned}
			&\Prob_{(u,i)}
			\Bigl(
			\mathcal R_{y,h},\
			N^{\mathrm s}_{(0,h]}=1,\
			\tau^{\mathrm s}\in I_{\tau,h},\
			Z\in I_{z,y,h},\
			\mathcal N_h,\
			G_{u,\tau^{\mathrm s},\omega}(Z)\in A
			\Bigr)                                                    \\
			&\quad =
			\E\Bigl[
			\mathbf 1_{\mathcal R_{y,h}}\,
			\Prob_{(u,i)}
			\Bigl(
			N^{\mathrm s}_{(0,h]}=1,\
			\tau^{\mathrm s}\in I_{\tau,h},\
			Z\in I_{z,y,h},\
			\mathcal N_h,\
			G_{u,\tau^{\mathrm s},\omega}(Z)\in A
			\,\Big|\,
			\sigma(L^{\mathrm r})
			\Bigr)
			\Bigr].
		\end{aligned}
		\]
		
		Conditional on the event \(\{N^{\mathrm s}_{(0,h]}=1\}\), the jump time $\tau^{\mathrm s}$ is uniformly distributed on \((0,h]\), the jump size \(Z\) has density \((2r_0)^{-1}\mathbf 1_{(-r_0,r_0)}(z)\), and
		\(\tau^{\mathrm s}\) and \(Z\) are independent.
		Therefore, for every Borel sets
		\(C\subset(0,h]\) and \(D\subset\R\),
		\begin{align}
			\Prob\!\left(
			N^{\mathrm s}_{(0,h]}=1,\;
			\tau^{\mathrm s}\in C,\;
			Z\in D
			\right)&=
			\Prob\bigl(N^{\mathrm s}_{(0,h]}=1\bigr)\,
			\Prob\left(
			\tau^{\mathrm s}\in C
			\,\middle|\,
			N^{\mathrm s}_{(0,h]}=1
			\right)\,
			\Prob\left(
			Z\in D
			\,\middle|\,
			N^{\mathrm s}_{(0,h]}=1
			\right)
			\notag\\
			&=
			e^{-\rho h}(\rho h)\cdot \frac{\lambda(C)}{h}\cdot
			\int_D \frac{1}{2r_0}\mathbf 1_{(-r_0,r_0)}(z)\,dz
			\notag\\
			&=
			e^{-\rho h}
			\int_C\int_D
			\frac{\rho}{2r_0}\,
			\mathbf 1_{(-r_0,r_0)}(z)\,dz\,d\tau.
			\notag
		\end{align}
		For $(\tau,z)\in I_{\tau,h}\times I_{z,y,h}$, set
		\[
		\lambda_s^{u,\tau,z}
		:=
		\sum_{k\neq i}
		q_{ik}\bigl(\bar X_s^{u,\tau,z}\bigr),
		\qquad 0\le s\le h.
		\]
		The no-switch probability along the frozen one-jump path is
		\[
		\Prob_{(u,i)}
		\bigl(
		\mathcal N_h
		\,\big|\,
		\tau^{\mathrm s}=\tau,\,
		Z=z,\,
		\sigma(L^{\mathrm r})
		\bigr)
		=
		\exp\left(
		-\int_0^h\lambda_s^{u,\tau,z}\,ds
		\right).
		\]
		Since \(L^{\mathrm s}\) and \(L^{\mathrm r}\) are independent,
		we have,
		\begin{align}
			&\Prob_{(u,i)}
			\Bigl(
			N^{\mathrm s}_{(0,h]}=1,\
			\tau^{\mathrm s}\in I_{\tau,h},\
			Z\in I_{z,y,h},\
			\mathcal N_h,\
			G_{u,\tau^{\mathrm s},\omega}(Z)\in A
			\,\Big|\,
			\sigma(L^{\mathrm r})
			\Bigr)
			\notag\\
			&\qquad=
			e^{-\rho h}\frac{\rho}{2r_0}
			\int_{I_{\tau,h}}\int_{I_{z,y,h}}
			\ind_A\!\bigl(G_{u,\tau,\omega}(z)\bigr)\,
			\Prob_{(u,i)}\!\left(
			\mathcal N_h
			\,\middle|\,
			\tau^{\mathrm s}=\tau,\,
			Z=z,\,
			\sigma(L^{\mathrm r})
			\right)
			dz\,d\tau
			\notag\\
			&\qquad=
			e^{-\rho h}\frac{\rho}{2r_0}
			\int_{I_{\tau,h}}\int_{I_{z,y,h}}
			\ind_A\!\bigl(G_{u,\tau,\omega}(z)\bigr)\,
			\exp\!\Bigl(
			-\int_0^{h}\lambda_s^{u,\tau,z}\,ds
			\Bigr)
			dz\,d\tau. \notag
		\end{align}
		Therefore,
		\[
		\begin{aligned}
			P_h\bigl((u,i),A\times\{i\}\bigr)
			&\ge
			e^{-\rho h}\frac{\rho}{2r_0}
			\E\Biggl[
			\mathbf 1_{\mathcal R_{y,h}}
			\int_{I_{\tau,h}}\int_{I_{z,y,h}}
			\mathbf 1_A\bigl(G_{u,\tau,\omega}(z)\bigr)       \\
			&\hspace{8em}
			\times
			\exp\left(
			-\int_0^h\lambda_s^{u,\tau,z}\,ds
			\right)
			dz\,d\tau
			\Biggr].
		\end{aligned}
		\]
		
		On $\mathcal R_{y,h}$, the frozen one-jump path remains in $K_{y,h}^0$.
		Define
		\[
		q_{y,h}^*
		:=
		\sup_{\substack{v\in K_{y,h}^0\\ k\in\mathbb S}}
		\sum_{\ell\neq k}q_{k\ell}(v)
		<\infty .
		\]
		Then
		\[
		\exp\left(
		-\int_0^h\lambda_s^{u,\tau,z}\,ds
		\right)
		\ge
		e^{-q_{y,h}^*h}
		\]
		on $\mathcal R_{y,h}$. Hence
		\[
		\begin{aligned}
			P_h\bigl((u,i),A\times\{i\}\bigr)
			&\ge
			e^{-(\rho+q_{y,h}^*)h}\frac{\rho}{2r_0}
			\E\Biggl[
			\mathbf 1_{\mathcal R_{y,h}}
			\int_{I_{\tau,h}}
			\lambda\bigl(
			G_{u,\tau,\omega}^{-1}(A)\cap I_{z,y,h}
			\bigr)
			d\tau
			\Biggr].
		\end{aligned}
		\]
		
		Fix $\omega\in\mathcal R_{y,h}$ and $\tau\in I_{\tau,h}$. 
		We have
		$G_{u,\tau,\omega}:I_{z,y,h}\to\R$ is strictly increasing, of class $C^1$, and satisfies
		$\partial_zG_{u,\tau,\omega}\le\kappa_{2,y,h}$, and
		\[
		J_{y,h}^0\subset G_{u,\tau,\omega}(I_{z,y,h}).
		\]
		Let $H:=G_{u,\tau,\omega}^{-1}\big|_{J_{y,h}^0}:J_{y,h}^0\to I_{z,y,h}$ be the inverse map. Then
		the inverse-function theorem gives, for every Borel set
		$A\subset J_{y,h}^0$,
		\[
		\lambda\bigl(
		G_{u,\tau,\omega}^{-1}(A)\cap I_{z,y,h}
		\bigr)
		=\int_A H'(y)\,dy
		\ge
		\frac{1}{2\kappa_{2,y,h}}\lambda(A).
		\]
		Consequently,
		\[
		\begin{aligned}
			P_h\bigl((u,i),A\times\{i\}\bigr)
			&\ge
			e^{-(\rho+q_{y,h}^*)h}\frac{\rho}{2r_0}
			\mathbb P(\mathcal R_{y,h})
			\lambda(I_{\tau,h})
			\frac{1}{2\kappa_{2,y,h}}
			\lambda(A)                                                \\
			&\ge
			e^{-(\rho+q_{y,h}^*)h}
			\frac{3\rho h}{32r_0\kappa_{2,y,h}}
			\lambda(A).
		\end{aligned}
		\]
		Set
		\[
		\varepsilon_{y,h}^0
		:=
		e^{-(\rho+q_{y,h}^*)h}
		\frac{3\rho h}{32r_0\kappa_{2,y,h}}
		>0 .
		\]
		Taking $A=B\cap J_{y,h}^0$ gives the pointwise estimate
		\begin{equation}
			\label{eq:pointwise-fixed-h-minorization-from-geometry}
			P_h\bigl((u,i),B\times\{i\}\bigr)
			\ge
			\varepsilon_{y,h}^0
			\lambda(B\cap J_{y,h}^0),
			\qquad
			u\in U_{y,h}^0,\quad B\in\mathcal B(\mathbb R).
		\end{equation}
		
		It remains to make the interval radius and the minorization constant uniform
		over the compact set $K$, for this fixed value of $h$. For every $y\in K$,
		choose an open interval $V_{y,h}$ such that
		\[
		y\in V_{y,h},
		\qquad
		\overline{V_{y,h}}
		\subset
		U_{y,h}^0\cap J_{y,h}^0 .
		\]
		The family $\{V_{y,h}:y\in K\}$ covers $K$. Since $K$ is compact, choose
		points $y_1,\dots,y_M\in K$ such that
		\[
		K\subset\bigcup_{\ell=1}^M V_{y_\ell,h}.
		\]
		Write
		\[
		V_{\ell,h}:=V_{y_\ell,h},
		\qquad
		U_{\ell,h}^0:=U_{y_\ell,h}^0,
		\qquad
		J_{\ell,h}^0:=J_{y_\ell,h}^0,
		\qquad
		\varepsilon_{\ell,h}^0:=\varepsilon_{y_\ell,h}^0 .
		\]
		Since $\overline{V_{\ell,h}}\subset J_{\ell,h}^0$ and $J_{\ell,h}^0$ is open,
		\[
		r_{\ell,h}
		:=
		\operatorname{dist}
		\bigl(
		\overline{V_{\ell,h}},
		(J_{\ell,h}^0)^c
		\bigr)
		>0.
		\]
		Define
		\[
		\delta_{K,h}
		:=
		\frac12\min_{1\le\ell\le M}r_{\ell,h}>0,
		\qquad
		\varepsilon_{K,h}
		:=
		\min_{1\le\ell\le M}\varepsilon_{\ell,h}^0>0.
		\]
		
		Now fix $x\in K$. Choose an index $\ell(x)\in\{1,\dots,M\}$ such that
		\[
		x\in V_{\ell(x),h}.
		\]
		Set
		\[
		U_{x,h}:=V_{\ell(x),h},
		\qquad
		J_{x,h}:=(x-\delta_{K,h},x+\delta_{K,h}).
		\]
		Then $U_{x,h}$ is an open interval containing $x$. If $x'\in U_{x,h}$, then
		\[
		x'\in V_{\ell(x),h}\subset U_{\ell(x),h}^0.
		\]
		Moreover, by the definition of $\delta_{K,h}$,
		\[
		J_{x,h}\subset J_{\ell(x),h}^0 .
		\]
		Applying \eqref{eq:pointwise-fixed-h-minorization-from-geometry} with
		$y=y_{\ell(x)}$ gives, for every $B\in\mathcal B(\mathbb R)$,
		\[
		\begin{aligned}
			P_h\bigl((x',i),B\times\{i\}\bigr)
			&\ge
			\varepsilon_{\ell(x),h}^0
			\lambda(B\cap J_{\ell(x),h}^0)                         \\
			&\ge
			\varepsilon_{K,h}
			\lambda(B\cap J_{x,h}).
		\end{aligned}
		\]
		This proves \eqref{eq:fixed-h-local-minorization}.
	\end{proof}

	We also need a one-step switching lemma at fixed time.
	
	\begin{lem}
		\label{lem:fixed-h-switch}
		Suppose Assumptions~\ref{ass:Levy} and
		Assumption~\ref{ass:erg-coeff} hold.
		Fix a compact interval $K\subset\mathbb R$ and distinct regimes
		$i,j\in\mathbb S$. For every $\eta>0$ there exists
		$h_{K,\eta}^{ij}>0$ such that, for every fixed
		$h\in(0,h_{K,\eta}^{ij}]$, there exists a constant
		$\xi_{K,\eta,h}^{ij}>0$ satisfying
		\begin{equation}
			\label{eq:fixed-h-switch}
			P_h\bigl((x,i),(x-\eta,x+\eta)\times\{j\}\bigr)
			\ge
			\xi_{K,\eta,h}^{ij},
			\qquad x\in K .
		\end{equation}
	\end{lem}
	
	\begin{proof}
		Set $K^\eta:=\{u\in\mathbb R:\operatorname{dist}(u,K)\le \eta\}$.
		Then $K^\eta$ is compact. Write
		\[
		q_k(u):=\sum_{\ell\neq k}q_{k\ell}(u),
		\qquad k\in\mathbb S.
		\]
		By Assumption~\ref{ass:erg-coeff},
		\[
		q_{K^\eta}^*
		:=
		\sup_{\substack{u\in K^\eta\\k\in\mathbb S}}
		q_k(u)
		<\infty, \quad  
		\underline q_{K^\eta}^{ij}
		:=
		\inf_{u\in K^\eta}q_{ij}(u)
		>0 .
		\]
		
		For $k\in\{i,j\}$ and $y\in\mathbb R$, let $X^{k,y}$ denote the solution of
		the continuous equation with the regime frozen at $k$ and initial condition
		$X^{k,y}_0=y$:
		\[
		dX_t^{k,y}
		=
		b(X_t^{k,y},k,\alpha_0)\,dt
		+
		c(X_{t-}^{k,y},k,\gamma_0)\,dL_t .
		\]
		By standard Doob--Gronwall estimate, choose $h_{K,\eta}^{ij}>0$ such that, for every
		$h\in(0,h_{K,\eta}^{ij}]$,
		\begin{equation}
			\label{eq:frozen-localization-switch}
			\inf_{\substack{y\in K^{\frac{\eta}{2}}\\ k\in\{i,j\}}}
			\mathbb P\left(
			\sup_{0\le r\le h}|X_r^{k,y}-y|<\frac{\eta}{4}
			\right)
			\ge \frac12 .
		\end{equation}
		
		For $y\in K^{\eta/2}$ and $t\in[0,h]$, define
		\[
		R_j(y,t)
		:=
		\Prob_{(y,j)}
		\left(
		\sup_{0\le r\le t}|X_r-y|<\frac{\eta}{4},
		\ \Lambda_r=j\ \text{for all }r\in[0,t]
		\right).
		\]
		Under $\Prob_{(y,j)}$, on the event $\{\Lambda_r=j\text{ for all }
		r\in[0,t]\}$, the continuous component is the frozen-regime process
		$X^{j,y}$ up to time $t$. 
		Let $\mathcal G_t^{i,x}:=\sigma\{X_r^{i,x}:0\le r\le t\}$, $ N_{ij}(t) :=\sum_{0<s\le t}\mathbf 1_{\{\Lambda_{s-}=i,\ \Lambda_s=j\}}$ and $\tau_1$, $\tau_2$ be the first and second switching time of $\Lambda$.
		Therefore,
		\[
		\begin{aligned}
			R_j(y,t)
			&=
			\E\left[
			\mathbf 1_{\left\{
				\sup_{0\le r\le t}|X_r^{j,y}-y|<\eta/4
				\right\}}
			\Prob\left(
			\tau_1>t
			\,\middle|\,
			\mathcal G_t^{j,y}
			\right)
			\right]                                                     \\
			&=
			\E\left[
			\mathbf 1_{\left\{
				\sup_{0\le r\le t}|X_r^{j,y}-y|<\eta/4
				\right\}}
			\exp\left\{
			-\int_0^t q_j(X_r^{j,y})\,dr
			\right\}
			\right].
		\end{aligned}
		\]
		On the event $\left\{\sup_{0\le r\le t}|X_r^{j,y}-y|<\frac{\eta}{4}\right\}$,
		we have $X_r^{j,y}\in K^\eta$ for all $0\le r\le t$. Hence, by \eqref{eq:frozen-localization-switch},
		\[
		\begin{aligned}
			R_j(y,t)
			&\ge
			e^{-q_{K^\eta}^*t}
			\Prob\left(
			\sup_{0\le r\le t}|X_r^{j,y}-y|<\frac{\eta}{4}
			\right)                                                   \\
			&\ge
			\frac12 e^{-q_{K^\eta}^*t},
		\end{aligned}
		\]
		for $y\in K^{\eta/2}$ and $0\le t\le h$.

		Fix $h\in(0,h_{K,\eta}^{ij}]$ and $x\in K$. For
		$s\in[h/3,2h/3]$, define
		\[
		A_{i,x}(s)
		:=
		\left\{
		\sup_{0\le r\le s}|X_r^{i,x}-x|<\frac{\eta}{4}
		\right\}.
		\]
		Then define 
		\[
		D_{x,h}:= \left\{
		\tau_1 \in [h/3,2h/3], \, \Lambda_{\tau_1}=j,\, \tau_2 > h,\, \sup_{0\le r\le \tau_1}|X_r-x|<\frac{\eta}{4}, \sup_{\tau_1\le r\le h}|X_r-X_{\tau_1}|<\frac{\eta}{4}
		\right\}.
		\]
		We estimate $\Prob_{(x,i)}(D_{x,h})$. Using strong Markov property, we have
		\[
		\begin{aligned}
			\Prob_{(x,i)}(D_{x,h})
			&=
			\E_{(x,i)}
			\left[
			\int_{h/3}^{2h/3}
			\mathbf 1_{\left\{\tau_1\ge s,\,\sup_{0\le r<s}|X_r-x|<\eta/4\right\}}
			\mathbf 1_{\left\{
				\tau_2>h,\,
				\sup_{s\le r\le h}|X_r-X_s|<\eta/4
				\right\}}
			\,dN_{ij}(s)
			\right]                                                     \\
			&=
			\E_{(x,i)}
			\left[
			\int_{h/3}^{2h/3}
			\mathbf 1_{\left\{\tau_1\ge s,\,\sup_{0\le r<s}|X_r-x|<\eta/4\right\}}
			R_j(X_{s-},h-s)
			\,dN_{ij}(s)
			\right]                                                     \\
			&=
			\E_{(x,i)}
			\left[
			\int_{h/3}^{2h/3}
			\mathbf 1_{\left\{\tau_1\ge s,\,\sup_{0\le r<s}|X_r-x|<\eta/4\right\}}
			R_j(X_{s-},h-s)
			q_{ij}(X_{s-})
			\,ds
			\right]
			\\
			&=
			\int_{h/3}^{2h/3}
			\E\left[
			\mathbf 1_{A_{i,x}(s)}
			\mathbf 1_{\{\tau_1>s\}}
			q_{ij}(X_s^{i,x})
			R_j(X_s^{i,x},h-s)
			\right]\,ds                                               \\
			&=
			\int_{h/3}^{2h/3}
			\E\left[
			\mathbf 1_{A_{i,x}(s)}
			\E\!\left[
			\mathbf 1_{\{\tau_1>s\}}
			\,\middle|\,
			\mathcal G_s^{i,x}
			\right]
			q_{ij}(X_s^{i,x})
			R_j(X_s^{i,x},h-s)
			\right]\,ds                                               \\
			&=
			\int_{h/3}^{2h/3}
			\E\left[
			\mathbf 1_{A_{i,x}(s)}
			\exp\left\{
			-\int_0^s q_i(X_r^{i,x})\,dr
			\right\}
			q_{ij}(X_s^{i,x})
			R_j(X_s^{i,x},h-s)
			\right]\,ds .
		\end{aligned}
		\]
		
		For $s\in[h/3,2h/3]$, on the event $A_{i,x}(s)$, we have $X_r^{i,x}\in K^\eta$ for $0\le r\le s$ and $X_s^{i,x}\in K^{\eta/4}\subset K^{\eta/2}$.
		Thus, by the estimate for $R_j$ and \eqref{eq:frozen-localization-switch}
		\[
		\begin{aligned}
			\Prob_{(x,i)}(D_{x,h})
			&\ge
			\int_{h/3}^{2h/3}
			\Prob(A_{i,x}(s))\,
			e^{-q_{K^\eta}^*s}\,
			\underline q_{K^\eta}^{ij}\,
			\frac12 e^{-q_{K^\eta}^*(h-s)}
			\,ds                                                        \\
			&\ge
			\frac14\,
			\underline q_{K^\eta}^{ij}
			\int_{h/3}^{2h/3}
			e^{-q_{K^\eta}^*h}
			\,ds                                                        \\
			&=
			\frac14\,
			\underline q_{K^\eta}^{ij}
			\frac{h}{3}
			e^{-q_{K^\eta}^*h}.
		\end{aligned}
		\]
		It is easy to check that
		\[
		D_{x,h}
		\subset
		\left\{
		(X_h,\Lambda_h)\in (x-\eta,x+\eta)\times\{j\}
		\right\}.
		\]
		Consequently,
		\[
		\begin{aligned}
			P_h\bigl((x,i),(x-\eta,x+\eta)\times\{j\}\bigr)
			&=
			\Prob_{(x,i)}
			\left(
			(X_h,\Lambda_h)\in(x-\eta,x+\eta)\times\{j\}
			\right)                                                   \\
			&\ge
			\Prob_{(x,i)}(D_{x,h})                                   \\
			&\ge
			\frac14\,
			\underline q_{K^\eta}^{ij}
			\frac{h}{3}
			e^{-q_{K^\eta}^*h}.
		\end{aligned}
		\]
		For the fixed value of $h\in(0,h_{K,\eta}^{ij}]$, set
		\[
		\xi_{K,\eta,h}^{ij}
		:=
		\frac14\,
		\underline q_{K^\eta}^{ij}
		\frac{h}{3}
		e^{-q_{K^\eta}^*h}
		>0.
		\]
		
	\end{proof}

	\bigskip
	


\begin{thebibliography}{10}
		
		\bibitem{Aalen1978}
		Odd~O. Aalen, \emph{Nonparametric inference for a family of counting
			processes}, The Annals of Statistics \textbf{6} (1978), no.~4, 701--726.
		
		\bibitem{Andersen1993}
		Per~K. Andersen, {\O}rnulf Borgan, Richard~D. Gill, and Niels Keiding,
		\emph{Statistical models based on counting processes}, Springer, New York,
		1993.
		
		\bibitem{applebaum2009levy}
		David Applebaum, \emph{L{\'e}vy processes and stochastic calculus}, Cambridge
		university press, 2009.
		
		\bibitem{AzaisMullerGueudin2016}
		Romain Aza{\"i}s and Aur{\'e}lie Muller-Gueudin, \emph{Optimal choice among a
			class of nonparametric estimators of the jump rate for
			piecewise-deterministic {Markov} processes}, Electronic Journal of Statistics
		\textbf{10} (2016), no.~2, 3648--3692.
		
		\bibitem{BarndorffNielsen1997}
		Ole~E. Barndorff-Nielsen, \emph{Normal inverse gaussian distributions and
			stochastic volatility modelling}, Scandinavian Journal of Statistics
		\textbf{24} (1997), no.~1, 1--13.
		
		\bibitem{bhattacharya1982functional}
		Rabi~N Bhattacharya, \emph{On the functional central limit theorem and the law
			of the iterated logarithm for markov processes}, Zeitschrift f{\"u}r
		Wahrscheinlichkeitstheorie und verwandte Gebiete \textbf{60} (1982), no.~2,
		185--201.
		
		\bibitem{bladt2005statistical}
		Mogens Bladt and Michael S{\o}rensen, \emph{Statistical inference for
			discretely observed {M}arkov jump processes}, Journal of the Royal
		Statistical Society Series B: Statistical Methodology \textbf{67} (2005),
		no.~3, 395--410.
		
		\bibitem{Borgan1984}
		{\O}rnulf Borgan, \emph{Maximum likelihood estimation in parametric counting
			process models, with applications to censored failure time data},
		Scandinavian Journal of Statistics \textbf{11} (1984), no.~1, 1--16.
		
		\bibitem{ContTankov2004}
		Rama Cont and Peter Tankov, \emph{Financial modelling with jump processes},
		Chapman \& Hall/CRC Financial Mathematics Series, Chapman \& Hall/CRC, Boca
		Raton, 2004.
		
		\bibitem{down1995exponential}
		Douglas Down, Sean~P Meyn, and Richard~L Tweedie, \emph{Exponential and uniform
			ergodicity of markov processes}, The Annals of Probability \textbf{23}
		(1995), no.~4, 1671--1691.
		
		\bibitem{friedman2006stochastic}
		Avner Friedman, \emph{Stochastic differential equations and applications},
		Courier Corporation, 2006.
		
		\bibitem{Gobet2002}
		Emmanuel Gobet, \emph{Lan property for ergodic diffusions with discrete
			observations}, Annales de l'I.H.P. Probabilités et statistiques \textbf{38}
		(2002), no.~5, 711--737 (eng).
		
		\bibitem{Hibbah2020}
		El~Houcine Hibbah, Hamid El~Maroufy, Christiane Fuchs, and Taib Ziad, \emph{An
			{MCMC} computational approach for a continuous time state-dependent regime
			switching diffusion process}, Journal of Applied Statistics \textbf{47}
		(2020), no.~8, 1354--1374.
		
		\bibitem{Kessler1997}
		Mathieu Kessler, \emph{Estimation of an ergodic diffusion from discrete
			observations}, Scandinavian Journal of Statistics \textbf{24} (1997), no.~2,
		211--229.
		
		\bibitem{Kes97}
		\bysame, \emph{Estimation of an ergodic diffusion from discrete observations},
		Scand. J. Statist. \textbf{24} (1997), no.~2, 211--229.
		
		\bibitem{kessler2012statistical}
		Mathieu Kessler, Alexander Lindner, and Michael S{\o}rensen, \emph{Statistical
			methods for stochastic differential equations}, Monographs on Statistics and
		Applied Probability \textbf{124} (2012), 7--12.
		
		\bibitem{KrellSchmisser2021}
		Nathalie Krell and Emeline Schmisser, \emph{Nonparametric estimation of jump
			rates for a specific class of piecewise deterministic {Markov} processes},
		Bernoulli \textbf{27} (2021), no.~4, 2362--2388.
		
		\bibitem{Kulik2009}
		Alexey~M. Kulik, \emph{Exponential ergodicity of the solutions to {SDE}'s with
			a jump noise}, Stochastic Processes and their Applications \textbf{119}
		(2009), no.~2, 602--632.
		
		\bibitem{mao2006stochastic}
		Xuerong Mao and Chenggui Yuan, \emph{Stochastic differential equations with
			{M}arkovian switching}, Imperial college press, 2006.
		
		\bibitem{Masuda2007}
		Hiroki Masuda, \emph{Ergodicity and exponential \(\beta\)-mixing bounds for
			multidimensional diffusions with jumps}, Stochastic Processes and their
		Applications \textbf{117} (2007), no.~1, 35--56.
		
		\bibitem{Mas13as}
		\bysame, \emph{Convergence of {G}aussian quasi-likelihood random fields for
			ergodic {L}\'evy driven {SDE} observed at high frequency}, Ann. Statist.
		\textbf{41} (2013), no.~3, 1593--1641.
		
		\bibitem{MariMari2023}
		Carlo Mari and Emiliano Mari, \emph{Deep learning based regime-switching
			models of energy commodity prices}, Energy Systems \textbf{14} (2023),
		913--934.
		
		\bibitem{MasudaUehara2017}
		Hiroki Masuda and Yuma Uehara, \emph{Two-step estimation of ergodic {L}{\'e}vy
			driven {SDE}}, Statistical Inference for Stochastic Processes \textbf{20}
		(2017), no.~1, 105--137.
		
		\bibitem{MerlevedePeligrad2013}
		Florence Merlev\`ede and Magda Peligrad, \emph{Rosenthal-type inequalities for
			the maximum of partial sums of stationary processes and examples}, The Annals
		of Probability \textbf{41} (2013), no.~2, 914--960.
		
		\bibitem{meyn1992criteria}
		Sean~P. Meyn and R.~L. Tweedie, \emph{Stability of markovian processes i:
			Criteria for discrete-time chains}, Advances in Applied Probability
		\textbf{24} (1992), no.~3, 542--574.
		
		\bibitem{meyn1993stability}
		Sean~P Meyn and Richard~L Tweedie, \emph{Stability of markovian processes iii:
			Foster--lyapunov criteria for continuous-time processes}, Advances in Applied
		Probability \textbf{25} (1993), no.~3, 518--548.
		
		\bibitem{meyn2012markov}
		\bysame, \emph{Markov chains and stochastic stability}, Springer Science \&
		Business Media, 2012.
		
		\bibitem{Sat99}
		Ken-iti Sato, \emph{L\'evy processes and infinitely divisible distributions},
		Cambridge Studies in Advanced Mathematics, vol.~68, Cambridge University
		Press, Cambridge, 1999, Translated from the 1990 Japanese original, Revised
		by the author.
		
		\bibitem{UchidaYoshida2012}
		Masayuki Uchida and Nakahiro Yoshida, \emph{Adaptive estimation of an ergodic
			diffusion process based on sampled data}, Stochastic Processes and their
		Applications \textbf{122} (2012), no.~8, 2885--2924.
		
		\bibitem{Viennet1997}
		Gabrielle Viennet, \emph{Inequalities for absolutely regular sequences:
			application to density estimation}, Probability Theory and Related Fields
		\textbf{107} (1997), no.~4, 467--492.
		
		\bibitem{Xi2009}
		F.~B. Xi, \emph{Asymptotic properties of jump-diffusion processes with
			state-dependent switching}, Stochastic Processes and their Applications
		\textbf{119} (2009), no.~7, 2198--2221.
		
		\bibitem{XiYin2011}
		F.~B. Xi and G.~Yin, \emph{Jump-diffusions with state-dependent switching:
			existence and uniqueness, feller property, linearization, and uniform
			ergodicity}, Science China Mathematics \textbf{54} (2011), no.~12,
		2651--2667.
		
		\bibitem{XiYin2017}
		Fubao Xi and G.~George Yin, \emph{On feller and strong feller properties and
			exponential ergodicity of regime-switching jump diffusion processes with
			countable regimes}, SIAM Journal on Control and Optimization \textbf{55}
		(2017), no.~3, 1789--1818.
		
		\bibitem{yin2009hybrid}
		G~George Yin and Chao Zhu, \emph{Hybrid switching diffusions: properties and
			applications}, vol.~63, Springer Science \& Business Media, 2009.
		
		\bibitem{Yoshida2021}
		N.~Yoshida, \emph{Simplified quasi-likelihood analysis for a locally
			asymptotically quadratic random field}, arXiv preprint arXiv:2102.12460,
		2021.
		
		\bibitem{Yuzhong2025}
		Yuzhong Cheng and Hiroki Masuda, \emph{Statistical inference for ergodic
			diffusion with {M}arkovian switching}, Discrete and Continuous Dynamical
		Systems - B \textbf{30} (2025), no.~10, 3910--3940.
		
		\bibitem{Yos11}
		Nakahiro Yoshida, \emph{Polynomial type large deviation inequalities and
			quasi-likelihood analysis for stochastic differential equations}, Ann. Inst.
		Statist. Math. \textbf{63} (2011), no.~3, 431--479.
		
		\bibitem{ZhuYinBaran2017}
		Chao Zhu, G.~Yin, and Nicholas~A. Baran, \emph{Feynman--kac formulas for
			regime-switching jump diffusions and their applications}, SIAM Journal on
		Control and Optimization \textbf{55} (2017), no.~2, 1045--1085.
		
	\end{thebibliography}
\end{document}